\documentclass[11pt,twoside]{amsart}

\usepackage{amsmath, amsthm, amssymb, mathrsfs}
\usepackage{array}

\usepackage{hyperref}

\usepackage[normalem]{ulem}
\usepackage{color}

\input xy
\xyoption{all}

\newcommand{\Hom}{\mathrm{Hom}}

\newcommand{\brokrarr}{\vphantom{\to}\mathrel{\smash{{-}{\rightarrow}}}}

\newcommand{\Aut}{\mathrm{Aut}}
\newcommand{\Alt}{\mathrm{A}}   
\newcommand{\GL}{\mbox{\boldmath$\rm GL$}}

\newcommand{\Gal}{\mathrm{Gal}}

\newcommand{\tr}{\mathrm{\rm tr}}

\newcommand{\Z}{\mathbb{Z}}
\newcommand{\Gm}{\mathbb{G}}
\newcommand{\rk}{\mathrm{rank}}
\newcommand{\PP}{\mathbb{P}}

\newcommand{\Q}{\mathbb{Q}}

\newcommand{\Spec}{\mathrm{Spec}}
\newcommand{\Pic}{\mathrm{Pic}}

\newcommand{\Conv}{\mathrm{Conv}}
\newcommand{\Cone}{\mathrm{Cone}}
\newcommand{\omch}{\omega^{\vee}}

\newcommand{\Stab}{\mathrm{Stab}}
\newcommand{\Div}{\mathrm{Div}}
\newcommand{\divT}{\mathrm{div}}
\newcommand{\Thatdual}{\widehat{T}^0}
\newcommand{\ThatLdual}{\widehat{T}^0_L}

\newcommand{\GmL}{\Gm_{m,L}}
\newcommand{\Gmk}{\Gm_{m,\Bbbk}}
\newcommand{\Gmks}{\Gm_{m,\Bbbk_s}}
\newcommand{\cO}{\mathcal{O}}
\newcommand{\SF}{\mathrm{SF}}
\newcommand{\kgp}{\Bbbk-\mathrm{grp}}
\newcommand{\mult}{\mathrm{mult}}
\newcommand{\Span}{\mathrm{Span}}
\newcommand{\MP}{\mathrm{MP}}
\newcommand{\dP}{\mathrm{dP}}
\newcommand{\coker}{\mathrm{coker}}
\newcommand{\Orb}{\mathrm{Orb}}
\newcommand{\aff}{\mathrm{aff}}
\newcommand{\SK}{\mathrm{SK}}
\newcommand{\lcm}{\mathrm{lcm}}
\newcommand{\sgn}{\mathrm{sgn}}
\newcommand{\wt}{\mathrm{wt}}
\newcommand{\Aff}{\mathrm{Aff}}
\newcommand{\bk}{\Bbbk}
  
\def\A{\mathbb{A}}

\def\cO{\mathcal{O}}

\def\Z{\mathbb{Z}}
\def\Q{\mathbb{Q}}

\def\N{\mathbb{N}}
\def\R{\mathbb{R}}

\def\e{\mathbf{e}}

\def\Ind{\mathrm{Ind}}
\def\Cl{\mathrm{Cl}}

\def\P{\mathbb{P}}
\def\x{\mathbf{x}}

\def\f{\mathbf{f}}
\def\bc{\mathbf{c}}
\def\d{\mathbf{d}}
\def\id{\mathrm{id}}

\def\0{\mathbf{0}}
\def\bu{\mathbf{u}}
\def\v{\mathbf{v}}

\def\w{\mathbf{w}}
\def\n{\mathbf{n}}
\def\m{\mathbf{m}}

\newtheorem{thm}{Theorem}[section]
\newtheorem*{thm*}{Theorem}
\newtheorem{prop}[thm]{Proposition}
\newtheorem*{prop*}{Proposition}
\newtheorem{lem}[thm]{Lemma}
\newtheorem{cor}[thm]{Corollary}
\newtheorem*{cor*}{Corollary}

\theoremstyle{definition}   
\newtheorem{defn}[thm]{Definition}
\newtheorem{rem}[thm]{Remark}

\newtheorem*{ex*}{Example}

\title{Demazure models of Algebraic tori in dimension 4}

\author{Nicole Lemire}

\begin{document}

\begin{abstract}
  For a field $\bk$, algebraic $\bk$-tori are algebraic $\bk$-groups which become 
isomorphic to a split $\bk$-torus, after base change to the separable closure of 
$\bk$.   Algebraic $\bk$-tori  are classified up 
to isomorphism by continuous integral representations of the absolute Galois group over $\bk$.  In fact, an algebraic $\bk$-torus can always be split by a finite Galois extension $L/\bk$ with Galois splitting group $G=\Gal(L/\Bbbk)$.
Work of Voskresenskii  used the  existence of a smooth projective model 
of an algebraic $\bk$-torus in order to determine important invariants  to determine its birational properties.

A Demazure model of an algebraic $\bk$-torus $T$  split by Galois extension $L/\bk$ and splitting group $G=\Gal(L/\Bbbk)$ is a smooth projective $T_L$-toric variety $X_{\Sigma}$ such that the natural action of $G$ on $T_L$ extends to $X_{\Sigma}$.
A smooth projective model of the algebaic $\bk$-torus $T$ can then be taken to be the quotient of a Demazure model by its splitting group. 
The existence of Demazure models of algebraic $\bk$-tori was determined by Brylinski and further refined by Colliot-Th\'el\`ene, Harari and Skorobogatov.
Voskresenskii and Kunyavskii’s birational classifications of algebraic $\bk$-tori in dimensions 2 and 3 involved constructions of Demazure models of the  algebraic $\bk$-tori corresponding to maximal finite subgroups of 
$\GL(r,\Z)$ for $r=2,3$. We discuss some explicit constructions of Demazure models of algebraic $\bk$-tori, making connections with the 
defining integral representations of their splitting groups.  The constructions determine smooth projective Demazure models of the algebraic tori 
corresponding to maximal finite subgroups of $\GL(4,\Z)$.  Since projective 
toric varieties are determined by lattice polytopes, our constructions focus on 
some highly symmetric families of  polytopes, such as root polytopes and central transportation polytopes.   \end{abstract}

\maketitle

\section{Introduction}

For a field $\bk$, an algebraic $\bk$-torus $T$ of dimension $r$ is a $\bk$-form of a split $\bk$-torus.  That is, $T$ is an algebraic $\bk$-group which is isomorphic to a split torus after base change to the separable closure of $\bk$. 
Algebraic $\bk$-tori play a significant role in the study of algebraic $\bk$-groups over non-closed fields. Algebraic $\bk$-tori are classified up to isomorphism by integral representations of the absolute Galois group over $\bk$.
An algebraic $\bk$-torus of dimension $r$ split by a finite Galois extension with Galois group $G$ is classified up to isomorphism by its character lattice, a rank $r$ $G$-lattice, or equivalently by the
conjugacy class of $G$ as a finite subgroup of $\GL_r(\Z)$. The affine ring of an algebraic torus $T$ split by a Galois extension $L/\bk$ with Galois group G is given by the multiplicative invariant ring over $L$ of its character lattice $\widehat{T_L}$ as a $G$-lattice.

Many birational invariants for algebraic $\bk$-tori are defined in terms of a smooth projective model for the algebraic $\bk$-torus, $T$: that is, a smooth projective variety $X$ containing $T$ as a dense open set.
In principal, an algebraic $\bk$-torus over a field $\bk$ always has a smooth toric projective model,~\cite{Bry79,CTHS05}
 but descriptions might not be so practical.

Given a algebraic $\bk$-torus $T$ with Galois splitting extension $L/\bk$, and corresponding Galois splitting group $G=\Gal(L/\Bbbk)$, and a smooth projective model $X$ of $T$, Voskresenskii showed that the $G$-module structure of the Picard group of $X_L$,  $\Pic(X_L)$, and its cohomological properties reflected birational properties of $T$.  In particular, he showed that $\Pic(X_L)$ is a $G$-flasque lattice: $G$-lattices whose Tate cohomological subgroups $\hat{H}^{-1}(H,\Pic(X_L))$ vanish for all subgroups $H$ of $G$.  

He showed that a $G$-flasque resolution of the character lattice of $T_L$, had $G$-flasque lattice $\Pic(X_L)$ and that the $G$-stable permutation class of $\Pic(X_L)$ was independent of the choice of smooth projective model $X$.

This set-up allowed Voskresenskii to show that such an algebraic $\bk$-torus with Galois splitting field $L/\bk$ and Galois splitting group $G=\Gal(L/\Bbbk)$ is stably $\bk$-rational (that is, $T\times \Gmk^s$ is rational
for some $s\ge 0$) if and only if the Picard group $\Pic(X_L)$ is stably permutation as a $G$-lattice.~\cite{Vos70,Vos74,Vos98}

Later, Saltman showed that such an algebraic $\bk$-torus $T$ is  retract rational~\cite{Sa84b}
(that is, if the identity map on $T$ factors rationally through some projective space) if and only if the Picard group $\Pic(X_L)$ is  a direct summand of a $G$-permutation lattice.

Voskresenskii~\cite{Vos83}  produced a method to construct a smooth projective model of an algebraic $\bk$-torus $T$ with Galois splitting field $L/\bk$ and Galois splitting group $G=\Gal(L/\Bbbk)$ as a quotient of a smooth projective toric variety, harnessing  Demazure's seminal work on smooth toric varieties~\cite{Dem70}.  The construction determines a smooth complete fan in the vector space dual to the character lattice: $\widehat{T}^0_L\otimes \Q$ such that $G$ permutes the action of the cones of the fan.  Then $T_L$ is embedded
$G$-equivariantly as a dense open subset of $X_{\Sigma}$ and a smooth projective model $X$ of $T$ can be taken to be a quotient of $X_{\Sigma}$ under the action of $G$.
Then since $X_L=X_{\Sigma}$, the flasque resolution of $X_{\Sigma}$ can be determined as its divisor class group sequence:
$$0\to \widehat{T_L}\stackrel{\mathrm{div}}{\to} \Div_{T_L}(X_{\Sigma})\to \Pic(X_{\Sigma})\to 0$$
A smooth projective toric $T_L-$variety $X_{\Sigma}$ such that the natural action of $G$ on $T_L$ extends to $X_{\Sigma}$ is called a Demazure model of $T_L$.

An algorithm to construct such a Demazure model for an algebraic $\bk$-torus was determined by Brylinski~\cite{Bry79}, and then refined by Colliot-Th\'el{\`e}ne, Harari and Skorobogatov~\cite{CTHS05} due to some technical issues.
Voskresenskii proved that all algebraic $\bk$-tori of dimension 2 were rational by analysis of the Demazure models of the 2 algebraic tori corresponding to maximal subgroups of $\GL(2,\Z)$~\cite{Vos67}.

Kunyavskii's birational classification of algebraic $\bk$-tori in dimension 3 hinged on constructions of the Demazure models of the 4 algebraic tori corresponding to maximal finite subgroups of $\GL_3(\Z)$~\cite{Kun87}.

Not long after Voskresenskii's original construction of a flasque resolution for an algebraic torus in terms of the Picard group of a smooth projective model,
Colliot-Th\'el{\`e}ne and Sansuc~\cite{CTS77} determined a representation-theoretic technique of producing a flasque resolution for an algebraic $\bk$-torus in terms of its flasque resolution, as background for their determination of the R-equivalence classes of an algebraic torus.    Flasque and coflasque resolutions of algebraic $\bk$-tori (and their generalization to algebraic $\bk$-groups~\cite{CT08}) have had many subsequent important applications. In particular, Hoshi and Yamasaki~\cite{HY17}, used  the representation-theoretic  techniques of Colliot-Th\'el{\`e}ne and Sansuc, together with clever GAP algorithms, to determine the stably rational and retract rational algebraic $\bk$-tori of dimension 4 and 5.

In this paper, we  determine a Demazure model and resulting flasque resolution for the 9 algebraic tori corresponding to maximal finite subgroups of $\GL_4(\Z)$.
The  toric variety of a root system is a smooth projective  toric variety whose fan has maximal cones corresponding to the Weyl chambers of a root system.  That the
toric variety of a root system can be used as a Demazure model for an algebraic $\bk$-torus corresponding to the root lattice with the action of the automorphism group of the root system is mentioned in the literature, eg.~\cite{KV84}.

We prove that all but 2 of the algebraic tori corresponding to the maximal finite subgroups of $\GL_4(\Z)$ have Demazure models which can be constructed as toric varieties of a root system, their products or products of previously known Demazure models.

We determine equivariant projective simplicial models for another natural infinite family of algebraic tori, which correspond to maximal finite subgroups of $\GL(n,\Z)$: those with character lattice determined by the weight lattice of $A_n$, $\Lambda(A_n)$, equipped with the action of the automorphism group of $A_n$: $\Aut(A_n)$. For subgroups $G$ of $W(A_n)$, the algebraic torus corresponding to $(\Lambda(A_n),G)$ is a norm-1 torus.

A projective toric $T_L$-model for an algebraic $\bk$-torus $T$ with character lattice over a splitting field $L$ given by $(\widehat{T_L},\Gal(L/\Bbbk))=(\Lambda(A_n),\Aut(A_n))$ can be determined as the projective toric variety $X_{\Sigma}$ corresponding to  the fan $\Sigma$ which consists of cones over faces of the root polytope $P(A_n)$, the convex hull of the roots of $A_n$.  Although $X_{\Sigma}$ is not smooth, much is known about the structure of the root polytope $P(A_n)$.  Work of Vinberg~\cite{Vin90}; Cellini, Marietti~\cite{CM14,CM15}; and Ardila~\cite{ABHPS11}, among others determined the structure of this polytope, such as the structure of its faces, and (non-equivariant) triangulations, in terms of representation-theoretic connections to the extended Weyl group of $A_n$.  We constructed a simplicial equivariant projective toric model for the algebraic tori corresponding to $(\Lambda(A_n),\Aut(A_n))$.  Note that $(\Lambda(A_n),\Aut(A_n))$ corresponds to a maximal subgroup of $\GL(n,\Z)$.  In the cases $n=3$ and $n=4$, we construct a smooth Demazure model.  
The construction is explicit: the vertices, and simplices have an interesting combinatorial construction.  The construction is related to the unimodular triangulation of the root polytope for $A_n$~\cite{ABHPS11,CM15}.
The construction for $n=3$ is related to  Kunyavskii's Demazure model in that case.

Klyachko~\cite{Kly83} determined Demazure models $X_{mn}$ for a special class of algebraic tori corresponding to the tensor product of two root lattices of type $A$:
$(\Z A_{m-1}\otimes \Z A_{n-1},S_m\times S_n)$, assuming that $m$ and $n$ are relatively prime.  The toric varieties $X_{mn}$ are constructed as the toric variety of the central transportation polytopes $\Gamma_{mn}$.  These models were further studied in joint work with Voskresenskii~\cite{KV84}.
Among other things, they showed that the  algebraic tori corresponding to 
$(\Z A_{m-1}\otimes \Z A_{n-1},S_m\times S_n)$, for relatively prime $m,n$ are stably rational, and that if $n\equiv 1\bmod m$, are  in fact, rational.
Klyachko's varieties provide alternative Demazure models for the algebraic tori corresponding to $(\Z A_{2r},\Aut(A_{2r}))=(\Z A_{2r}\otimes \Z A_1, S_{2r+1}\times S_2)$.

It turns out that the remaining maximal subgroup of $\GL(4,\Z)$ corresponds to a lattice $(\Z A_2\otimes \Z A_2,(S_3\times S_3)\rtimes C_2\times C_2)$ which extends the action on $(\Z A_2)^{\otimes 2}$ from Klyachko's case.  We construct a Demazure model for the corresponding algebraic torus. The construction has some relation with the underlying Klyachko variety $X_{33}$. We remark that Jason Palombaro, in his thesis,~\cite{Pal25}, determined an independent argument for the same construction as part of ongoing joint work~\cite{LP}.

For each Demazure model of an algebraic torus that we discuss, we determine the associated flasque resolution.  In the final section, we compare the flasque resolutions obtained from representation-theoretic techniques initiated by Colliot-Th\'el\`ene and Sansuc to those determined from the divisor class sequence of a Demazure model. The flasque resolution determined from the Demazure model of an algebraic torus often beautifully reflects the geometry and the representation theory of the character lattice as a module for its splitting group. However, it is sometimes true that the resulting flasque resolutions are not at all minimal.  We show that representation-theoretic arguments can produce smaller flasque resolutions. 

The interest in these explicit Demazure models stems from connections to some outstanding cases of stably rational algebraic tori in dimension 4, whose rationality is unknown.  So far, most proofs of rationality for algebraic tori have used the construction of the Demazure model.  In particular, for $\Alt_5$, the alternating group on 5 letters, the algebraic tori with character lattices $(\Lambda(A_4),H)$ where $H=\Alt_5$ or $\Alt_5\times C_2$,  have been shown to be stably rational, but their rationality is unknown~\cite{HY17,Lem17}.  A further study of the cohomological invariants, Chow groups,  and K-theory of these varieties would be feasible with these explicit constructions and may shed more light on these  questions.  In addition, since flasque resolutions of algebraic tori are an important tool, these explicit constructions should have further applications.

This paper is organised as follows. The preliminary material sets up notation and reviews definitions and basic results.  It is set up to be self-contained, but so that an expert can select the material to skip. In Section 2, notation, definitions and basic results are reviewed about polytopes  and toric varieties for split $\bk$-tori.  In Section 3, algebraic $\bk$-tori, their toric models, their Demazure models, and their flasque resolutions are discussed.   In Section 4, crystallographic root systems and their associated affine root systems are discussed, to set the scene for the discussion of root polytopes and the toric variety of a root system.  In Section 5, we describe the Dade groups, maximal subgroups of $\GL(k,\Z)$ for $k=2,3,4$ in terms of root system data.
In Section 6, we explain how the toric variety of a root system can be used to determine Demazure models for certain algebraic tori, and determine the associated flasque resolution.  We determine  Demazure models for all but 2 of the algebraic tori corresponding to Dade groups in $\GL(4,\Z)$.  In Section 7, we construct a simplicial $\Aut(A_n)$-invariant projective toric model for the algebraic torus corresponding to the weight lattice of $A_n$ with the action of the automorphism group of $A_n$. We relate the corresponding fan to the unimodular triangulation of the root polytope of $A_n$. Then we construct a Demazure model for the algebraic torus corresponding to $(\Lambda(A_4),\Aut(A_4))$.
In Section 8, we compare our constructions with the Demazure models determined by Kunyavskii.
In Section 9, we discuss the Demazure model corresponding to the toric variety of the $m\times n$ central transportation polytope, following work of Klyachko and Voskresenskii.  We discuss the related construction of a Demazure model for the remaining Dade group of $\GL(4,\Z)$.
In Section 10, we compare the construction of  flasque resolutions of algebraic $\bk$-tori using the divisor class sequence of a Demazure model versus the representation-theoretic technique initiated by Colliot-Th\'el{\`e}ne and Sansuc.

\section*{Acknowledgements}
The majority of the research for this paper was undertaken at the Max Planck Institute for Mathematics in Bonn.  The author wishes to thank MPIM Bonn for the wonderful research conditions there, the hospitality and support.  The last sections of the paper, including the last Demazure model, and much of the editing, were completed at Laboratoire de Combinatoire et d'informatique, at Universit\'e de Qu\'ebec \`a Montr\'eal.  The author would like to thank the members of LaCiM for their warm hospitality and support.  

\section{Preliminaries}

In this section, we fix notation and recall the known definitions and results that we use most frequently in the paper.

\subsection{Basic Notation}

Let $\bk$ be a field.
A $\bk$-variety is a separated integral scheme of finite type over $\bk$.
An algebraic $\bk$-group is a smooth affine group variety of finite type over $\bk$.
For integers $a\le b$, the interval $[a,b]$ will denote the set $\{n\in \Z:a\le n\le b\}$.  For $n\in \N$, the interval $[1,n]$ will be denoted by $[n]$.

For reference, see \cite[I,Section 1]{KKMSD73}, \cite[Section 1.1]{CLS11}.

A $\Z$-lattice M is a free abelian group of finite rank. The group algebra of $M$ over $\bk$ will be denoted as $\bk[M]$.  A semigroup $S\subset M$ using the addition of $M$ defines a semigroup algebra $\bk[S]$ as a $\bk$-subalgebra of $\bk[M]$.  $\bk[S]$ is a finitely $\bk$-algebra of $\bk[M]$ if $S$ is finitely generated as a semigroup.  The quotient field of $\bk[M]$ is $\bk(M)$.  The quotient field of $\bk[S]$ is $\bk(M)$ if $S$ generates $M$ as a group.  

A finitely generated semigroup $S\subseteq M$ which generates $M$ as a group is \emph{saturated} in $M$ if for $0\ne r\in \N$ and $\m\in M$, $rm\in S$ implies that $\m\in S$.
In this case, $\bk[S]$ is integrally closed in $\bk(S)=\bk(M)$ and so is a normal $\bk$-algebra.

\subsection{Rational Polyhedral Cones}

For further information, see, for example,~\cite[1.2]{CLS11}.

We will fix a non-degenerate bilinear pairing of $\Z$-lattices $M$ and $N$:
$$M\times N\to \Z, (m,n)\to \langle m,n\rangle$$

This determines a natural duality isomorphism $M^0=\Hom_{\Z}(M,\Z)\cong N$.

A subset $\sigma$ of $N_{\R}$ is a \emph{rational polyhedral cone} if $\sigma=\Cone(\x_1,\dots,\x_s)=\sum_{i=1}^s\R^+\x_i$ for some vectors $\x_1,\dots,\x_s\in N_{\Q}$.
If $\sigma$ is a rational polyhedral cone in $N_{\R}$ then 
$$\sigma^{\vee}=\{\m\in M_{\R}:\langle \m,\bu\rangle\ge 0, \bu\in \sigma\}$$
is a rational polyhedral cone in $M_{\R}$, called the \emph{dual cone} of $\sigma$. Note that $\sigma^{\vee\vee}=\sigma$.

A rational polyhedral cone $\sigma$ in $N_{\R}$ defines a  subspace $\sigma^{\perp}$ of $M_{\R}$
$$\sigma^{\perp}=\{\m\in M_{\R}:\langle \m,\bu\rangle=0, \bu\in \sigma\}.$$
The duality pairing between $M$ and $N$ induces a duality pairing between the lattices $\sigma^{\perp}\cap M$ and $N_{\sigma}=N/\langle N\cap \sigma\rangle$.

A non-zero vector $m\in M_{\Q}$ determines a hyperplane in $N_{\R}$ 
$$H_{\m}=\{\bu\in N_{\R}: \langle \m,\bu\rangle=0\}$$
and a positive half-space
$$H_{\m}^+=\{\bu\in N_{\R}: \langle \m,\bu\rangle \ge 0\}$$
$H_{\m}$ is a \emph{supporting hyperplane} of $\sigma\in N_{\R}$ if $\sigma\subseteq H_{\m}^+$.

A rational polyhedral cone can be alternatively expressed as a finite intersection of its supporting positive half-spaces.  In particular, if the dual cone 
 $\sigma^{\vee}=\Cone(\m_1,\dots,\m_s)$ for some $\m_i\in M_{\Q}$,
then $\sigma=\cap_{i=1}^sH_{\m_i}^+$.

A \emph{face} $\tau$ of a rational polyhedral cone $\sigma$ is $\tau=H_{\m}\cap \sigma$ for some $\m\in \sigma^{\vee}\cap M_{\Q}$.  $\tau$ is again a rational polyhedral cone.

A rational polyhedral cone $\sigma$ is \emph{strongly convex} if $\{0\}$ is a face of $\sigma$.

\subsection{Polytopes}
For further information, see, for example,~\cite[Section 2.2]{CLS11}.

Fix a non-degenerate bilinear pairing of $\Z$-lattices $M$ and $N$:
$$M\times N\to \Z, (\m,\n)\to \langle \m,\n\rangle$$

    A \emph{rational polytope} $P\subseteq M_{\R}$ is the convex hull of a finite set $S\subseteq M_{\Q}$.  That is,
    $$P=\Conv(S)=\{\sum_{\v\in S}\lambda_{\v}\v: \lambda_{\v}\in \R_+, \sum_{\v\in S}\lambda_{\v}=1\}$$

    The dimension of a rational polytope $P\subseteq M_{\R}$ is the dimension of the subspace of $M_{\R}$ spanned by $\{\v-\v':\v\in P\}$ for some $\v'\in P$.
    A rational polytope $P=\Cone(S)\subseteq M_{\R}$ is a \emph{lattice polytope} if $S\subseteq M$.
If $P=\Conv(S)$ is a lattice polytope, then so is $rP=\Conv(rS)$ for any integer $r\ge 2$.

    A nonzero vector $\bu\in N_{\Q}$ and $b\in \Q$ determine an affine hyperplane 
    $$H_{\bu,b}=\{\m\in M_{\R}:\langle \m,\bu\rangle =b\}$$
    and closed half-spaces:
    $$H_{\bu,b}^+=\{\m\in M_{\R}:\langle \m,\bu\rangle \ge b\}, H_{\bu,b}^-=\{\m\in M_{\R}: \langle \m,\bu\rangle \le b\}$$

A subset $Q\subseteq P$ is a face of $P$, $Q\preceq P$
if there exist $0\ne \bu\in N_{\Q}, b\in \Q$ with
$$Q=H_{\bu,b}\cap P, P\subseteq H^+_{\bu,b}$$
Then $H_{\bu,b}$ is a supporting affine hyperplane in this case.
Note that $Q$ is also a polytope, with $Q=\Conv(S\cap H_{\bu,b})$ if $P=\Conv(S)$.

Important faces are those of dimensions 0,1,$\dim(P)-1$, which are called respectively: vertices, edges and facets.

Rational polytopes can alternatively be described as a bounded finite intersection of closed half-spaces.  

A full-dimensional lattice polytope $P\subseteq M_{\R}$ has a canonical such description. Each facet $F$ of $P$  has a unique supporting affine hyperplane of the form $H_{\bu_F,-a_{F}}$ where $\bu_F\in N, a_F\in \Z$ and so
$$P=\bigcap_{F\mbox{\tiny{ facet of }} P}H_{\bu_F,-a_F}^+$$

A lattice polytope $P\subseteq M_{\R}$ is \emph{normal} if the lattice points of $P$ generate those of $rP$ for all $r\ge 1$, i.e. $rP\cap M=r(P\cap M), r\ge 1$.
An important result about a full-dimensional lattice polytope $P$ of dimension at least 2 is that one can always find a positive integer multiple which is normal.  More precisely, for all $r\ge \dim(P)-1$, $rP$ is a normal polytope. (See ~\cite[2.2.12]{CLS11}.)

We will often assume that our lattice polytope $P$ in $M_{\R}$ is full-dimensional and contains $\0$ in its interior.
In this case, we may define a dual polytope to $P$ in $N_{\R}$.
The polar dual polytope to $P$ is then
$$P^0=\{\bu\in N_{\R}: \langle \bu,\v\rangle \ge -1, \v\in P\}$$
Note that $P=(P^0)^0$ and that $P^0=\Conv\{\frac{\bu_F}{a_F}: F\mbox{ is a facet of }P\}$.
The $r$-dimensional faces of $P$ correspond bijectively to the $(n-r-1)$-dimensional faces of $P^0$ for $0\le r\le n-1$.

Suppose additionally that our full-dimensional lattice polytope $P$ in $M_{\R}$ satisfies $-P=P$.  Then $F$ is a facet of $P$ with supporting affine hyperplane $H_F=H_{\bu_F,-a_F}$ if and only if $-F$ is a face of $P$ with supporting affine hyperplane $H_{-F}=H_{-\bu_F,-a_F}$.
The polar dual polytope to $P$ can alternatively be expressed as
$$P^0=\{\bu\in N_{\R}: \langle \bu,\v\rangle \le 1, \v\in P\}$$

\subsection{Toric Varieties for a split torus}

Let $\bk$ be a field. There is a natural bijection between split $\bk$-tori $T\cong \Gmk^r$ of dimension $r$ and $\Z$-lattices of rank $r$.  Given a split $\bk$-torus $T$ of rank $r$, its character lattice $M=\widehat{T}=\Hom_{\kgp}(T,\Gmk)=\{\chi_m:m\in M\}$
is a rank $r$ $\Z$-lattice.  Conversely, given a rank $r$ $\Z$-lattice $M$, $T=\Spec(\bk[M])$ is a split $\bk$-torus, where $\bk[M]$ indicates the $\bk$-group algebra of $M$.

Let $T$ be an algebraic $\bk$-torus, not necessarily split.

A \emph{toric $T$-variety} is a normal $\bk$-variety $X$ on which $T$ acts faithfully, together with a dense open $T$-orbit $U$.

A \emph{toric $T$-model} is a normal $\bk$-variety $X$ on which $T$ acts, together with an dense open $T$-orbit $U$ which is isomorphic to $T$.

Note that a toric $T$-model $X$ is equipped with a $T$-equivariant open embedding $T\cong U\subseteq X$ which determines the $T$-action on $X$.
Since the action of $T$ on itself is faithful, so is the action of $T$ on $X$.
$X$ is geometrically irreducible as a $\bk$-variety since $U_{\bk_s}\cong \Gmks^r$.

In Section~\ref{section:toricmodels}, we will discuss these definitions for an arbitrary algebraic $\bk$-torus, and explain how they fit in with other definitions in the literature.
For now, note that if $T$ is split, note that the definitions of toric $T$-varieties and toric $T$-models are the same.  In particular, if $\bk$ is algebraically closed, all algebraic $\bk$-tori are split. 
This essentially means that for the theory of toric $T$-varieties where $T$ is a split $\bk$-torus, one can follow the results of the (extensive) literature for toric varieties over an algebraically closed field.
We will make references to \cite{KKMSD73} and \cite{CLS11}.

We recall the characterization of normal affine toric $T$-varieties for a split $\bk$-torus $T$.

\begin{thm}~\cite[I,\S 1,Theorem 1]{KKMSD73} Let $M$ be a $\Z$-lattice of finite rank and let $T=\Spec(\bk[M])$.
  The correspondence $S\to \Spec(\bk[S])$ defines a bijection between the set of finitely generated semigroups $S\subseteq M$ which are saturated in $M$ and which generate $M$ as a group and the set of isomorphism classes of normal affine toric $T$-varieties.
\end{thm}

For a split $\bk$-torus $T$, the geometric properties of toric $T$-varieties
are more easily described in terms of the cocharacter lattice $N=\Thatdual$ of the split $\bk$-torus: more precisely in terms of rational polyhedral cones in $N_{\R}=\Thatdual_{\R}$.

Let $T$ be a split $\bk$-torus $T$ with character lattice $M=\hat{T}=\Hom(T,\Gmk)=\{\chi^{\m}:\m\in M\}$ and cocharacter lattice $N=\Hom(\Gmk,T)=\{\lambda_{\n}:\n\in N\}$. 
Then there is a non-degenerate bilinear form  $M\times N\to \Z,(\m,\n)\to \langle \m,\n\rangle$ such that $\chi^{\m}\circ \lambda^{\n}(t)=t^{\langle \m,\n\rangle}$.
Conversely, given a $\Z$-lattice $M$ and its $\Z$-dual $N=M^{0}=\Hom(M,\Z)$, the pairing $M\times N\to \Z, (\m,f)\to f(\m)$ determines a non-degenerate bilinear form.  Then the split $\bk$-torus $T=\Spec(\bk[M])$ has character lattice $M$ and cocharacter lattice $N$.

 There is a bijection between the set of strongly convex rational polyhedral cones $\sigma\in N_{\R}$ and the set of finitely generated semigroups $S\subseteq M$ which are saturated in $M$ and generate $M$ as a group. This bijection  is 
determined by the map $\sigma\to \sigma^{\vee}\cap M$. It provides an update to the last theorem.

\begin{thm}~\cite[I,\S 1,Theorem 1']{KKMSD73}
  Let $T$ be a split $\bk$-torus.
  
  The correspondence $\sigma\to \Spec (\bk[\sigma^{\vee}\cap \widehat{T}])\equiv U_{\sigma}$ defines a bijection between the set of strongly convex rational polyhedral cones $\sigma\subseteq \Thatdual_{R}$ and the set of isomorphism classes of normal affine toric $T$-varieties.
\end{thm}

Geometric properties of the normal affine toric $T$-varieties $U_{\sigma}$ can be described in terms of the strongly convex rational polyhedral cones $\sigma\subseteq N_{\R}$ where $N=\Thatdual$.

A strongly convex rational polyhedral cone $\sigma\subseteq N_{\R}$ is called \emph{simplicial} if $\sigma\cap N$ can be generated by a linearly independent subset of $N$.
It is called \emph{smooth} if $\sigma\cap N$ can be generated by a subset of a $\Z$-basis of $N$.

\begin{prop}\cite[1.3.12,1.3.20]{CLS11},\cite[I,\S 1,Theorem 4;p.19]{KKMSD73}
  
  For a strongly convex rational polyhedral cone $\sigma\subseteq N_{\R}$,   $U_{\sigma}$ is an orbifold
  if and only if $\sigma$ is simplicial. $U_{\sigma}$ is smooth if and only if $\sigma$ is smooth.
\end{prop}

\begin{prop}~\cite[I,\S 1, Theorem 3]{KKMSD73},\cite[p.106-7]{CLS11}
  Let $\sigma$ be a strongly convex rational polyhedral cone and $\tau=\sigma\cap H_{\m}$ a face corresponding to $\m\in \sigma^{\vee}\cap M$, then $$U_{\tau}=\Spec(\bk[\sigma^{\vee}\cap M]_{\chi^{\m}})=(U_{\sigma})_{\chi^{\m}}$$
  is a basic open subset of $U_{\sigma}$ determined by the character $\chi^{\m}$.
  This determines a unique $T$-equivariant embedding of $U_{\tau}$ into $U_{\sigma}$.

In particular, if $\sigma_1,\sigma_2$ are two strongly convex rational polyhedral cones in $N_{\R}$, then for $\tau=\sigma_1\cap \sigma_2$, there exists a common supporting hyperplane, that is, there exists $\m\in (\sigma_1^{\vee}\cap (-\sigma_2)^{\vee})\cap M$ such that
$\sigma_1\cap H_{\m}=\tau=\sigma_2\cap H_{\m}$.  This produces a gluing isomorphism
$g_{\sigma_2,\sigma_1}:(U_{\sigma_1})_{\chi^{\m}}\cong (U_{\sigma_2})_{\chi^{-\m}}$ which is the identity on $U_{\tau}$.
\end{prop}

To determine normal toric varieties, one needs the definition of a fan.
\begin{defn}~\cite[3.1.2]{CLS11}
  
  A fan $\Sigma$ in $N_{\R}$ is a finite collection of cones $\sigma\subseteq N_{\R}$ such that:
  \begin{itemize}
  \item Every $\sigma\in \Sigma$ is a strongly convex rational polyhedral cone.
  \item For all $\sigma\in \Sigma$, each face of $\sigma$ is in $\Sigma$.
  \item For all $\sigma_1,\sigma_2\in \Sigma$, the intersection $\sigma_1\cap \sigma_2$ is a face of each and so is in $\Sigma$.
\end{itemize}
\end{defn}

    The toric variety of a fan $\Sigma$ is then the algebraic $\bk$-variety defined by the quotient map $X_{\Sigma}=\sqcup_{\sigma\in \Sigma}U_{\sigma}/\sim$ where $u_1\sim u_2$ if and only if $u_1\in U_{\sigma_1}, u_2\in U_{\sigma_2}$, for some $\sigma_1,\sigma_2\in \Sigma$ with $u_2=g_{\sigma_2,\sigma_1}(u_1)$.
 That is, $X_{\Sigma}$ is determined by the normal affine toric varieties corresponding to the cones of the fan, glued according to the intersection data of the fan.

Properties of $\Sigma$ determine geometric properties of $X_{\Sigma}$.

\begin{thm}~\cite[3.1.19,11.4.8]{CLS11}

  Let $\Sigma\subseteq N_{\R}$ be a fan.
      \begin{itemize}
      \item $X_{\Sigma}$ is smooth if every cone $\sigma$ in $\Sigma$ is smooth.
      \item $X_{\Sigma}$ is an orbifold if every cone $\sigma$ in $\Sigma$ is simplicial.
      \item $X_{\Sigma}$ is compact in the classical topology if $\Sigma$ is complete, i.e. if the support of $\Sigma$, $|\Sigma|=\cup_{\sigma\in \Sigma}\sigma$ is $N_{\R}$.
\end{itemize}
      \end{thm}

\begin{thm}~\cite[3.2.6,3.2.7]{CLS11},\cite[I,\S 2,Theorem 6]{KKMSD73}
  
Let $T$ be a split $\bk$-torus.
      \begin{enumerate}
        \item The correspondence $\Sigma\to X_{\Sigma}$ defines a bijection between the fans of $\Thatdual_{\R}$ and the isomorphism classes of normal toric $T$-varieties. 
        \item The map $\sigma\to U_{\sigma}$ defines a bijection between the cones of a fan  $\Sigma$ on $\Thatdual_{\R}$ and the set of $T$-invariant open affine subsets of $X_{\Sigma}$.
          
        \item The  map which assigns to each $U_{\sigma}$ the unique $T$-orbit $\cO(\sigma)\cong \Spec(\bk[\sigma^{\perp}\cap \widehat{T}])$ which is closed in $U_{\sigma}$, defines a bijection between the cones of the fan $\Sigma$ and its $T$-orbits.
 \item Let $\tau,\sigma\in \Sigma$. $\tau$ is a face of $\sigma$ if and only if $\cO(\sigma)\subseteq \overline{\cO(\tau)}$.
\item Let $\sigma\in \Sigma$. $V(\sigma)=\overline{\cO(\sigma)}$ is a normal toric $\Spec(\bk[\sigma^{\perp}\cap \widehat{T}])$-variety of dimension $\rk(\widehat{T})-\dim(\sigma)$. 
      \end{enumerate}

    \end{thm}

    From the last part of the theorem, we note that $\Sigma(r)$, the set of $r$-dimensional cones of $\Sigma$ are in bijection with the $T$-invariant closed subvarieties of $X_{\Sigma}$ of codimension $r$.
    In particular, the $T$-invariant divisors of $X_{\Sigma}$ are precisely determined by the rays of the fan $\Sigma$.  Given a ray $\rho\in \Sigma(1)$, $D_{\rho}=\overline{\cO(\rho)}$ is the corresponding $T$-invariant divisor.
   Then $\Div_{T}(X_{\Sigma})=\oplus_{\rho\in \Sigma(1)}\Z D_{\rho}$ is the group of $T$-invariant Weil divisors on $X_{\Sigma}$.  For each ray $\rho\in \Sigma(1)$, the semigroup $\rho\cap N$ has a unique minimal generator $\bu_{\rho}$. A $T$-invariant principal divisor on $X_{\Sigma}$ must take the form
    $$\divT(\chi^{\m})=\sum_{\rho\in \Sigma(1)}\langle \m,\bu_{\rho}\rangle D_{\rho}$$ for some $\m\in M$.

   The divisor class group $\Cl(X_{\Sigma})$ of the  toric variety corresponding to a complete fan $\Sigma$, $X_{\Sigma}$,  can be determined in terms of the $T$-invariant divisors:

   \begin{prop}~\cite[4.1.2,4.1.3]{CLS11}
     
      For a complete fan $\Sigma$ on $N_{\R}$, there is a short exact sequence
      $$0\to M\stackrel{\mathrm{div}}{\to} \Div_{T_N}(X_{\Sigma})\to \Cl(X_{\Sigma})\to 0$$
      where the map $\divT: M\to \Div_{T}(X_{\Sigma}), \m\to \sum_{\rho\in \Sigma(1)}\langle \m,\bu_{\rho}\rangle D_{\rho}$.
    \end{prop}

 Note that if our complete fan $\Sigma$ is additionally smooth then our variety $X_{\Sigma}$ is smooth and the proposition determines the Picard group of $X_{\Sigma}$ since  the Picard group and the divisor class group coincide for smooth varieties.  That is, if $\Sigma$ is also smooth, $\Pic(X_{\Sigma})=\Cl(X_{\Sigma})$.

 For the normal toric variety $X=X_{\Sigma}$ determined by a fan $\Sigma$ and a $T$-invariant Weil divisor $D$ of $X$,  one can describe the global sections of the  coherent $\cO_X$-sheaf $\cO_X(D)$ determined as:
 
 $$\cO_{X}(D)(U)=\{f\in k(X)^*: (\divT(f)+D)\vert_U\ge 0\}\cup \{0\}, U \mbox{ open in } X$$
 in terms of the lattice points of a polyhedron determined by $D$.
 
 \begin{prop}~\cite[4.3.2,4.3.3,4.3.8]{CLS11}

  Let $D=\oplus_{\rho\in \Sigma(1)}a_{\rho}D_{\rho}$ be a $T$-invariant Weil divisor on $X_{\Sigma}$ for a fan $\Sigma$. The polyhedron $P_D$ in $M_{\R}$ corresponding to $D$ is
$$P_D=\bigcap_{\rho\in \Sigma(1)}H_{\bu_{\rho},-a_{\rho}}.$$
Then the sheaf of the divisor $D$, $\cO_{X_{\Sigma}}(D)$, has global sections 
   $$\Gamma(X_{\Sigma},\cO_{X_{\Sigma}}(D))=\bigoplus_{\m\in P_D\cap M}k\chi^{\m}$$
determined by the lattice points of the polyhedron $P_D$ in $M_{\R}$. 

   When $\Sigma$ is complete, $P_D$ is a polytope in $M_{\R}$ and $\Gamma(X_{\Sigma},\cO_{X_{\Sigma}}(D))$ is a finite dimensional $\bk$-vector space determined by the lattice points of $P_D$.
\end{prop}

    If $\Sigma$ is a complete fan in $N_{\R}$, a $T$-invariant Weil divisor $D$ is Cartier if and only if $D$ is principal on each $U_{\sigma}$ for the full-dimensional cones $\sigma\in \Sigma(n)$.

    So given a $T$-invariant Cartier divisor $D=\sum_{\rho \in \Sigma(1)}a_{\rho}D_{\rho}$ on $X_{\Sigma}$ for a complete fan $\Sigma$, 
    there exist $\m_{\sigma}\in M$ for each $\sigma\in \Sigma$, such that
    $$D_{U_{\sigma}}=\divT(\chi^{-\m_{\sigma}})$$ or equivalently such that $\langle \m_{\sigma},\bu_{\rho}\rangle \ge -a_{\rho}$ for each $\rho\in \Sigma(1)$. 
    We say that the \emph{Cartier data} of $D$ is $\{\m_{\sigma}:\sigma\in \Sigma\}$.
    $D$ determines a well-defined map $\varphi_D:N_{\R}\to \R$, called a \emph{support function}, which is linear on each cone $\sigma\in \Sigma$ and integer-valued on $N$:
    $(\varphi_D)(\bu)=\langle \m_{\sigma},\bu\rangle, u\in \sigma$.

The set of such support functions for a complete fan $\Sigma$ is called $\SF(\Sigma,N)$
and is in bijection with the set of $T$-invariant Cartier divisors on $\Sigma$ under the map $D\to \varphi_D$. (See ~\cite[4.2.12]{CLS11}.)

If $D$ is an ample Cartier divisor on a normal variety $X$,
there exists some $r\in \N$ such that $rD$ is very ample.
This means that $\cO_X(rD)$ is generated by global sections and that it determines a closed embedding into some projective space.

One can phrase the condition for a $T$-invariant divisor on $X_{\Sigma}$ to determine a closed embedding into projective space using support functions.

\begin{thm}~\cite[6.1.14]{CLS11}
Let $\Sigma$ be a complete fan in $N_{\R}$.
A $T$-invariant Cartier divisor $D$ with Cartier data on $X_{\Sigma}$ with Cartier data $\{m_{\sigma}:\Sigma\}$ is ample if and only if its support function $\varphi_D:N_{\R}\to \R$ is strictly convex: that is, it is a convex function
such that
$$\varphi_D(\bu)=\langle \m_{\sigma},\bu\rangle \mbox{ if and only if }\bu\in \sigma$$
\end{thm}

Note that there are many equivalent conditions to check whether a support function $\varphi_D$ is strictly convex. See, for example,~\cite[Lemma 6.1.13]{CLS11},
\cite[Prop 1]{CTHS05}.

For our complete fan $\Sigma$, and an ample Cartier divisor $D$ such that the multiple $rD$ is very ample, 
the global sections of the line bundle associated to $\cO_{X_{\Sigma}}(rD)$ are generated by
$\{\chi^{\m_i}:i=1,\dots, s\}$ where $rP_{D}\cap M=\{\m_1,\dots,\m_s\}$ are the lattice points of the polytope of the divisor $rD$.
Then the morphism determined by the line bundle $\cO_{X_{\Sigma}}(rD)$
$$\phi_{rD}:X_{\Sigma}\to \PP^{s-1},p\to (\chi^{\m_1}(p),\dots,\chi^{\m_s}(p))$$
is a closed embedding. 

A full-dimensional lattice polytope $P\subseteq M_{\R}$ determines a complete fan called the \emph{normal fan} of $P$.  A facet has a unique supporting affine hyperplane of the form $H_{\bu_F,-a_{F}}$ where $\bu_F\in N, a_F\in \Z$ so that
$$P=\bigcap_{F\mbox{\tiny{ facet of }} P}H_{\bu_F,-a_F}^+$$

For each face $Q\preceq P$, $\sigma_Q=\Cone\{\bu_F: Q\preceq F\}$ is a strictly convex rational polyhedral cone in $N_{\R}$ of dimension $\dim(N_{\R})-\dim(Q)$.

Then $\Sigma_P=\{\sigma_Q: Q\preceq P\}$ is a complete fan in $N_{\R}$.
In the special case in which the full-dimensional lattice polytope $P\subseteq M_{\R}$ contains $0$ in its interior, then the normal fan of $P$ consists of cones over faces of the rational (possibly integral) polytope $P^0$.

A full-dimensional lattice polytope $P$ also determines an ample $T$-invariant Cartier divisor on $X_{\Sigma_P}$, $D_P=\sum_{F\mbox{\tiny{ facet of }} P}a_FD_F$.
In fact, $D_P$ restricted to $U_{\sigma_{\v}}$ for each $\v\in P\cap M$ is $D_P\vert_{U_{\sigma_{\v}}}=\divT(\chi^{-{\v}})$.
The Cartier data of $D_P$ is then determined by the set of lattice points $P\cap M$.   

The polytope of the divisor $D_P$ is $P$.  If $r\ge n-1$ then $rP$ is normal and $\varphi_{D_P}$ recovers the projective embedding of $X_P$ into projective space via $rP$.  

To summarize: a full-dimensional lattice polytope $P$ on $M_{\R}$ determines both a complete  fan $\Sigma_P$ on $N_{\R}$ and a $T$-invariant ample Cartier divisor $D_P$ on $X_P$.  This in turn determines a $T$-invariant closed projective embedding of the normal $T$-variety $X_P$.
But this map $P\to (\Sigma_P,D_P)$ is in fact a bijection.  (See: \cite[6.2.1]{CLS11}.) In other words, given a complete fan $\Sigma$ on $M_{\R}$ and a $T$-invariant ample Cartier divisor on $X_{\Sigma}$ determines a $T$-invariant closed projective embedding of $X_{\Sigma}$.

We call a complete fan $\Sigma$ on $N_{\R}$ \emph{projective} if there exists a $T$-invariant ample Cartier divisor $D$ on $X_{\Sigma}$.  This then determines a $T$-invariant closed projective embedding of $X_{\Sigma}$.

\subsection{Toric Morphisms and Resolutions of Singularities}

\begin{defn}~\cite[3.3.1]{CLS11}

  Let $T_i,i=1,2$ be split $\bk$-tori with character lattices $M_i,i=1,2$ and cocharacter lattices $N_i,i=1,2$.
  Let $\Sigma_i$ be a fan in $(N_i)_{\R}$, $i=1,2$.
  
      A $\Z$-linear map $\overline{\phi}$ is \emph{compatible} with $\Sigma_1$ and $\Sigma_2$ if and only if for every $\sigma_1\in \Sigma_1$, there exists a $\sigma_2\in \Sigma_2$ such that $(\overline{\phi})_{\R}(\sigma_1)\subseteq \sigma_2$.

      Note that if $\sigma_1\in \Sigma_1,\sigma_2\in \Sigma_2$ are as in the definition, there is a natural induced morphism $U_{\sigma_1}\to U_{\sigma_2}$ with respect to the $T_1$ and $T_2$ actions.
      
      A morphism $\phi:X_{\Sigma_1}\to X_{\Sigma_2}$ is a \emph{toric morphism} if $\phi$ restricts to a morphism $T_1\to T_2$ of algebraic groups.
      Note that toric morphisms are equivariant morphisms for the $T_1$ and $T_2$ actions.
    \end{defn}

\begin{thm}~\cite[3.3.4]{CLS11}

  The $\Z$-linear map $\overline{\phi}:N_1\to N_2$ induced by a toric morphism $\phi:X_{\Sigma_1}\to X_{\Sigma_2}$ defines a bijection between the set of toric morphisms between $X_{\Sigma_1}$ and $X_{\Sigma_2}$ and the set of $\Z$-linear maps between $N_1$ and $N_2$ which are compatible with $\Sigma_1$ and $\Sigma_2$.

 In particular, the toric automorphisms of $X_{\Sigma}$ are in bijection with $\Aut_{\Sigma}(N)$, the subgroup of $\Aut(N)$ which are compatible with $\Sigma$ and itself.     
    \end{thm}

    A refinement of a fan $\Sigma$ on $N_{\R}$ is a fan $\Sigma'$ on $N_{\R}$ such that every cone in $\Sigma'$ is contained in a cone of $\Sigma$ and $|\Sigma'|=|\Sigma|$.
    If $\Sigma'$ is a refinement of a fan $\Sigma$, then $\id_N:N\to N$ is compatible with $\Sigma',\Sigma$ and so determines a toric morphism $X_{\Sigma'}\to X_{\Sigma}$.

    \begin{defn}~\cite[p. 515]{CLS11}

      A \emph{star subdivision} of a fan $\Sigma$ in $N_{\R}$ and a primitive element $\v\in |\Sigma|\cap N$ (i.e. such that $\v$ extends to a $\Z$-basis of $N$)
       is the set of cones
       $$\Sigma^*(\v)=\Sigma_0(\v)\cup \Sigma'(\v)$$
       where $\Sigma_0(\v)=\{\sigma\in \Sigma: \v\not\in \sigma\}$
       and $\Sigma'(\v)=\{\Cone(\tau,\v):\sigma\in \Sigma\setminus \Sigma_0(\v),\tau\in \Sigma_0(\v), \v\cup \tau\subseteq \sigma\}$.
     \end{defn}

    \begin{prop}~\cite[11.1.6]{CLS11}

      A star subdivision $\Sigma^*(\v)$ of a fan $\Sigma$ in $N_{\R}$ is a refinement of the fan $\Sigma$ such that $\Sigma^*(\v)(1)=\Sigma(1)\cup \{\Cone(\v)\}$ and such that the induced toric morphism $\phi:X_{\Sigma^*(\v)}\to X_{\Sigma}$ is projective.
     \end{prop}

    \begin{thm}~\cite[11.1.9]{CLS11},\cite[I \S 2, Lemma 1-3]{KKMSD73}
      
       Every fan $\Sigma$ has a refinement $\Sigma'$ such that
       \begin{itemize}
       \item All cones in $\Sigma'$ are smooth.
       \item All smooth cones of $\Sigma$ are contained in $\Sigma'$.
       \item $\Sigma'$ is obtained from $\Sigma$ by a finite sequence of star subdivisions.
       \item The toric morphism $\phi:X_{\Sigma'}\to X_{\Sigma}$ is a projective resolution of singularities.
       \item If $X_{\Sigma}$ has a $T$-equivariant ample Cartier divisor $D$, then there exists  a $T$-equivariant ample Cartier divisor $D'$ on $X_{\Sigma'}$.  
\end{itemize}
       \end{thm}

     For our applications, we will use the barycentric subdivision of a fan.
     The barycentre $\v_{\sigma}$  of a cone $\sigma\in \Sigma$ with primitive generators
     $\bu_{\rho},\rho\in \sigma(1)$ is the minimal generator of $\Cone(\sum_{\rho\in \sigma(1)}\bu_{\rho})\cap N$.  The barycentric subdivision of $\Sigma$ is obtained from $\Sigma$ by a finite sequence of star subdivisions of the cones in $\Sigma$ with respect to their barycentres in order of decreasing dimension of the cones in $\Sigma$.  It is independent of the order of the cones in a given dimension.  The resulting simplicial fan  replaces each chain of cones in $\Sigma$, $\sigma_1\subseteq \cdots\subseteq \sigma_n$ where $\dim(\sigma_i)=i$, with a cone $\Cone(\v_{\sigma_1},\dots,\v_{\sigma_n})$. (See, for example, ~\cite[11.1.7]{CLS11}).

     To pass from a simplicial fan to a smooth fan, we need to subdivide the simplicial cones into smooth cones.  The multiplicity of a simplicial cone $\sigma$ with primitive generators $u_1,\dots,u_n$ is
     $$\mult(\sigma)=[\Span_{\R}(\sigma)\cap N:\sum_{i=1}^n\Z\bu_i].$$

     A simplicial cone $\sigma$ is smooth if and only if $\mult(\sigma)=1$.  Let
     $P_{\sigma}=\{\sum_{i=1}^n\lambda_i\bu_i:0\le \lambda_i<1\}$.
     It turns out that for a simplicial cone $\sigma$, the multiplicity of $\sigma$ can be reinterpreted as the number of interior lattice points in $P_{\sigma}$.  That is, $\mult(\sigma)=|P_{\sigma}\cap N|$.
     For any nonzero interior lattice point $\v$  of $\sigma\in \Sigma$, a star-subdivision $\Sigma^*(\v)$ replaces the cone $\sigma$ with cones of the form $\Cone(\tau,\v)$ where $\tau$ is a face of $\sigma$.  The multiplicities of the new cones $\Cone(\tau,\v)$ created in the star subdivision are strictly smaller than that of the multiplicity of $\sigma$. (See \cite[I,\S 1,Lemma 3]{KKMSD73}.)

     So given any normal projective toric $T$-variety  $X$ corresponding to a split torus  $T$, we have that there exists a fan $\Sigma$ on $N_{\R}$ such that $X\cong X_{\Sigma}$. We can then resolve singularities and obtain a smooth projective toric $T$-variety.        

\section{Algebraic \texorpdfstring{$\bk$}{k}-tori and their Toric Models}

\subsection{Algebraic \texorpdfstring{$\bk$}{k}-tori}

References for this section include~\cite{Vos83,Vos98}.

An algebraic $\bk$-torus $T$ is an algebraic $\bk$-group which is a $\bk$-form of a split torus, $\Gmk^r$.  
That is, after base change to the separable closure $\bk_s$, it becomes isomorphic to $\Gmks^r$.
Isomorphism classes of $\bk$-forms of a fixed quasi-projective $\bk$-variety $Y$ are determined by classes in
the non-abelian Galois cohomology group, $H^1(\Gamma_{\bk},\Aut_{\bk_s}(Y_{\bk_s}))$.
In our case,  the compact group $\Gamma_{\bk}$ acts trivially on the discrete group $\Aut_{\bk_s-\mathrm{gp}}(\Gmks^r)\cong \GL(r,\Z)$, and so the set of isomorphism classes of algebraic $\bk$-tori of dimension $r$ is in bijection with $H^1(\Gamma_{\bk},\Aut_{\bk_s-\mathrm{gp}}(\Gmks^r))=\Hom_{\mathrm{conts}}(\Gamma_{\bk},\GL(r,\Z))$.  

Explicitly, if $T$ is an algebraic $\bk$-torus, there exists a $\bk_s$-isomorphism $f:T_{\bk_s}\to \Gmks^r$.  For each $\sigma\in \Gamma_{\bk}$, $\sigma\in \Aut_k(\bk_s)$ induces natural $\bk$-automorphisms  of $T_{\bk_s}=T\otimes_k \bk_s$ and  $\Gmk^r\otimes_k \bk_s$
via the $\bk_s$ factor.  We will again denote these automorphisms by $\sigma$.
We define $\sigma(f)=\sigma\circ f\circ \sigma^{-1}$.
Then $\sigma(f)$ determines a map $\sigma(f):T_{\bk_s}\to \Gmks^r$ is also a $\bk_s$-isomorphism.  Then $h:\Gamma_{\bk}\to \Aut_{\bk_s}(T_{\bk_s})=\GL(r,\Z), \sigma\to f^{-1}\circ \sigma(f)$ is a 1-cocycle.  Since $\Gamma_{\bk}$ acts trivially on $\GL(r,\Z)$, $h$ is, in fact, a continuous group homomorphism. We can identify an algebraic $\bk$-torus $T$ of dimension $r$  with $h\in \Hom_{\mathrm{conts}}(\Gamma_{\bk},\GL(r,\Z))$ and hence with $h(\Gamma_{\bk})\in \GL(r,\Z)$.
Since any continuous homomorphism from the compact group $\Gamma_{\bk}$ to the discrete group $\GL(r,\Z)$ must have finite image, $h(\Gamma_{\bk})$ is a finite subgroup of $\GL(r,\Z)$.  By Galois theory, $\ker(h)$ is an open normal subgroup of $\Gamma_{\bk}$ of finite index, and so there exists a finite Galois extension $L/\bk$ such that $h(\Gamma_{\bk})\cong \Gal(L/\Bbbk)$.  

In particular, there is a bijection between the set of isomorphism classes of algebraic $\bk$-tori of dimension $r$  and the set of conjugacy classes of  finite subgroups of $\GL(r,\Z)$ which occur as Galois groups over $\bk$.    
Over an algebraically closed field $\overline{\bk}$, all algebraic $\overline{\bk}$-tori are split.
By contrast, at least for small values of $r$,  the finite subgroups of $\GL(r,\Z)$ may all be determined as Galois groups over $\Q$.

A $\bk$-form of a quasiprojective variety can always be split by a finite Galois extension.  In the case of an algebraic $\bk$-torus $T$ determined by $h\in \Hom_{\mathrm{conts}}(\Gamma_{\bk},\GL(r,\Z))$, the image $G=h(\Gamma_{\bk})$ is a finite subgroup of $\GL(r,\Z)$ and so there exists a finite Galois extension $L/\bk$, such that
$\Gamma_{\bk}/\ker(h)=\Gal(L/\Bbbk)$.  Note that $L$ splits $T$ so that  $T_L\cong \GmL^r$ and also $h$ factors through $h':G=\Gal(L/\Bbbk)\to \GL(r,\Z)$.

Let $L/\bk$ be a finite Galois extension with Galois group $G=\Gal(L/\Bbbk)$.  There is a bijection between the set of algebraic $\bk$-tori of dimension $r$ split by $L/\bk$ up to isomorphism and the set of $G$-lattices of rank $r$ up to isomorphism.
Explicitly, an algebraic $\bk$-torus $T$ split by $L$ determines its character lattice over $L$:  $\widehat{T_L}=\Hom_{L-\mathrm{gp}}(T_L,\GmL)$ which is equipped with a natural action of $G=\Gal(L/\Bbbk)$.  Conversely, given a $G$-lattice $M$, with group action $h:G\to \Aut(M)$, $G$ acts on the group algebra $L[M]=L\otimes_k \bk[M]$ via its Galois and lattice actions $u:G\to \Aut(L[M]), g\to g\otimes h(g)$ which induces an action $u^*:G\to \Aut_{L-\mathrm{gp}}(T_L)$ on $T_L$ since $T_L=\Spec(L[M])$. Since $L\otimes_{L^G}L[M]^G\cong L[M]$, one sees that  $T=\Spec(L[M]^G)$ is an algebraic $\bk$-torus split by $L/\bk$.

We now fix an algebraic $\bk$-torus $T$.  Let $L/\bk$ be a finite Galois splitting field with Galois group $G=\Gal(L/\Bbbk)$.  We note then that from the correspondence between algebraic $\bk$-tori split by $L/\bk$ and $G$-lattices, that the algebraic $\bk$-torus $T$ can be recovered from $T_L$ and $G$  as the quotient $T=T_L/u^*(G)$, where $u^*:G\to \Aut_{L-\mathrm{gp}}(T_L)$ is induced by $G$ acting via the Galois and lattice actions on the affine algebra of $T_L$.

\subsection{Toric Models of Algebraic Tori}
\label{section:toricmodels}
References for this section include: \cite{MP97,Xie18,Dun16,KV84,Kly82,Vos83}.
We remark that the definitions of toric models of algebraic tori in the literature are not at all uniform, and use different terminology.  Earlier references only discuss smooth projective toric models.  

We recall some definitions which we made in the section on toric varieties for a split algebraic $\bk$-torus.

Let $T$ be an algebraic $\bk$-torus, not necessarily split.

A \emph{toric $T$-variety} is a normal $\bk$-variety $X$ on which $T$ acts faithfully, together with a dense open $T$-orbit $U$.

A \emph{toric $T$-model} is a normal $\bk$-variety $X$ on which $T$ acts, together with an dense open $T$-orbit $U$ which is isomorphic to $T$.

Note that a toric $T$-model $X$ is equipped with a $T$-equivariant open embedding $T\cong U\subseteq X$ which determines the $T$-action on $X$.
Since the action of $T$ on itself is faithful, so is the action of $T$ on $X$.
$X$ is geometrically irreducible as a $\bk$-variety since $U_{\bk_s}\cong \Gmks^r$.

Let $T$ be an algebraic $\bk$-torus split by a finite Galois extension $L/\bk$ with Galois group $G=\Gal(L/\Bbbk)$.

Then $T_L$ is a split $L$-torus.
Let $M=\widehat{T_L}$ and $N=\ThatLdual$ be the character and cocharacter lattices of $T_L$. Note that $G=\Gal(L/\Bbbk)$ acts naturally on $T_L$ and hence on its character and cocharacter lattices.
Let $X$ be a toric $T$-model.  Then $X_L$ is a toric $T_L$-model. It is uniquely determined by a fan $\Sigma$ in $N_{\R}$ so that $X_L\cong X_{\Sigma}$.
Since each cone has $\{0\}$ as a face, the open embedding $T_L=X_{0}\to X_L\cong X_{\Sigma}$ is a toric morphism determined by the identity on $N$.
We recall that $G$ acts on $M$ via $h:G\to \Aut(M)$ and $G$ acts on $L[M]$ via $u:G\to \Aut(L[M]), g\to g\otimes h(g)$. This induces the action $u^*:G\to \Aut(T_L)$ since $T_L=\Spec(L[M])$.
The action of $G$ on $T_L=X_{0}$ extends to a toric automorphism of  $X_L\cong X_{\Sigma}$ making the open embedding $T_L=X_0\to X_L\cong X_{\Sigma}$ a $G$-equivariant map.  The action of $G$ on $X_{\Sigma}$ is then induced from an element of $\Aut_{\Sigma}(N)$.  If $h^*:G\to \Aut(N)$ denotes the action of $G$ on the dual lattice $N=M^0$, this means that $h^*(G)$ is contained in $\Aut_{\Sigma}(N)$.  That is, $G$ permutes the cones of $\Sigma$.  So the action $u^*:G\to \Aut(T_L)$ extends to $u^*:G\to \Aut(X_{\Sigma})$.

If we additionally assume that our toric $T$-model is a projective $\bk$-variety, then so is the toric $T_L$-model $X_L\cong X_{\Sigma}$.
Since $X_{\Sigma}$ is a projective $L$-variety, the quotient $X_{\Sigma}/u^*(G)$ exists.~\footnote{\cite[III, Prop 19,Ex 1]{Se88} Note that if $H$ is a finite group acting on a $\bk$-variety $Y$, a necessary and sufficient condition for the existence of the $\bk$-variety $Y/H$ is that each point in $Y$ have a $H$-invariant open affine neighbourhood. In that case, $Y/H$ is glued together from the quotients of the $H$-invariant open affine varieties which cover $Y$.  A quasi-projective $\bk$-variety satisfies this condition.}

The open embedding $T_L/u^*(G)\to X_{\Sigma}/u^*(G)$ recovers the open embedding $T\to X$.

Conversely, we can construct a projective toric $T$-model.  Suppose $L/\bk$ splits $T$ and $G=\Gal(L/\Bbbk)$, and $N=\ThatLdual$ is the cocharacter lattice of $T_L$.
We can construct a projective $G$-invariant toric $T_L$-model by  determining a projective fan $\Sigma$ on $N_{\R}$ which is $G$-invariant. That is, such that
$G$ acts by permuting the cones of $\Sigma$. We obtain a $G$-equivariant open embedding of $T_L$ into $X_{\Sigma}$, where the action on $X_{\Sigma}$ is given by $u^*:G\to \Aut(X_{\Sigma})$.
Then, we will take our projective toric $T$-model to be $X_L=X_{\Sigma}/u^*(G)$.
We note that by construction, $X_L\cong X_{\Sigma}$.

To summarize,  if $T$ is an algebraic $\bk$-torus split by the Galois extension $L/\bk$ with Galois group $G$, projective toric $T$-models are completely determined by a projective $G$-invariant fan on $(\ThatLdual)_{\R}$.

\begin{rem}
For references, see ~\cite{Xie18,MP97,Kly82,KV84}.

  For a non-split algebraic $\bk$-torus $T$, a toric $T$-variety is not necessarily a toric $T$-model.  However if $X$ is a toric $T$-variety with dense open $T$-orbit $U$, then $U$ is a principal homogeneous variety over $T$.  That is, $U$ determines a class in $H^1(\bk,T)$.
If $U$ has a $\bk$-rational point $x$, then the map $T\to U, t\to tx$ determines an isomorphism  and hence $X$ is a toric $T$-model.
Note that if $T$ is split, then any principal homogeneous variety is trivial and so $U$ must have a $\bk$-rational point.  This means that all toric $T$-varieties are toric $T$-models if $T$ is split.

In the general case, a toric $T$-variety $X$ determines a toric $T$-model $X^*$ which is unique up to $T$-isomorphism and which satisfies $X_{\bk_s}\cong (X^*)_{\bk_s}$.  $X^*$ is called the associated toric $T$-model of $X$.  Note that $X^*$ is determined by $(X\times U)/T$ where $T$ acts diagonally on $X\times U$.  Conversely, $X$ is given by $(X^*\times U)/T$ where $T$ acts on $X^*\times U$ via $t(x,u)=(xt,t^{-1}u)$.
\end{rem}

\subsection{Projective toric models and Demazure models}
~\label{section:demazure}

\smallskip

References for this section include~\cite{Vos83,Vos98,Kun87,KM08,CTHS05}.

Given an algebraic $\bk$-torus $T$ with Galois splitting field $L/\bk$ and Galois group $G=\Gal(L/\Bbbk)$, we let $M=\widehat{T_L}$ and $N=M^0=\ThatLdual$ be the character and cocharacter lattices of $T_L$.

We need to construct a $G$-invariant projective fan $\Sigma$ on $N_{\R}$ which we will then refine to a $G$-invariant smooth projective fan.
In~\cite{CTHS05}, they start with an arbitrary projective fan  on $N_{\R}$ and  determine a projective $G$-invariant fan by taking intersections of
$G$-orbits of cones. They then proceed to determine a smooth projective $G$-invariant fan from the projective $G$-invariant fan in a number of steps.

We will instead start with producing a projective $G$-invariant fan $\Sigma_0$ on $N_{\R}$, which we will use to determine a smooth projective $G$-invariant fan on $N_{\R}$.

We will use two natural approaches, which are dual to each other.

\noindent\textbf{Approach 1:}

Take a $G$-stable finite subset $S$ of $M$ which contains a $\Z$-basis of $M$ and let $P=\Conv(S)$ be the associated $G$-invariant polytope. Then $\Sigma_0=\Sigma_P$, the normal fan of $P$,  is a projective $G$-invariant fan in $N_{\R}$. If we additionally assume that $P$ has  $0$ as an interior point, then $\Sigma_0=\Sigma_P$ would consist of cones over faces of $P^0$. In this case, the polar dual polytope $P^0$ is a full-dimensional rational polytope in $N_{\R}$.

\noindent\textbf{Approach 2:}

Let $S$ be a finite subset of $N_{\Q}$ such that: $S$ is $G$-stable, $P=\Conv(S)$ is a full-dimensional rational polytope in $N_{\R}$ which contains 0 in its interior and  such that $P^0$ is a lattice polytope in $M$ whose vertex set contains a $\Z$-basis of $M$.
Under these assumptions, $\Sigma_{P^0}$, the normal fan of the lattice polytope $P^0$,  then consists of cones over faces of $P$.  Then $\Sigma_0=\Sigma_{P^0}$ is a projective $G$-invariant fan in $N_{\R}$.

The second approach is probably closest to what was used in the literature, for example in~\cite{Vos83,KV84,Kun87}.

In either case,
for our algebraic $\bk$-torus $T$ split by a Galois extension $L/\bk$ with Galois group $G$, a $G$-invariant projective fan $\Sigma_0$ on $(\ThatLdual)_{\R}$ determines an  $G$-equivariant projective toric $T_L$-variety $X_{\Sigma_0}$.  Hence we have a $G$-equivariant open embedding of $T_L$ into $X_{\Sigma_0}$.  Passing to the quotient by $G$, we obtain a projective toric $T$-model.  

For an arbitrary algebraic $\bk$-torus $T$ split by $L/\bk$ with Galois group $G$, the resulting $G$-invariant projective fan on $(\ThatLdual)_{\R}$ is not even simplicial.  But we can resolve singularities $G$-equivariantly to obtain a simplicial and then smooth $G$-equivariant projective fan $\Sigma$ on $(\ThatLdual)_{\R}$.  We will call the resulting smooth projective $G$-equivariant toric $T_L$-model, $X_{\Sigma}$, a \emph{Demazure model} for $T$.  Note that a Demazure model for $T$ is not unique.  The open $G$-equivariant embedding of $T_L$ into $X_{\Sigma}$, determines an open embedding of $T$ into a smooth projective toric $T$-model, after passing to the quotient by $G$.  We remark that the process of determining a $G$-invariant simplicial projective fan and then a $G$-invariant smooth projective fan requires one to fully understand the action of the group on the faces  of the polytope used in the construction.

\subsection{Flasque Resolutions of algebraic k-tori}

\smallskip

References for this section include~\cite{Vos98,CTS77}.

Recall that an algebraic $\bk$-variety $X$ is $\bk$-\emph{rational} if $X$ is birationally isomorphic to $\P^n_k$ for $n=\dim(X)$.  $X$ is \emph{stably $\bk$-rational} if $X\times_k \PP^m$ is $\bk$-rational for some $m\ge 0$. $X$ is \emph{retract $\bk$-rational} if there exists a rational map $f:X\brokrarr \A^n_{\bk}$ with a rational section $g:\A^n_{\bk}\brokrarr X$.  (That is, such that $g\circ f=\id$.)

A $G$-lattice $M$ is called $G$-\emph{flasque} (respectively $G$-\emph{coflasque}) if $H^1(H,M^0)=0$ for all subgroups $H\le G$ ($H^1(H,M)=0$ for all subgroups $H\le G$). Note that all $G$-permutation lattices are $G$-flasque and $G$-coflasque.  The same holds for a $G$-permutation projective lattice, i.e a $G$-lattice which is a direct summand of a $G$-permutation lattice.
Two $G$-lattices $X,Y$ are said to be \emph{similar} if there exists a $G$-permutation lattice $P$ such that $X\oplus P\cong Y$.

Let $T$ be an algebraic $\bk$-torus with Galois splitting field $L/\bk$ and Galois group $G=\Gal(L/\Bbbk)$.

Voskresenskii determined a necessary and sufficient condition for an algebraic $\bk$-torus $T$ to be stably $\bk$-rational assuming the existence of a smooth projective $T$-model $X$. (He first assumed that $\bk$ was of characteristic zero to use Hironaka's theorem to ensure that such a smooth projective model existed.) 
Let $\Div_{X_L-T_L}(X_L)$ be the group of divisors of $X_L$ with support in the closed subset $X_L-T_L$ and let $\Pic(X_L)$ be the Picard group of $X_L$.  

He showed that
$$\0\to \widehat{T_L}\to \Div_{X_L-T_L}(X_L)\to \Pic(X_L)\to \0$$
is an exact sequence of $G=\Gal(L/\Bbbk)$-lattices such that  $\Div_{X_L-T_L}(X_L)$ is a $G$-permutation lattice and such that $\Pic(X_L)$ is a flasque $G$-lattice.

Moreover, he showed that if $Y$ were another smooth projective $T$-model, then $\Pic(X_L)$ and $\Pic(Y_L)$ are similar as $G$-lattices.
This allowed him to define a birational invariant for $T$: $\rho(T)=[\Pic(X_L)]$, the similarity class of the flasque lattice $\Pic(X_L)$, for any smooth projective toric $T$-model $X$.

He then proved that an algebraic $\bk$-torus $T$ is stably $\bk$-rational if and only if $\rho(T)=0$, equivalently if and only if $\Pic(X_L)$ is stably permutation as a $G$-lattice.
Later, Saltman~\cite{Sa84b} showed that an algebraic $\bk$-torus $T$ is retract-rational if and only if $\rho(T)$ is the class of a permutation projective $G$-lattice.

A flasque  resolution for a $G$-lattice $M$ is a short exact sequence of $G$-lattices $$0\to M\to P\to F\to 0$$
such that $P$ is $G$-permutation and $F$ is $G$-flasque.
$F$ is determined up to similarity by $M$.  That is, given another flasque resolution for $M$: $$0\to M\to P'\to F'\to 0,$$  $F$ and $F'$ are similar as $G$-lattices.

The $\Z$-dual of the short exact sequence corresponding to a coflasque resolution for a $G$-lattice $M$ determines a flasque resolution for its dual $M^0$.  It is always possible in principle to determine a coflasque resolution for a $M$ (and hence, by duality, a flasque resolution for $M^0$), assuming one knows the structure of the lattice restricted to any subgroup.  In fact, $\oplus_{H\le G}\Z[G/H]\otimes M^H\to M$ gives a surjection that determines a coflasque resolution.  However, note that in order to work with such a resolution, and to determine properties of its kernel, one needs an understanding of the subgroup structure of $G$ up to conjugacy and the fixed point lattices for each of its subgroups.

Nevertheless, this shows that one can construct a flasque resolution for $T$ by determining a flasque resolution for $\widehat{T_L}$ as a $\Gal(L/\Bbbk)$-lattice, without the necessity of constructing a smooth projective model.  This allows one to determine birational properties of the torus from the structure of the character lattice $\widehat{T_L}$ as a $\Gal(L/\Bbbk)$-lattice.

There are many important applications of flasque resolutions for algebraic tori  beyond the ones mentioned above.  For example, Colliot-Th\'el\`ene and Sansuc, in ~\cite{CTS77}, showed that a flasque resolution for an algebraic torus can be used to determine its set of $R$-equivalence classes.

\section{Root systems and Root Polytopes}

\subsection{Root systems and Affine root systems}

References for this material include~\cite{Hum90,Kan01,Bou02}.

Let $\Phi$ be an irreducible crystallographic reduced root system on the $\R$-
vector space $V=\R\Phi$, equipped with a positive definite bilinear form 
$(\cdot,\cdot)$.  This determines a non-degenerate duality pairing $V\times V^*\to \R$ denoted by $\langle \cdot,\cdot\rangle$.  For each $\alpha\in \Phi$, let  $s_{\alpha}$ be the associated reflection on $V$.  There is a unique $\alpha^{\vee}\in V^*$, called the coroot of $\alpha$, such that $s_{\alpha}(x)=x-\langle x,\alpha^{\vee}\rangle\alpha$.
Under the isomorphism $V\to V^*$ determined by the positive definite bilinear form, we identify $\alpha^{\vee}\in V^*$ with the vector $\frac{2}{(\alpha,\alpha)}\alpha\in V$.  Under this identification, 
$\langle \alpha,\beta^{\vee}\rangle=(\alpha,\beta^{\vee})$ for all $\alpha,\beta\in \Phi$.
Then $\Phi^{\vee}=\{\alpha^{\vee}: \alpha\in \Phi\}$ is  an irreducible crystallographic reduced root system in $V^*$.  Note that the coroot system $\Phi^{\vee}$ is often identified with a root system in $V$ under the identification $\alpha^{\vee}=\frac{2\alpha}{(\alpha,\alpha)}$.

Let $\Pi=\{\alpha_1,\dots,\alpha_n\}$ be a basis of simple roots for $\Phi$ and 
$\Phi_+$ be the set of positive roots with respect to $\Pi$. Then $\Pi^{\vee}=\{\alpha_1^{\vee},\dots,\alpha_n^{\vee}\}$ forms a corresponding  basis of simple roots for the dual root system $\Phi^{\vee}$.

The Weyl group of $\Phi$ is determined as
$$W(\Phi)=\langle s_{\alpha}:\alpha\in \Phi\rangle$$
The automorphism group of the root system $\Phi$ is the set of $\R$-linear 
automorphisms of $\R\Phi$ which leaves $\Phi$ stable.  $\Aut(\Phi)$ is the semidirect product of $W(\Phi)$ by the diagram automorphism group $D(\Phi)$, which preserves the Dynkin graph of $\Phi$. 

Let $\Z\Phi$ be the root lattice of $\Phi$.  Note that $\Pi$ is a $\Z$-basis
of $\Phi$.  
The set of  fundamental coweights of $\Phi$, $\{\omch_1,\dots,\omch_n\}$,  
is the dual basis of $\Pi$ with respect to $(\cdot,\cdot)$.  The fundamental coweights  form a basis of the coweight lattice $\Lambda(\Phi^{\vee})$ of $\Phi$. Note that $\Lambda(\Phi^{\vee})$ is also the weight lattice of the coroot system $\Phi^{\vee}$. 
Dually, the set of fundamental weights of $\Phi$, $\{\omega_1,\dots,\omega_n\}$, is the dual basis of $\Pi^{\vee}$ with respect to $(\cdot,\cdot)$ and forms a basis of the weight lattice $\Lambda(\Phi)$ of $\Phi$.

Every positive root in $\Phi$ can be written as a non-negative $\Z$-linear combination of roots in $\Pi$.
There is a natural partial order on roots of $\Phi$: $\alpha\ge \beta$ if and only if $\alpha-\beta\in \Phi_+$.
Under this partial order, there is a unique highest root: $\theta\in \Phi_+$.  We write $\theta=\sum_{i=1}^nm_i\alpha_i$ uniquely as a non-negative linear combination of simple roots in $\Pi$.

Associated to an irreducible crystallographic (reduced) root system $\Phi$, there is an affine Weyl group.  If $H_{\alpha}=\ker(s_{\alpha})$ is the reflecting hyperplane of the root $\alpha\in \Phi$, then for each $r\in \Z$, there is a reflection $s_{\alpha,r}$ in the affine hyperplane $H_{\alpha,r}$.  The affine Weyl group of $\Phi$ is the group generated by the affine reflections, that is:
$$W_{\aff}(\Phi)=\langle s_{\alpha,r}:\alpha\in \Phi,r\in \Z\rangle$$  $W_{\aff}(\Phi)$ is a subgroup of $\Aff(V)$, the group of affine transformations of $V$.  
Under the identification of $\Phi^{\vee}$ with a root system in $V$, the reflection $s_{\alpha,r}$ in the affine hyperplane $H_{\alpha,r}$ can be identified as $s_{\alpha,r}=t(r\alpha^{\vee})s_{\alpha}$ where $t(v):V\to V$ is the translation of $V$ by the vector $v$.  Using this observation, one can show that the coroot lattice $\Z\Phi^{\vee}\cong \Z t(\Phi^{\vee})$ is a normal subgroup of $W_{\aff}(\Phi)$ and that 
$W_{\aff}(\Phi)$ is the semidirect product of $\Z\Phi^{\vee}$ by $W(\Phi)$.

Letting $\delta$ be the basic imaginary root, one can associate an  affine root system $\widehat{\Phi}=\Phi+ \Z\delta$ to $W_{\aff}(\Phi)$. Here the affine reflection $s_{\alpha,r}$ is associated with the affine root $\alpha+r\delta$.
Then  $\widehat{\Phi}$ is a crystallographic affine root system in $\R\widehat{\Phi}$, in which the inner product is extended from that on $\R\Phi$, by declaring the kernel to be $\R\delta$. Let $\alpha_0=-\theta+\delta$ and identify $s_{\alpha_0}=s_{-\theta,1}$.  Then $\widehat{\Pi}=\{\alpha_0\}\cup \Pi$ is a root basis for $\widehat{\Phi}$, and the set of positive roots of $\widehat{\Phi}$ with respect to $\widehat{\Pi}$ is $$\widehat{\Phi}^+=\Phi^+\cup (\Phi+\Z^+\delta)$$
The affine Weyl group $W_{\aff}(\Phi)$ is generated by $\{s_{\alpha}:\alpha\in \Pi\}\cup \{s_{\alpha_0}\}$, i.e. the set of affine reflections corresponding to the roots in $\widehat{\Pi}$.
To determine the extended Dynkin diagram of $\widehat{\Phi}$, an additional node  is added to the Dynkin diagram of $\Phi$ which corresponds to $\alpha_0=-\theta+\delta$.  

Note that if $\Phi$ is simply-laced, then $(\alpha,\alpha)=2$ for all $\alpha\in \Phi$ and so $\Phi^{\vee}=\Phi$ and the root and coroot lattices, and the weight and coweight lattices, respectively, coincide.

There is a classification of irreducible crystallographic root systems $\Phi$ into 4 infinite families: $A_n,n\ge 1$, $B_n,n\ge 2$, $C_n,n\ge 2$, $D_n,n\ge 4$ and exceptional root systems $E_6,E_7,E_8,F_4,G_2$.
To each irreducible crystallographic root system $\Phi$ associated a Dynkin diagram with vertices corresponding to the simple roots $\Pi$. The edges encode information about $\langle \alpha_i,\alpha_j^{\vee}\rangle$ for $\alpha_i,\alpha_j\in \Pi$.

An arbitrary crystallographic root system $\Phi$ is a disjoint union of
irreducible root systems $\Phi=\sqcup_{i=1}^s\Phi_i^{r_i}$, where $\Phi_i, i=1,\dots,s$ are distinct irreducible root systems. The Weyl group of $\Phi$ is the product of the Weyl groups of its irreducible root subsystems.  So $W(\Phi)=\prod_{i=1}^sW(\Phi_i)^{r_i}$. Note that $W(\Phi)=\prod_{i=1}^sW(\Phi_i)^{r_i}$ acts diagonally on the root lattice $\Z\Phi=\oplus_{i=1}^s(\Z\Phi_i)^{r_i}$.
The automorphism group of $\Phi$ is $\Aut(\Phi)=\prod_{i=1}^s\Aut(\Phi_i)^{r_i}\rtimes S_{r_i}$.  Note that each factor $\Aut(\Phi_i)^{r_i}\rtimes S_{r_i}$ acts as a wreath product on the root lattice $(\Z\Phi_i)^{r_i}$.

\subsection{Root Polytopes}~\label{section:rootpolytope}
Root polytopes arise naturally in the representation theory of Lie groups and algebraic groups, among other applications, and as such can be considered a classical subject.  We follow the exposition of ~\cite{CM14,CM15}, but acknowledge many earlier references, such as ~\cite{Vin90}.

Let $\Phi$ be a crystallographic reduced root system.  The \emph{root polytope} of $\Phi$ is the convex hull of $\Phi$: $P(\Phi)=\Conv(\Phi)$.\footnote{Warning: Note that some references (eg. \cite{Mes11}) refer to the root polytope of a root system as the convex hull of the positive roots of a root system.}
Note that the Weyl group and the automorphism group of $\Phi$ act naturally on $P(\Phi)$ by permuting its faces.

In this section, we will assume that $\Phi$ is irreducible. Let $\Pi=\{\alpha_1,\dots,\alpha_n\}$ be a set of simple roots and $\{\omch_1,\dots,\omch_n\}$ be the fundamental coweights.

Let $[n]:=\{1,\dots,n\}$.
  For each $I\subseteq [n]$  and each $i\in [n]$, we set
  $$F_I=\{x\in P(\Phi): \langle x,\omch_i\rangle=m_i, i\in I\}$$ and $F_i=F_{\{i\}}$.
  The $F_I$ are proper faces of $P(\Phi)$, all containing the highest root $\theta$.  They are called the coordinate faces of $P(\Phi)$.
  Note that $F_{\emptyset}=P(\Phi)$, and, for all $I\ne \emptyset$, $F_I=\cap_{i\in I}F_i$ is a proper non-empty face.  The set of faces $\{F_I:\emptyset \ne I\subseteq [n]\}$, are the standard parabolic faces of $P(\Phi)$.
  Set $V_i=F_i\cap \Phi$ and $V_I=F_I\cap \Phi$.  Note that $F_I=\Conv(V_I)$.
  
A face of $P(\Phi)$ of the form $wF_I$ for some $w\in W(\Phi)$ and $I\subseteq [n]$ is called a parabolic face.

The root polytope $P(\Phi)$ is  naturally associated with the extended root system $\widehat{\Phi}$.
Let $\widehat{\Pi}$ be an extension of $\Pi$ to a simple system of $\widehat{\Phi}$ and set $\alpha_0=-\theta+\delta$ so that $\widehat{\Pi}=\Pi\cup \{\alpha_0\}$.
Given $I\subseteq [1,n]$, let $\Pi_I=\{\alpha_i:i\in I\}$, $\Phi_I=\Phi\cap \Z\Pi_I$, the parabolic root system determined by $I$ with Weyl group $W(\Phi_I)=\langle s_{\alpha_i}:i\in I\rangle$.
Let $\widehat{\Pi}_I=\Pi_I\cup\{\alpha_0\}$ and $\widehat{\Phi}_I=\widehat{\Phi}\cap \widehat{\Pi}_I$ be the parabolic root system determined by $\widehat{\Pi}_I$.

Let $\Gamma_0(I)$ be the set of roots of $\widehat{\Pi}$ lying in the connected component of $\alpha_0$ in the Dynkin graph of $\widehat{\Phi}_{[n]-I}$, and let
$$\overline{I}=\{j\in [n]:\alpha_j\not\in \Gamma_0(I)\}$$
$$\partial I:=\{i\in I, \exists \beta\in \Gamma_0(I), (\beta,\alpha_i)\ne 0\}.$$

Note that $j\in \overline{I}$ if and only if $\alpha_j$ is not connected to $\alpha_0$ in the Dynkin graph of $\widehat{\Phi}_{[n]-I}$ and $i\in \partial I\subseteq I$ if and only  $\alpha_i$  is connected to a root in the connected component of $\widehat{\Phi}_{[n]-I}$ containing $\alpha_0$.

\begin{thm}~\label{thm:facesrootpolytope} 
  Let $I\subseteq [n]$.  Then:
  \begin{enumerate}
  \item $V_I$ has a unique minimum element $\alpha_I$ and $V_I=\{\alpha\in \Phi^+: \alpha\ge \alpha_I\}$.
  \item $\{J\subseteq [n]: F_J=F_I\}=\{J\subseteq [n]: \partial I \subseteq J\subseteq \overline{I}\}$.
  \item The dimension of $F_I$ is $n-|\overline{I}|$.
  \item The stabilizer subgroup of $F_I$ is $W(\Phi_{[n]-\partial I})$.
  \item All faces of $P(\Phi)$ are parabolic.  More precisely, the set of faces:
    $$\{F_I:I\subseteq [n], \widehat{\Phi}_{[n]-I}\mbox{ is irreducible}\}$$
    is a complete set of representatives of the $W(\Phi)$-orbits.
  \item  The facets are the maximal elements in the set $\{F_i: i\in [n]\}$.  $F_i$ is a facet if and only if $\widehat{\Phi}_{[n]-\{i\}}$ is irreducible.
  \end{enumerate}
\end{thm}

\begin{proof} The faces of weight polytopes (the convex hull of a Weyl group orbit of an admissible weight) were first described by Vinberg in~\cite[Prop 3.2]{Vin90}. Root polytopes are a special case of weight polytopes as the convex hull of the Weyl group orbit of the highest root.  This theorem, proved in ~\cite{CM15} and summarized in ~\cite{CM14}, is a  much more detailed description of the faces of root polytopes in terms of the associated affine root system.  
\end{proof}

        \section{Dade Groups}
        \label{section:dade}
        
        Dade determined the maximal finite subgroups of $\GL(4,\Z)$ up to conjugacy in ~\cite{Dad65}.  Presumably, the maximal finite subgroups of $\GL(r,\Z)$ up to conjugacy for $r=2,3$ were known by the time of his paper.  At any rate, they can be described using a subset of his techniques.   Tahara determined all the finite subgroups of $\GL(3,\Z)$ up to conjugacy in ~\cite{Tah71}. In this section, we describe the lattices preserved by the maximal subgroups of $\GL(r,\Z)$ up to isomorphism for $r\le 4$.  We recall that an algebraic $\bk$-torus $T$ of dimension $r$ split by a finite Galois extension $L/\bk$ can be determined by its character lattice $\widehat{T_L}$ considered as a $\Gal(L/\Bbbk)$-lattice.  We will later discuss the Demazure models for the algebraic $\bk$-tori of dimension $r\le 4$ corresponding to maximal subgroups of $\GL(r,\Z)$.
        
        There are only two  maximal finite subgroups of $\GL(2,\Z)$.  Their corresponding lattices are
        $$(\Z B_2,\Aut(B_2))=((\Z A_1)^2,\Aut(A_1^2)), (\Z A_2,\Aut(A_2)).$$

The maximal finite subgroups of $\GL(3,\Z)$ and their corresponding lattices are as follows:
\begin{lem}~\cite{Tah71}
  
Let $G_s=$\textup{DadeGroup(3,s)} for $s=1,\dots,4$.
Then the corresponding lattices are:
\begin{enumerate}
\item $(M_{G_1},G_1)=(\Z A_2\oplus \Z A_1,\Aut(A_2)\times \Aut(A_1))$.
\item $(M_{G_2},G_2)=((\Z A_1)^3,\Aut(A_1^3))$
\item $(M_{G_3},G_3)=(\Z A_3, \Aut(A_3))$.  
\item $(M_{G_4},G_4)=(\Lambda(A_3), \Aut(A_3))$.
\end{enumerate}
\end{lem}

The maximal subgroups of $\GL(4,\Z)$ and their corresponding lattices are as follows:
\begin{lem}~\cite{Dad65}
  
Let $G_s=$\textup{DadeGroup(4,s)} for $s=1,\dots,9$.
Then the corresponding lattices are:
\begin{enumerate}
\item $(M_{G_1},G_1)=((\Z A_1)^2\oplus \Z A_2,\Aut(A_1^2)\times \Aut(A_2))$.
\item $(M_{G_2},G_2)=(\Lambda(A_3)\oplus \Z A_1,\Aut(A_3)\times \Aut(A_1))$.
\item $(M_{G_3},G_3)=(\Z A_3\oplus \Z A_1,\Aut(A_3)\times \Aut(A_1))$
\item $(M_{G_4},G_4)=(\Z A_2\otimes \Z A_2,((W(A_2)\times W(A_2))\rtimes C_2)\times C_2)$
\item $(M_{G_5},G_5)=(\Z A_2\oplus \Z A_2, \Aut(A_2\times A_2))=(\Z A_2\oplus \Z A_2, (\Aut(A_2)\times \Aut(A_2))\rtimes C_2)$.
\item $(M_{G_6},G_6)=(\Z A_4,\Aut(A_4))$.
\item $(M_{G_7},G_7)=(\Lambda(A_4),\Aut(A_4))$.
\item $(M_{G_8},G_8)=((\Z A_1)^4,\Aut(A_1^4))=(\Z B_4,W(B_4))$.
\item $(M_{G_9},G_9)=(\Z D_4,\Aut(D_4))=(\Z F_4,W(F_4))$.
\end{enumerate}
\end{lem}

\begin{proof}
This result was deduced from the presentation of the Dade Groups in the GAP library in ~\cite{Lem17,GAP4},
  We show how this result can be deduced from the proof (rather than the statement) of the classification of maximal subgroups of $\GL(4,\Z)$ in~\cite{Dad65}.
  Dade expressed these groups as automorphism groups of quadratic forms in the main theorem of his paper.  But in the proof he provided more details.  His descriptions can then be shown to be equivalent to the above lattice descriptions in terms of root systems.

  Given a fixed lattice $L$ inside a real-vector space $V$, and a finite subset $S\subseteq L$, he defines $G(S,L)$ as the subgroup of $\GL(V)$ which stabilizes $L$ and $S$.  In the course of the proof, he determines each maximal subgroup of $\GL(4,\Z)$ as a group of the form $G(S,L)$ for a lattice $L$ in a 4-dimensional real vector space.

  For the group Dade calls $Qn$, he takes $L=\Span_{\Z}\{1,i,j,k,\frac{1+i+j+k}{2}\}$ to be  the lattice of unit quaternions with basis $\{1,i,j,\frac{1+i+j+k}{2}\}$.
  The set $S$ is the unit group $U$ of the ring of integral quaternions.
  $S=\{\pm 1,\pm i,\pm j,\pm k, \frac{\pm 1,\pm i,\pm j,\pm k}{2}\}$.
  Considering that the short roots of the root system $F_4$ are  $$S_0=\{\pm \e_i,\frac{\sum_{i=1}^4\epsilon_i\e_i}{2},\epsilon_i=\pm 1,i=1,\dots,4\}$$ and that a root system basis of $F_4$ is $\{\e_2-\e_3,\e_3-\e_4,\e_4,\frac{\e_1-\e_2-\e_3-\e_4}{2}\}$,
  it is straightforward to see that $\beta=\{\e_1,\dots,\e_3,\frac{\e_1+\e_2+\e_3+\e_4}{2}\}$ is an alternate $\Z$-basis of $\Z F_4$.
    Since $S_0$ is the set of short roots of $F_4$, the reflections of $F_4$ stabilize $S_0$ and so $W(F_4)$ stabilizes $S_0$.  Since $\beta\subseteq S_0\subset \Z\beta=\Z F_4$, we see that $W(F_4)$ stabilizes this lattice.  Then by order arguments, the group $Q_n$ corresponds to  $(\Z F_4,W(F_4))$.  
    So this corresponds to $(M_{G_9},G_9)$ above.
  Note that this lattice may also be described as $\Lambda(D_4)$ with group $\Aut(D_4)$.

  The group Dade calls $Cu_n$ corresponds to $G(S,L)$ where $L=\oplus_{i=1}^n\Z\e_i$ and $S=\{\pm \e_i:i=1,\dots,n\}$.
  This lattice corresponds to $(\Z B_n,W(B_n))$ or equivalently $((\Z A_1)^n,
  \Aut(A_1^n))$ so that for $n=4$, we have recovered $(M_{G_8},G_8)$.
  For $n=3$, we have recovered the second Dade group.  For $n=2$, we have recovered the first maximal subgroup.
  
  The group Dade calls $Su_n$ corresponds to $G(S,L)$ where 
  $\epsilon_n:\oplus_{i=1}^n\Z\e_i\to \Z$, $\e_i\to 1$  is the augmentation map, $L=\ker(\epsilon_n)$ is the augmentation ideal and $S=\{\e_i-\e_j:1\le i\ne j\le n\}$.
  This lattice corresponds to $(\Z A_n,\Aut(A_n))$.
  For $n=4$, we have recovered $(M_{G_6},G_6)$.  We note that for $n=3$, we recover the second Dade subgroup.  For $n=2$, we recover the second maximal subgroup.

  The group Dade calls $Pu_n$ corresponds to $G(S,L)$ where $L=(\oplus_{i=1}^{n+1}\Z \e_i)/\Z(\sum_{i=1}^{n+1} \e_i)$ and $S=\{\pm \e_i+\Z(\sum_{i=1}^{n+1}\e_i):i=1,\dots,n+1\}$.  In the paper, $L$ is identified with the image $L'$ of the  projection of the lattice $\oplus_{i=1}^{n+1}\Z\e_i$ in  $\Q A_n$. Under this projection, $S$ is mapped onto $S'=\{\e_i-\frac{\sum_{j=1}^n\e_j}{n+1}:i=1,\dots,n\}$.
  This lattice corresponds to $(\Lambda(A_n),\Aut(A_n))$.  For $n=4$, we have recovered $(M_{G_7},G_7)$.  For $n=3$, we recover the 4th Dade subgroup.

  The remaining maximal groups are determined by combinations of the above families of lattices from lower dimensions.

  The group which Dade denotes as $Su_2\otimes Cu_2$ has lattice $\Z A_2 \oplus (\Z A_1)^2$ and group $\Aut(A_2\times A_1^2)$ and so recovers $(M_{G_1},G_1)$.

  The group denoted by $Su_3\otimes Cu_1$ has lattice $\Z A_3\oplus \Z A_1$ and group $\Aut(A_3)\times \Aut(A_1)$ and so recovers $(M_{G_2},G_2)$.
  Note that $Su_2\otimes Cu_1$ recovers the first Dade subgroup of dimension 3.

  The group denoted by $Pu_3\otimes Cu_1$ has lattice $\Lambda(A_3)\oplus \Z A_1$ and group $\Aut(A_3)\times \Aut(A_1)$ and so recovers $(M_{G_3},G_2)$.
  Note that $Su_2\otimes Cu_1$ recovers the first Dade subgroup of dimension 3.

  The group denoted by $Su_2^{(2)}$ has lattice $\Z A_2\oplus \Z A_2$ and group the wreath product of $\Aut(A_2)$ by $S_2$, which is the same as $\Aut(A_2\times A_2)$.  This recovers $(M_{G_5},G_5)$.

  The group denoted by $Su_2^{\otimes 2}$ has lattice $\Z A_2\otimes \Z A_2$ and subset
  $$S=\{(\e_i-\e_j)\otimes (\e_r-\e_s): 1\le i\ne j\le 3,1\le r<s\le 3\}$$
  The group is determined as $((S_3\times S_3)\rtimes C_2)\times C_2$.
  This recovers $(M_{G_4},G_4)$.
\end{proof}

\section{Toric Variety of a Root System as a Demazure model}
\label{section:toricroot}

Let $\Phi$ be a crystallographic reduced root system.  
Take $T$ to be an algebraic $\bk$-torus split by Galois extension $L/\bk$ with Galois group $\Gal(L/\Bbbk)=G$, such that $(\widehat{T_L},G)=(\Z\Phi,\Aut(\Phi))$.
Note that these algebraic tori correspond to maximal tori of adjoint semisimple algebraic groups.
The construction of a Demazure model for this algebraic $\bk$-torus was described in ~\cite[p. 236]{KV84}.  We explain this construction in terms of the root polytope and the polar root polytope and determine the associated flasque resolution.

To find a Demazure model, we will take the first approach from Section~\ref{section:demazure}. 
We will first assume that $\Phi$ is irreducible and use the notation of the previous section.  $\Phi$ is a finite $G$-stable subset of $\Z\Phi$ which contains a basis of $\Z\Phi$.  The root polytope $P(\Phi)=\Conv(\Phi)$ contains 0 in its interior and is $G$-stable.  Then the toric variety $X_{P(\Phi)}$ corresponding to the normal fan of $P(\Phi)$ is a normal $G$-invariant projective toric $T_L$-model.
The normal fan of $P(\Phi)$ is the fan determined by cones over faces of the polar dual polytope $P(\Phi)^0$.

The polar dual polytope $P(\Phi)^0$ is called the polar root polytope.
By ~\cite{CM16}, the polar root polytope can be described as:
$$P(\Phi)^0=\cup_{w\in W(\Phi)}w\overline{\mathcal{A}_{\Phi}}$$
where 
$$\overline{\mathcal{A}_{\Phi}}=\{x\in (\R\Phi)^*: \langle \alpha,x\rangle \ge 0, \alpha\in \Pi, \langle \alpha_0,x\rangle\le 1\}$$
is the closure of the fundamental alcove $\mathcal{A}_{\Phi}$ of the affine Weyl group asssociated to $\Phi$.

The closure of the fundamental alcove
$$\overline{\mathcal{A}_{\Phi}}=\Conv(\0,\frac{\omega^{\vee}_i}{m_i}:i=1,\dots,n)$$
is an $n$-simplex~\cite[VI,\S 2.2, Corollary
 ]{Bou02}

The cones over the facets of this polytope are clearly the Weyl chambers of $\Phi$, so that the normal fan of $P(\Phi)$ consists of the Weyl chambers of $\Phi$ and all of its faces.

In particular, the projective toric variety $X_{P(\Phi)}$ is the same as $X(\Phi)$, the toric variety corresponding to the root system $\Phi$.
This is a smooth projective $\Aut(\Phi)$-invariant toric variety.

The maximal cones of the fan $\Sigma_{P(\Phi)}$ are $\Cone(w\mathcal{A}_{\Phi})=\Cone(w\omch_i:i=1,\dots,n)$ for $w\in W(\Phi)$.
So the rays of the fan are determined by the $W=W(\Phi)$-orbits of the fundamental coweights.  Let $D_{w\omch_i}$ be the divisor corresponding to $w\omch_i$ for $w\in W(\Phi),i=1,\dots,n$
and let $E_i(\Phi)=W\cdot \omch_i$ be the $W(\Phi)$-orbit of the $i$th coweight $\omch_i$. Note that the stabilizer subgroup of $\omch_i$ is $W_i=\langle s_{\alpha_j}: j\ne i\rangle$, the ith maximal parabolic subgroup.
Let $$\Div^i_{\Phi}=\oplus_{\lambda\in E_i(\Phi)}\Z D_{\lambda},i=1,\dots,n$$
Let $\Div_{\Phi}=\oplus_{i=1}^n\Div^i_{\Phi}$.
Note that $W(\Phi)$ stabilizes each $\Div^i_{\Phi}$.  But also $\Aut(\Phi)$ also acts on $\Div_{\Phi}$.  $\Aut(\Phi)$ is the semidirect product of $W(\Phi)$ by the diagram automorphism group $D(\Phi)$.  An element of $D(\Phi)$ permutes the simple roots and induces the same permutation on the fundamental coweights.
It then also permutes the $W(\Phi)$-lattices $\Div^i_{\Phi}$.  Explicitly, if $d\in D(\Phi)$ induces the permutation $\sigma_d$ on the simple roots, then $d(w\omch_i)=w^d\omch_{\sigma_d(i)}$ where $w^d=dwd^{-1}$.

This implies that the well-known divisor class sequence for $X(\Phi)$ (see, for example, \cite[(1.1.2)]{BrJo08} for a different derivation) is then invariant under the $\Aut(\Phi)$ action.  Since $X(\Phi)$ is smooth, we have that $\Cl(X(\Phi))=\Pic(X(\Phi))$.
We note that $\Div_T(X(\Phi))=\Div_{\Phi}$.
The divisor class sequence for $X(\Phi)$  considered as a short exact sequence of $\Aut(\Phi)$-lattices will then determine a $\Aut(\Phi)$-flasque resolution for $\Z\Phi$.

\begin{prop}\label{prop:toricrootflasque} A flasque resolution for $\Z\Phi$ as an $G\le \Aut(\Phi)$ lattice is given by

  $$\0\to \Z\Phi\stackrel{\divT}{\to} \Div_{\Phi} \to \Pic(X(\Phi))\to 0$$
  
  where $\Div_{\Phi}=\oplus_{j=1}^n\oplus_{\lambda\in E_j(\Phi)}\Z D_{\lambda}$,  and $\divT(\alpha_i)=\sum_{j=1}^n\sum_{\lambda\in E_j(\Phi)}\langle \alpha_i,\lambda\rangle \lambda$, for all $i=1,\dots,n$.
\end{prop}

\begin{cor}\label{cor:toricrootflasqueweyl} A flasque resolution for $\Z\Phi$ as a $W(\Phi)$-lattice can be determined as
  $$0\to \Z\Phi\to \oplus_{i=1}^n\Z[W(\Phi)/W_{\omega_i}]\to \Pic(X(\Phi))\to 0$$

  where $W_{\omega_i}=\langle s_{\alpha_j}:j\ne i\rangle$
  is the $i$th  maximal parabolic subgroup of $W(\Phi)$.
\end{cor}

\begin{rem}
  The $\Z$-dual map to $\divT:\Z\Phi\to \Div_{\Phi}$ has a particularly simple description.  It is the $\Aut(\Phi)$-map $\varphi$ determined by $\varphi(D_{\omega^{\vee}_i})=\omega^\vee_i, i=1,\dots,n$.
  Since the divisor class sequence determines a flasque resolution of $\Z\Phi$ with respect to any subgroup of $\Aut(\Phi)$,
  $$\0\to \Pic(X(\Phi))^0\to \Div_{\Phi}\stackrel{\varphi}{\to}\Lambda(\Phi^{\vee})\to 0$$ gives a coflasque resolution of $\Lambda(\Phi^{\vee})$ with respect to any subgroup of $\Aut(\Phi)$.
  Restricting to $W(\Phi)$, and identifying $\Div_{\Phi}$ with $\oplus_{i=1}^n\Z[W(\Phi)/W_i]$, we see that $\varphi(wW_i)=w\omch_i$, for $w\in W(\Phi)$ and $i=1,\dots,n$.
\end{rem}

\begin{rem}   Note that given a semisimple adjoint algebraic group $G$ with maximal torus $T$, the closure of $T$ in the wonderful compactification of $G$ gives a natural construction of the toric variety associated to the Weyl chambers of the root system associated to $G$ and $T$.  Although one can determine the closure of a maximal torus of an arbitrary semisimple algebraic group $G$ in its wonderful compactification, the resulting toric variety may not be smooth in general. 
\end{rem}

\begin{rem}
  The flasque resolutions of Proposition~\ref{prop:toricrootflasque} and Corollary~\ref{cor:toricrootflasqueweyl} also work in the case that $\Phi$ is a reducible crystallographic root system.
  Note that if $\Phi=\cup_{i=1}^s\Phi_i$ is a decomposition of $\Phi$ into a union of irreducible root systems, then $\Z\Phi=\oplus_{i=1}^s\Z\Phi_i$, $P(\Phi)=\prod_{i=1}^sP(\Phi_i)$, $P(\Phi)^0=\prod_{i=1}^sP(\Phi_i)^0$, $\Div_{\Phi}=\oplus_{i=1}^s\Div_{\Phi_i}$ and $X(\Phi)=\prod_{i=1}^sX(\Phi_i)$.
  Since each $X(\Phi_i)$ is smooth and projective, so is $X(\Phi)$.
  From our previous discussion of the relationship between the Weyl and automorphism groups of the irreducible components of $\Phi$ to those of $\Phi_i$, we see that $X(\Phi)$ is $W(\Phi)$ and $\Aut(\Phi)$-stable.
\end{rem}
  
  \begin{rem}
    Note that it is sufficient for our applications to consider simply-laced root systems.  For an irreducible non-simply laced root system, one can express the Weyl group as an automorphism group of a simply-laced root system.
Note that $(\Z B_n,\Aut(B_n))=((\Z A_1)^n,\Aut(A_1^n))$, $(\Z C_n,\Aut(C_n))=(\Z D_n,\Aut(D_n))$, $(\Z F_4,\Aut(F_4))=(\Z F_4,W(F_4))=(\Z D_4,\Aut(D_4))$ and \newline 
$(\Z G_2,\Aut(G_2))=(\Z G_2,W(G_2))=(\Z A_2,\Aut(A_2))$.

Also note: if $\Phi$ is an irreducible root system, which is not necessarily simply laced, the root polytope of $\Phi$ coincides with the root polytope of the subroot system of long roots in $\Phi$.  
  \end{rem}

  \begin{prop}
    Let $T$ be an algebraic $\bk$-torus split by a Galois extension $L/\bk$ such that \newline
  $(\widehat{T_L},\Gal(L/\Bbbk))$  corresponds to one of the maximal finite subgroups of $\GL(4,\Z)$: $(M_{G_i},G_i), i=1,\dots,9$.
    Then, unless $i=2,4,7$, there exists a (possibly) reducible root system $\Phi$ such that $(M_{G_i},G_i)=(\Z\Phi,\Aut(\Phi))$.  This implies that a Demazure model can be given as a product of the toric varieties of irreducible simply-laced root systems.
    More specifically, for the algebraic tori corresponding to
    $(\widehat{T_L},\Gal(L/\Bbbk))=(M_{G_i},G_i)$, we may take the following Demazure models:
    \begin{itemize}
\item $(M_{G_1},G_1)=(\Z\Phi,\Aut(\Phi))$ for $\Phi=A_2\times A_1^2$: $X(A_2)\times X(A_1)^2=X(A_2)\times (\P^1)^2$.
\item $(M_{G_3},G_3)=(\Z\Phi,\Aut(\Phi))$ for $\Phi=A_3\times A_1$: $X(A_3)\times X(A_1)=X(A_3)\times \P^1$.
\item $(M_{G_5},G_5)=(\Z\Phi,\Aut(\Phi))$ for $\Phi=A_2\times A_2$: $X(A_2)\times X(A_2)$.
\item $(M_{G_6},G_6)=(\Z\Phi,\Aut(\Phi))$ for $\Phi=A_4$: $X(A_4)$.
\item $(M_{G_8},G_8)=(\Z\Phi,\Aut(\Phi))$ for $\Phi=A_1^4$: $X(A_1)^4=(\P^1)^4$.
\item $(M_{G_9},G_9)=(\Z\Phi,\Aut(\Phi))$ for $\Phi=D_4$: $X(D_4)$.
    \end{itemize}
  \end{prop}
  
    \begin{rem}
We observe that the algebraic tori corresponding to both maximal subgroups of $\GL(2,\Z)$ and algebraic tori corresponding to 3 of the 4 maximal subgroups of $\GL(3,\Z)$ take the form $(\Z \Phi,\Aut(\Phi))$ for some root system and hence  the toric variety of $\Phi$ can be taken as a Demazure model.  We also observe that by Kunyavskii's results, we can determine a Demazure model for $(M_{G_2},G_2)$ as a product of his model for the algebraic torus corresponding to $(\Lambda(A_3),\Aut(A_3))$ with $X(A_1)=\P^1$.
We will discuss Kunyavskii's results further in Section~\ref{section:kunyavskii}. So we now need only determine Demazure models for the algebraic tori corresponding to $(\Lambda(A_4),\Aut(A_4))$ and to $(\Z A_2 \otimes \Z A_2,(W(A_2)\times W(A_2))\rtimes C_2\times C_2)$.
\end{rem}      

    \section{Algebraic tori corresponding to Weight lattices of type \texorpdfstring{$A_n$} {An}}

\subsection{Simplicial \texorpdfstring{$\Aut(A_n)$}{Aut(An)}-invariant projective fan}
    
  Let $T$ be an algebraic $\bk$-torus split by a Galois extension $L/\bk$ such that
  $(\widehat{T_L},\Gal(L/\Bbbk))=(\Lambda(A_n),\Aut(A_n))$.
  Note that these tori correspond to maximal tori for an simply connected algebraic $\bk$-group of type $A_n$.

Let $\Phi=A_n$.  
Note that $\Phi=\{\e_i-\e_j:1\le i\ne j\le n+1\}$ and $\Pi=\{\alpha_i=\e_i-\e_{i+1}: i=1,\dots,n\}$.
This root system is simply-laced, so that the weight lattice, and the coweight lattice coincide, as do the  root lattice, and coroot lattice.
The highest root in $A_n$ is $\theta=\sum_{i=1}^n\alpha_i=\e_1-\e_{n+1}$.
So for $A_n$, the coefficient of $\alpha_i$ for the highest root $\theta$ is $m_i=1$ for all $i=1,\dots,n$. Then, by our discussion in Section~\ref{section:toricroot},  the fundamental alcove
is then $\mathcal{A}_{A_n}=\Conv(\0,\omch_i:i=1,\dots,n)$, which shows that $P(A_n)^0$ is a lattice polytope.

So by the second approach to finding a normal $\Aut(A_n)$-invariant projective toric $T_L$-model  from Section~\ref{section:demazure}, we note that $\widehat{T_L}^0=\Lambda(A_n)^0=\Z A_n$ (since $A_n$ is simply-laced) and that $P(A_n)$ is an $\Aut(A_n)$-invariant polytope in $(\ThatLdual)_{\R}$ whose polar dual $P(A_n)^0$ is a lattice polytope. So $\Sigma_0$, the  normal fan of $P(A_n)^0$, or equivalently, the fan  of cones over faces of the root polytope $P(A_n)$ is a projective $\Aut(A_n)$-invariant fan in $\R A_n$ and $X_{\Sigma_0}$ is a normal $\Aut(A_n)$-projective $T_L$-toric model.

We apply the results of Section~\ref{section:rootpolytope} to $\Phi=A_n$ and $I=\{i\}$ to find the coordinate facets of $P(\Phi)$ and their stabilizer subgroups.
The affine Dynkin diagram for $A_n$ corresponds to an $(n+1)$-cycle on the vertices corresponding to  $\widehat{\Pi}=\{\alpha_i:i=0,\dots,n\}$.
Since removing any simple root $\alpha_i,i\in [n]$ from $\widehat{A_n}$ produces an irreducible subroot system, we see by Theorem~\ref{thm:facesrootpolytope} that all of coordinate faces $\{F_i:i=1,\dots,n\}$ are facets. In the special case of $\Phi=A_n$, and $I=\{i\}$, we see that $\partial I=I=\overline{I}=\{i\}$ since $\widehat{\Phi}_{[n]-\{i\}}$ is connected and $\alpha_i$ is not orthogonal to any non-equal adjacent root in $\widehat{\Pi}$.
From the results in Theorem~\ref{thm:facesrootpolytope},
$$\Stab_{W(A_n)}(F_i)=\{w\in W(A_n): w(F_i)=F_i\}=S_i\times S_{n+1-i}$$
whereas the pointwise stabilizer subgroup of this facet is trivial.

Under the action of $\Aut(A_n)$, a set of orbit representatives of the facets can be taken as $$\{F_i:i=1,\dots,\lceil n/2\rceil \}.$$   This shows that
$$\Stab_{\Aut(A_n)}(F_i)=\Stab_{W(A_n)}(F_i)=S_i\times S_{n+1-i}$$

The set of vertices of $F_i$ are
$V_i=A_n\cap F_i=\{\alpha\in A_n: \langle \alpha,\omch_i\rangle =1\}$. 
Since $\alpha\in \Phi$ implies that $\alpha=\sum_{i=1}^n\langle \alpha,\omch_i\rangle\alpha_i$, we see that the coefficient of $\alpha_i$ in $\alpha\in V_i$ must be 1.
This shows that the minimal element of $V_i$ is $\alpha_i$.
This shows that the set of vertices of $F_i$ may be described as:

$$V_i=\{\alpha\in A_n: \alpha\ge \alpha_i\}$$

Note that  $\alpha=\e_r-\e_s=\sum_{j=r}^{s-1}\alpha_{j}\in V_i$ if and only if $r\in [1,i] $ and $s\in [i+1,n+1]$.

For $I\subseteq [1,n+1]$, let $\Delta_I=\Conv(\e_j:j\in I)$.  Then $\Delta_I$ is an $(|I|-1)$-simplex.
$$\varphi_i: \Delta_{[1,i]}\times \Delta_{[i+1,n+1]}\to F_i: (\e_r,\e_{s})\to \e_r-\e_s , r\in [1,i], s\in [i+1,n+1]$$
is an isomorphism of polytopes which is equivariant with respect to the $S_i\times S_{n+1-i}$ action. 

From Theorem~\ref{thm:facesrootpolytope}, the $S_{n+1}$-orbit representatives of the faces of $P(A_n)$ are in bijection with the connected subsets of the affine Dynkin diagram of $A_n$ which contain $\alpha_0$.
More specifically, since the affine Dynkin diagram of $A_n$ is a cycle on vertices $\{0,\dots,n\}$, $F_I$ is an $S_{n+1}$-orbit representative for $I\subseteq [1,n+1]$, if and only if its complement in the $(n+1)$-cycle is connected if and only if $I$ is an interval $[i,j]\subseteq [1,n+1]$.
So such a coordinate face is then
$$F_{[i,j]}=\cap_{r=i}^jF_r=F_i\cap F_j=\Conv\{\e_r-\e_s: r\in [1,i], s\in [j+1,n+1]\}$$
Note that
$F_{[i,j]}\cong \Delta_{i-1}\times \Delta_{n-j}$.  If $i=j$, then the $i$th coordinate facet $F_i\cong \Delta_{i-1}\times \Delta_{n-i}$.

We note that this shows that the set of all faces of $P(A_n)$ is
$\{Q_{I,J}: \emptyset\ne I,J\subseteq [1,n+1], I\cap J=\emptyset\}$
where $Q_{I,J}=\Conv(\e_i-\e_j: i\in I,j\in J\}$. Then
$Q_{I,J}\cong \Delta_{|I|-1}\times \Delta_{|J|-1}$ has dimension $|I|+|J|-2$.
This is because $Q_{[1,i],[j+1,n+1]}=F_{[i,j]}$ and all other faces are obtained as $S_{n+1}$-translates.  We note that $\sigma(Q_{I,J})=Q_{\sigma(I),\sigma(J)}$ for all $\sigma\in S_{n+1}$ and the $S_{n+1}$-orbit of $F_{[i,j]}$ is then 
$$\{Q_{I,J}:I,J\subseteq [1,n+1], I\cap J=\emptyset, |I|=i,|J|=n+1-j\}$$

Note that $-\id$ maps $Q_{I,J}$ to $Q_{J,I}$ and so $\Aut(A_n)$ permutes all faces of $P(A_n)$.

Let 
$$s_{I,J}=\frac{|J|(\sum_{i\in I}\e_i)-|I|(\sum_{j\in J}\e_j)}{\gcd(|I|,|J|)}$$
Then $s_{I,J}$ is an interior lattice point in $\Cone(Q_{I,J})$ which is fixed by the action of $\Stab(Q_{I,J})$.
Note that the barycentre of $\Cone(Q_{I,J})$, the minimal generator of $\Cone(\sum_{i\in I,j\in J}\e_i-\e_j)\cap N$, is $s_{I,J}$.

We have that the normal fan of $P(A_n)^0$ is
$$\Sigma_{P(A_n)^0}=\{\Cone(Q_{I,J}): \emptyset \ne I,J\subseteq [1,n+1],I\cap J=\emptyset\}\cup \{\0\}$$
Observe that a flag of faces in $\Sigma_{P(A_n)^0}$ takes the form:
$$0\subseteq\Cone(Q_{I_1,J_1})\subseteq \dots \subseteq \Cone(Q_{I_n,J_n})$$
where
$$I_1\subseteq \dots \subseteq I_n\subseteq [1,n+1], J_1\subseteq \dots \subseteq J_n\subseteq [1,n+1], I_r\cap J_r=\emptyset, |I_r|+|J_r|=r+1$$

After performing a barycentric subdivision on $\Sigma_{P(A_n)^0}$, we obtain a simplicial fan  whose simplices are determined as
$\Cone(s_{I_1,J_1},\dots,s_{I_n,J_n})$
where
$$I_1\subseteq \dots \subseteq I_n\subseteq [1,n+1], J_1\subseteq \dots \subseteq J_n\subseteq [1,n+1], I_r\cap J_r=\emptyset, |I_r|+|J_r|=r+1$$

We claim that this new fan is a $\Aut(A_n)$-invariant projective simplicial fan on $\R A_n$ which is a subdivision of the original $\Aut(A_n)$-invariant projective fan $\Sigma_{P(A_n)^0}$.
The barycentric subdivision of a fan always determines a simplicial fan.
Our fan is clearly $\Aut(A_n)$-invariant by construction.  The barycentric subdivision of a fan is the result of a finite sequence of star subdivisions with respect to the barycentres of the cones in reverse order of dimension. (See, for example,~\cite[11.1.10]{CLS11}).  
A fan $\Sigma$ determines a projective toric variety $X_{\Sigma}$ if and only if $\Sigma$ is equipped with a strictly convex support function.  If a fan $\Sigma$ is equipped with a strictly convex support function, it is possible to adjust the support function to a strictly convex support function for $\Sigma^*(\v)$ for a star subdivision of $\Sigma$ with respect to $\v$, and hence also for the barycentric subdivision.  

In ~\cite{CTHS05}, they first construct a $G$-invariant projective fan on $N_{\R}$ given an arbitrary projective fan on $N_{\R}$ (Proposition 1).   In the second step (Proposition 2), they use their $G$-invariant projective fan on $N_{\R}$ to construct a $G$-invariant projective simplicial fan on $N_{\R}$ satisfying the additional property (*):
For any $\sigma\in \Sigma$, and $g\in G$, if $\sigma$ and $g\sigma$ are faces of a common cone $\tau\in \Sigma$ then $g\sigma=\sigma$.

To do this, for each $G$-orbit $\Omega$ of $\Sigma_0$, they choose a representative  $\sigma_{\Omega}$ in $\Sigma_0$ and  determine an element $x_{\sigma_{\Omega}}$ in the interior of $\sigma$ which is fixed by the stabilizer subgroup of $\sigma_{\Omega}$.
Then for any $\sigma\in \Sigma_0$ in the orbit of $\sigma_{\Omega}$, there exists $g\in G$ such that $\sigma=g\sigma_{\Omega}$.
They then set $x_{\sigma}=x_{g\sigma_{\Omega}}=gx_{\sigma_{\Omega}}$.  They show that this assignment is well-defined due to property (*).
Lastly they verify that the fan resulting from the barycentric subdivision of $\Sigma_0$ using $x_{\sigma}$ as the barycentres of each face $\sigma$ is 
$G$-invariant, simplicial and projective.

Note that the $\Aut(A_n)$-action on the the set of cones over faces 
$$\Sigma_{P(A_n)^0}=\{\Cone(Q_{I,J}): I,J\subseteq [1,n+1],I\cap J=\emptyset\}$$
has orbits
$$\Omega^r_i=\{\Cone(Q_{I,J}): |I|=i \mbox{ or }|J|=i, |I|+|J|=r+1\}$$
for each $1\le r\le n, 1\le i\le \lceil\frac{r}{2}\rceil$ where $\Cone(Q_{[1,i],[i+1,r+1]})$ is a representative of the orbit $\Omega^r_i$.  Note that if $g\in \Aut(A_n)$ maps $\Cone(Q_{I,J})$ to $\Cone(Q_{I',J'})$, then $g(s_{I,J})=s_{I',J'}$. 
So the construction of ~\cite{CTHS05}, starting with assigning $x_{\Omega^r_i}=s_{[1,i],[i+1,r+1]}$ for each $\Aut(\Phi)$-orbit of $\Sigma_{P(A_n)^0}$  would recover the same $\Aut(\Phi)$-invariant simplicial projective fan.

Orbit representatives of maximal cones in our new simplicial invariant projective fan $\Sigma$ with respect to $\Aut(A_n)$ are determined from each coordinate facet $F_i=\Conv(V_i)$ by fixing the minimal element $\alpha_i=\e_i-\e_{i+1}$ of $V_i$.
Then  $\Aut(A_n)$-orbit representatives of the maximal cones that result from subdividing $\Cone(F_i)$ are determined as
$\Cone(s_{I_1,J_1},\dots,s_{I_n,J_n})$
where $I_1=\{i\}$, $J_1=\{i+1\}$; $I_n=[1,i]$, $J_n=[i+1,n+1]$
and $I_1\subseteq \cdots \subseteq I_n, J_1\subseteq \cdots \subseteq J_n, |I_r|+|J_r|=r+1$.
Note that for each $r$, $I_{r}=[a_r,i]$ and $J_r=[i+1,b_r]$ for some $1\le a_r\le i$, $i+1\le b_r\le n+1$ such that $b_r-a_r=r$.

One can then visualize a  maximal cone coming from the subdivision of $\Cone(F_i)$ as a minimal staircase path from $(a_1,b_1)=(i,i+1)$ to $(a_n,b_n)=(1,n+1)$
on a rectangular grid with southwest corner $(i,i+1)$ and northeast corner $(1,n+1)$.
Note that  $(a_r,b_r)$ is either $(a_{r-1}-1,b_{r-1})$ (1 step north of $(a_{r-1},b_{r-1})$) 
or $(a_{r-1},b_{r-1}+1)$ (1 step east of $(a_{r-1},b_{r-1})$). Minimal staircase paths from $(i,i+1)$ to $(1,n+1)$ also correspond to minimal paths in $V_i=\Conv(\alpha\in A_n:\alpha\ge \alpha_i)$ between the minimal root $\alpha_i=\e_i-\e_{i+1}$ and the highest root $\theta=\e_1-\e_{n+1}$ with respect to the root order.

More explicitly, let the collection of minimal staircase paths  from $(a_1,b_1)=(i,i+1)$ to $(a_n,b_n)=(1,n+1)$ on the rectangular grid  with southwest corner $(i,i+1)$ and northeast corner $(1,n+1)$ be called $\MP(i)$.
Given a $C=((a_r,b_r):r=1,\dots,n)$ in $\MP(i)$, let $I_r=[a_r,i],J_r=[i+1,b_r],r=1,\dots,n$ and let $\gamma_C=\{s_{I_r,J_r}:r=1,\dots,n\}$.  Then the maximal cones contained in $\Cone(F_i)$ are
$\{\Cone(\gamma_C):C\in \MP(i)\}$.

We have proven the following:

\begin{prop}
  There exists a simplicial $\Aut(A_n)$-invariant projective fan $\Sigma^{\mbox{simp}}_{A_n}$ in $(\Z A_n)_{\R}$ with the following construction:
  The $\Aut(A_n)$-orbit representatives of the maximal cones are determined as
  $$\{\Cone(\gamma_{C}): C\in \MP(i),i=1,\dots,\left\lceil\frac{n}{2}\right\rceil\}$$
where $\MP(i)$ is the set of minimal staircase paths from $(i,i+1)$ to $(1,n+1)$
and $\gamma_C=\{s_{[a_r,i],[i+1,b_r]}:r=1,\dots,n\}$ for a minimal staircase path $C=\{[a_r,b_r]:r=1,\dots,n\}\in \MP(i)$.
\end{prop}

\begin{rem}~\label{rem:condition*}
  By construction, the $G=\Aut(\Phi)$ invariant fan $\Sigma= \Sigma^{\mathrm{simp}}_{A_n}$  satisfies the technical condition (*) of Proposition 2 of ~\cite{CTHS05}.
  (*) states that if  $\tau\in \Sigma$ and $g\in G$ are such that $\tau$  and $g\tau$ are both faces of the same cone $\sigma$ then $g\tau=\tau$.
  As they remark, this condition implies that the stabilizer of any cone coincides with its pointwise stabilizer.

  The condition is satisfied as in Prop 2 since cones in $\Sigma=\Sigma^{\mathrm{simp}}_{A_n}$ are determined by flags of cones in the initial $G$-invariant projective fan $\Sigma_0=\Sigma_{P(A_n)^0}$ and faces of such cones in $\Sigma$ are determined by subflags of cones of $\Sigma_0$.  If $\tau$ and $g\tau$ are faces of the same cone $\sigma$, then in the original $G$-invariant projective fan $\Sigma_0$, $\tau$ corresponds to a subflag of the flag of cones in $\Sigma_0$ determining the cone $\sigma$.  Since $g\tau$ corresponds to a subflag of cones of the same dimensions, and there is at most one cone of each dimension in the flag determining $\sigma$, $\tau$ and $g\tau$ must correspond to the same subflag of cones, and hence must coincide~\cite[Lemme 4,5]{CTHS05}.
\end{rem}

\subsection{Comparison to the unimodular triangulation of the root polytope}

Unfortunately, as we will see in Section~\ref{section:lowdim}, for $n\ge 4$, our simplicial $\Aut(A_n)$-invariant projective fan  $\Sigma^{\mathrm{simp}}_{A_n}$ is not smooth.
However, its construction reminds one strongly of the so-called ``staircase triangulation'' of $\Delta_{i-1}\times \Delta_{n-i}$. Note that unimodular triangulations of the root polytope for $A_n$ were determined independently using different techniques by Ardila et al and Cellini, Marietta~\cite{ABHPS11,CM15}.  (It is also related to the unimodular triangulations of the polytopes determined by the positive roots of $A_n$ as in~\cite{Mes11,GMP82}).  Cellini and Marietti take advantage of the symmetry of the $W(A_n)$-action, by constructing a triangulation of the coordinate facets and then extending it to the other facets using the $W(A_n)$- group action.  However, their unimodular triangulation is not invariant under the action of $W(A_n)$ (or $\Aut(A_n)$).
The problem is that although the triangulation is defined on Weyl group representatives of the facets, the triangulation of each $F_i$ is not equivariant under the stabilizer group of $F_i$.

Despite this issue, there is still a strong connection between the unimodular triangulation and the simplicial fan that we constructed.  
The staircase triangulation on $\Delta_{i-1}\times \Delta_{n-i}$ together with its isomorphism with $F_i=\Conv\{\alpha\in A_n:\alpha\ge \alpha_i\}$ induces a unimodular triangulation of $F_i$.
For each $C=\{(a_r,b_r):r=1,\dots,n\}$ in $\MP(i)$ (minimal staircase path from $(i,i+1)$ to $(1,n+1)$),  
we obtain a polytope $$P_C=\Conv(\e_{a_r}-\e_{b_r}:r=1,\dots,n)$$
and a corresponding $\Z$-basis $\beta_C=\{\beta_r=\e_{a_r}-\e_{b_r}:r=1,\dots,n\}$ of $\Z A_n$.
The corresponding unimodular triangulation of $F_i$ is then $\{P_C=\Conv(\beta_C):C\in \MP(i)\}$.

We wish to compare the sets $\beta_C$ and $\gamma_C$ for each $C\in \MP(i)$.  This comparison will help us to find a smooth $\Aut(A_4)$-invariant projective fan on $(\Z A_n)_{\R}$.
We will first recall how one can determine a smooth projective fan from a simplicial projective fan. 

In general, if a  simplicial projective fan $\Sigma$ is not smooth, it has at least one  simplicial but nonsmooth cone $\sigma$. We recall that the multiplicity of the cone $\sigma$ with linearly independent primitive generators $\x_1,\dots,\x_r$ is  $\mult(\sigma)=[\Span_{\R}(\sigma)\cap N:\sum_{i=1}^r\Z\x_i]$, the index of the subgroup generated by the primitive vectors of $\sigma$ in $N$.  This is the same as the number of lattice points in the parallelotope $P_{\sigma}=\{\sum_{i=1}^r\lambda_i\x_i:0\le \lambda<1\}$.  Recall that $\mult(\sigma)=1$ if and only if $\sigma$ is smooth if and only if the primitive generators of $\sigma$ form part of a basis of $N$.

  In order to subdivide a simplicial but nonsmooth cone $\sigma$ into a union of smooth cones, one finds its multiplicity, and if its multiplicity is greater than 1, one finds an interior lattice point $\v$ of $P_{\sigma}$.  Then each cone of the star-subdivision of $\sigma$ with respect to $\v$ has strictly smaller multiplicity than $\sigma$.  One can replace $\Sigma$ by $\Sigma^*(\v)$.  $\Sigma^*(\v)$ is then a projective fan, whose cones have smaller multiplicity than those of $\Sigma$.  This process can be repeated inductively to produce a smooth projective fan. See~\cite[11.1.8,11.1.9]{CLS11} or ~\cite[I,\S 2,Lemma 2,p.34]{KKMSD73} for more details.  

  We remark that if the original simplicial projective fan $\Sigma$ is additionally $G$-invariant and satisfies technical condition (*) of Proposition 2 of~\cite{CTHS05} (see Remark~\ref{rem:condition*}), it is shown in the proof of ~\cite[Prop 3]{CTHS05},
  that if $\sigma$ is  a simplicial but non-smooth cone, and $\v$ is an interior point in $P_{\sigma}$, then the star-subdivision of $\Sigma$ with respect to each vector in the $G$-orbit of $\v$ produces a $G$-invariant projective fan.  Each cone $\tau$ containing $\v$ and each $g\in G$ produces a new cone $\Cone(g\tau,g\v)$ of $\Sigma^*(g\v)$ with smaller multiplicity than that of $\sigma$.  This process can be repeated inductively to produce a $G$-invariant  projective fan, all of whose cones have multiplicity 1.

We now return to our specific situation and the $\Aut(A_n)$-invariant simplicial projective fan with maximal cones $\Cone(\gamma_C)$ for each $C\in \MP(i)$ and each $i=1,\dots,n$.  
Given such a $C=\{(a_r,b_r):r=1,\dots,n\}$ in $\MP(i)$, we consider the partial paths $C_t=\{(a_r,b_r):r=1,\dots,t\}$.
We will show that the $\beta_{C_t}$,$t=1,\dots,n$ can be used to inductively compute the multiplicity of $\Cone(\gamma_{C_t}),t=1,\dots,n$. Note that $\{\Cone(\gamma_{C_t}):t=1,\dots,n\}$ is the set of positive dimensional faces of $\Cone(\gamma_C)$.

\begin{lem}
   Let $C=\{(a_r,b_r):r=1,\dots,n\}\in \MP(i)$ be a minimal staircase path from $(i,i+1)$ to $(1,n+1)$.   Let $C_t=\{(a_r,b_r):i=1,\dots,t\}$ be a partial staircase path for some $1\le t\le n$.
  Let $\beta_{C_t}=\{\e_{a_r}-\e_{b_r}:r=1,\dots,t\}$ and $\gamma_{C_t}=\{s_{[a_r,i],[i+1,b_r]}:r=1,\dots,t\}$.
  Note that $C_n=C$ and $\Conv(\beta_C)$ is the polytope which corresponds to $C$ in the unimodular triangulation of $F_i$ and $\Cone(\gamma_C)$ is the maximal cone which corresponds to $C$ in the simplicial subdivision of $\Cone(F_i)$.
  
  Then for all $1\le t\le n$,
  $\Span_{\Z}(\gamma_{C_t})$ is a finite index sublattice of $$\Span_{\Z}(\beta_{C_t})=\Span_{Z}\{\e_a-\e_b:a\in [a_t,i],b\in [i+1,b_t]\}$$
  So $\mult(\Cone(\gamma_{C}))=[\Span_{\Z}(\beta_C):\Span_{\Z}(\gamma_C)]$ can be computed inductively from $$\mult(\Cone(\gamma_{C_t}))=[\Span_{\Z}(\beta_{C_t}):\Span_{\Z}(\gamma_{C_t})],t=1,\dots,n.$$
\end{lem}

\begin{proof}
We first observe that
$$V^i_r\equiv \Span_{\Z}\{e_a-e_b:a\in [a_r,i],b\in [i+1,b_r]\}$$
is the $\Z$-span of the positive roots corresponding to the rectangular array $[a_r,i]\times [i+1,b_r]$.   In fact, $\{\alpha_t: t\in [a_r,b_r]\}$ is a $\Z$-basis for $V^i_r$ since $\e_a-\e_b=\sum_{t=a}^b\alpha_t$ if $a\in [a_r,i],b\in [i+1,b_r]$.  
In particular, $V^i_n=\Z A_n$.  We will set $V^i_0=\{\0\}$.

  We first claim that if $C_t=\{(a_r,b_r):i=1,\dots,t\}$ is the partial staircase path from  $(a_1,b_1)=(i,i+1)$ to $(a_t,b_t)$, then $\beta_{C_t}$ is a $\Z$-basis for $V^i_t$.
In particular, this observation reproves the fact (from the unimodular triangulation) that $\beta_C$ is a $\Z$-basis of $\Z A_n$.

We prove this claim by induction on $r$.
$\beta_{C_1}=\{\alpha_i\}$ and so is a basis for $V^i_1$.  Assume $\beta_{C_r}$ is a basis for $V^i_{r}$.

It is clear by construction that $\beta_{C_{r+1}}\subseteq \{\e_a-\e_b: a\in [a_{r+1},i],b\in [i+1,b_{r+1}]\}$ so that $\Span_{\Z}(\beta_{C_{r+1}})\subseteq V^i_{r+1}$.

To show the reverse inclusion, we need to show that $\alpha_t\in \Span_{\Z}(\beta_{C_{r+1}})$ for $t\in [a_{r+1},b_{r+1}]$ since $\{\alpha_t:t\in [a_{r+1},b_{r+1}]\}$ is a $\Z$-basis of $V^i_{r+1}$.
By the staircase condition, 
$(a_{r+1},b_{r+1})=(a_r-1,b_r)$ or $(a_{r+1},b_{r+1})=(a_r,b_r+1)$.
So for $t\in [a_r,b_r]$, $\alpha_t\in V^i_r=\Span_{\Z}(\beta_{C_r})\subseteq \Span_{\Z}(\beta_{C_{r+1}})$.

If $(a_{r+1},b_{r+1})=(a_r-1,b_r)$, then $\alpha_{a_{r+1}}=(\e_{a_{r+1}}-\e_{b_r})-(\e_{a_r}-\e_{b_r})\in \Span_{\Z}(\beta_{C_{r+1}})$.

If $(a_{r+1},b_{r+1})=(a_r,b_r+1)$, then $\alpha_{b_{r+1}}=(\e_{a_{r}}-\e_{b_{r+1}})-(\e_{a_r}-\e_{b_r})\in \Span_{\Z}(\beta_{C_{r+1}})$.

(Either way, the additional simple root is the difference of the last two vectors in $\beta_{C_{r+1}}$.)

So $\beta_{C_t}$ is a $\Z$-basis for $V^i_{t}$ for all $t=1,\dots,n$ and, in particular, $\beta_C$ is a basis for $V^i_{n}=\Z A_n$.
Note that for $s<t$, $\beta_{C_s}\subset \beta_{C_t}$ and $\gamma_{C_s}\subset \gamma_{C_t}$. We set $\beta_{C_0}=\emptyset$.
We will prove by induction on $n$ that for each $t=1,\dots,n$, $\gamma_{C_t}\subseteq V^i_{t}\setminus V^i_{t-1}$.

Observe that $\gcd(|I|,|J|)s_{I,J}=\sum_{a\in I,b\in J}\e_a-\e_b$.

Since $\gamma_{C_1}=\beta_{C_1}=\{\e_i-\e_{i+1}\}$, and $V^i_0=\{\0\}$,  the result is trivial for $t=1$.

Assume $t>1$ and that the result holds for $\gamma_{C_{t-1}}$.  So $\gamma_{C_{t-1}}\subseteq V^i_{t-1}\subset V^i_{t}$.
Now $\gamma_{C_t}=\gamma_{C_{t-1}}\cup \{s_{[a_t,i],[i+1,b_t]}\}$.
By the observation above,  a positive integer multiple of
$s_{[a_t,i],[i+1,b_t]}\in V^i_t$, but $s_{[a_t,i],[i+1,b_t]}$ is not in  $V^i_{t-1}$.  So $s_{[a_t,i],[i+1,b_t]}\in (\Z A_n\cap \Q V^i_t)\setminus V^i_{t-1}$. 
Since a $\Z$-basis of $V^i_t$ is $\{\alpha_j: j\in [a_t,b_t]\}$ and this $\Z$-basis can be extended to a $\Z$-basis of $\Z A_n$, then $\Z A_n\cap \Q V^i_t=V^i_t$.

So, by induction, for each $t=1,\dots,n$, $\Span_{\Z}(\gamma_{C_t})$ is a finite index sublattice of $V^i_t=\Span_{\Z}(\beta_{C_t})$ and $\mult(\Cone(\gamma_{C_t})=[\Span_{\Z}(\beta_{C_t}):\Span_{\Z}(\gamma_{C_t})]$.
In particular, $\Span_{\Q}(\gamma_{C_t})=\Span_{\Q}(\beta_{C_t})$ so that $\Span_{\R}(\gamma_{C_t})\cap \Span_{\Z}(\beta_{C_t})=\Span_{\Z}(\beta_{C_t})$.  So $\mult(\Cone(\gamma_{C_t}))=[\Z\beta_{C_t}:\Z\gamma_{C_t}]$ as required.
\end{proof}

\begin{rem}
  There is another relatively natural way to determine the simplicial $\Aut(A_n)$-invariant fan from the unimodular triangulation of $P(A_n)$.
  As already observed, the unimodular triangulation of $P(A_n)$ is determined by the natural isomorphism of each coordinate facet $F_i$ with a product of simplices, and then the staircase triangulation of the product of simplices.
Although the staircase triangulation of
$\Delta_{i-1}\times \Delta_{n-i}$ is not $S_{i}\times S_{n+1-i}$-invariant, it is possible to first decompose the product of simplices into a union of product of simplices indexed by the elements of $S_{i}\times S_{n+1-i}$.

Let $I=\{i_1,\dots,i_r\}$ be an ordered subset of $[1,n+1]$ and $I_s=\{i_1,\dots,i_s\}$ for all $1\le s\le r=|I|$.
Let $\Delta_I=\Conv(\e_i:i\in I)$ be a standard simplex.  Let $\v_J=\frac{1}{|J|}\sum_{j\in J}\e_j$ for each subset $J$ of $[1,n+1]$.
The barycentric subdivision of $\Delta_I$ can be viewed as a decomposition into $|I|!$-simplices in $\oplus_{i\in I}\Q\e_i$, one for each element of $S_{I}$:
$$\Delta_I=\cup_{\sigma\in S_I}\Delta_I(\sigma)=\cup_{\sigma\in S_I}\sigma(\Delta_I(\id))$$
where here $$\Delta_I(\sigma)=\Conv(\v_{\sigma(I_j)}:j=1,\dots,|I|)$$ 

In fact, as is well known, and can be easily checked:
$$\Delta_I(\sigma)=\left\{\sum_{j=1}^ra_j\e_{i_j}: a_{\sigma(1)}\ge a_{\sigma(2)}\ge \cdots\ge a_{\sigma(r)},0\le \sum_{j=1}^ra_j\le 1\right\}$$
Note that this description shows that the subgroup of $S_{I}$ which stabilizes $\Delta_I(\sigma)$ coincides with the subgroup of $S_{I}$ which stabilizes it pointwise.

One can then decompose $\Delta_I\times \Delta_J$ for disjoint subsets $I,J\subseteq [1,n+1]$ as
$$\Delta_I\times \Delta_J=\cup_{(\sigma,\tau)\in S_I\times S_J}(\sigma,\tau)(\Delta_I(\id)\times \Delta(\id))$$

For disjoint subsets $I,J\subseteq [1,n+1]$, let
$$\delta_{I,J}:\oplus_{i\in I}\R\e_i\times \oplus_{j\in J}\R\e_j\to\oplus_{i\in I,j\in J}\R(\e_i-\e_j), (\v,\w)\to \v-\w$$
be the natural difference map.
Note that $\v_I-\v_J=\frac{1}{\lcm(|I|,|J|)}s_{I,J}$ if $I,J$ are disjoint sets and so $s_{I,J}$ is the primitive generator of $\Cone(\v_I-\v_J)$.

Note that $\delta_{[1,i],[i+1,n+1]}$ maps $\Delta_{[1,i]}\times \Delta_{[i+1,n+1]}$ isomorphically onto the $i$th coordinate facet.

One can then apply the map $\delta_{[1,i],[i+1,n+1]}$ to the decomposition of $\Delta_{[1,i]}\times \Delta_{[i+1,n+1]}$ to equivariantly subdivide $F_i$ with respect to the $S_{[1,i]}\times S_{[i+1,n+1]}$ action:

$$F_i=\cup_{\sigma\in S_{[1,i]},\tau\in S_{[i+1,n+1]}}F_i(\sigma,\tau)=
  \cup_{\sigma\in S_{[1,i]},\tau\in S_{[i+1,n+1]}}(\sigma,\tau)(F_i(\id,\id))$$

    Ordering $[1,i]$ in decreasing order and $[i+1,n+1]$ in increasing order, we can apply the staircase  triangulation to $\Delta_{[1,i]}(\id)\times \Delta{[i+1,n+1]}(\id)$ and hence to its image $F_i(\id,\id)$ under the difference map.
    For each minimal staircase path $C=\{(r_i,s_i):i=1,\dots,n\}$ in $\MP(i)$,
    $$F_i(\id,\id)=\cup_{C\in \MP(i)}\Conv(\widehat{\beta}_C)$$
    where $\widehat{\beta}_C=\{\v_{[a_r,i]}-\v_{[i+1,b_r]}:r=1,\dots,n\}$.
   Looking at the induced decomposition on cones, we see that
   $\Cone(\widehat{\beta_C})=\Cone(\gamma_C)$
   since $\Cone(\v_I-\v_J)=\Cone(s_{I,J})$ for $I,J$ disjoint subsets of $[1,n+1]$.  
Since $\Cone(F_i(\sigma,\tau))=(\sigma,\tau)\Cone(F_i(\id,\id))$, for all $(\sigma,\tau)\in S_{[1,i]}\times S_{[i+1,n+1]}$, one can use the group action to determine the rest of the decomposition.
This then recovers our simplicial $\Aut(A_n)$-invariant fan.
\end{rem}    
 
\subsection{Demazure model for the algebraic torus corresponding to \texorpdfstring{$(\Lambda(A_4),\Aut(A_4))$}{the weight lattice of the A4 root system under the action of its automorphism group}}.
\label{section:lowdim}

We now look at the low-dimensional cases to determine a smooth $\Aut(A_n)$-model of the algebraic torus corresponding to the character lattice $\Lambda(A_n)$ as a $\Aut(A_n)$-lattice for $n\le 4$.

We recall that the maximal cones of our normal $\Aut(A_n)$-invariant projective models $\Sigma_{P(A_n)^0}$ have $\Aut(A_n)$-orbits $\Cone(F_i),i=1,\dots,\lceil{n/2\rceil}$ corresponding to the coordinate facets $F_i=\Conv(\alpha\in A_n:\alpha\ge \alpha_i)$.

  We note first that for all $n$, $\Cone(F_1)=\Cone(\alpha\in A_n: \alpha\ge \alpha_1)$ corresponding to the cone over the coordinate facet $F_1$ is always smooth.
  It corresponds to the fact that there is a unique minimal path $C$ from $(1,2)$ to $(1,n+1)$ and that $\beta_{C}=\{\beta_r=\e_1-\e_r:r=2,\dots,n+1\}$ is clearly already a  $\Z$-basis for $\Z A_n$.  Correspondingly $\gamma_C=\{\gamma_r=r\e_1-\sum_{i=2}^{r+1}\e_i:r=1,\dots,n\}$ is a $\Z$-basis for $\Z A_n$ as $\gamma_r=\sum_{i=1}^r\beta_i$.
So in the simplicial model there is only one $\Aut(A_n)$-orbit of maximal cones from $\Cone(F_1)$.  
  
  For $n=2$, the fan of the original normal projective model $\Sigma_{P(A_2)^0}$ had only one $\Aut(A_2)$-orbit of maximal cones with representative $\Cone(F_1)=\Cone(\e_1-\e_2,\e_1-\e_3)$ and stabilizer subgroup $S_2$.  This is already a smooth cone, so there was actually no need for further subdivision in this case.
  However, it is still interesting to look at the result of the barycentric subdivision.  One obtains one $\Aut(A_2)$-orbit of maximal cones with representative $\Cone(\e_1-\e_2,2\e_1-\e_2-\e_3)$.  Note that the primitive generators form a basis of the root system $G_2$.

  For $n=3$, the fan of the original normal projective model had 2 $\Aut(A_2)$-orbits of maximal cones with representatives
\begin{itemize}
\item  $\Cone(F_1)=\Cone(\e_1-\e_2,\e_1-\e_3,\e_1-\e_4)$
 \item $\Cone(F_2)=\Cone(\e_2-\e_3,\e_2-\e_4,\e_1-\e_3,\e_1-\e_4)$
\end{itemize}
   corresponding to first two 2 coordinate facets
  $F_i=\Conv(\alpha\in A_3: \alpha\ge \alpha_i),i=1,2$.
  There are 6 cones in the orbit of $\Cone(F_1)$:
  $\pm\Cone(\e_i-\e_j: j\ne i)$, $i=1,2,3$.
  As before $\Cone(F_1)$ is smooth (as was predicted by $F_1\cong \Delta_2$) but $\Cone(F_2)$ is not simplicial.
  However a barycentric subdivision of $\Cone(F_2)$ with respect to $\v=(\e_1+\e_2)-(\e_3+\e_4)$, subdivides $\Cone(F_2)$ into 4 smooth cones:
  \begin{itemize}
  \item $\Cone(\e_i-\e_3,\e_i-\e_4,\e_1+\e_2-\e_3-\e_4), i=1,2$
  \item $\Cone(\e_1-\e_j,\e_2-\e_j,\e_1+\e_2-\e_3-\e_4), j=3,4$.
\end{itemize}
    This corresponds to Kunyavskii's construction - as will be discussed later in Section~\ref{section:kunyavskii}.

Our simplicial model is a bit larger, and is also smooth.
We have 3 orbits of maximal cones in our simplicial fan  with representatives:

From $\Cone(F_1)$:
$$\Cone(\e_1-\e_2,2\e_1-\e_2-\e_3,3\e_1-\e_2-\e_3-\e_4)$$
From $\Cone(F_2)$:
\begin{itemize}
\item $\Cone(\gamma_{C_1})=\Cone(\e_2-\e_3,2\e_2-\e_3-\e_4,\e_1+\e_2-\e_3-\e_4)$
\item $\Cone(\gamma_{C_2})=\Cone(\e_2-\e_3,\e_1+\e_2-2\e_3,\e_1+\e_2-\e_3-\e_4)$
\end{itemize}
  The last two correspond to the 2 minimal staircase paths $C_1=\{(2,3),(2,4),(1,4)\}$ and $C_2=\{(2,3),(1,3),(1,4)\}$ from $(2,3)$ to $(1,4)$.
It is easy to check directly that these are all smooth.
In terms of the unimodular triangulation, $\Span_{\Z}(\gamma_C)=\Span_{\Z}(\beta_C)$ for each minimal path $C$ from $(i,i+1)$ to $(1,4)$, $i=1,2$.

For $n=4$, the simplicial model is not smooth, but is quite close to being so.
There are again 2 orbits of facets represented by $F_1=\Conv(\alpha\ge \alpha_1)$ and $F_2=\Conv(\alpha\ge \alpha_2)$.
As was already pointed out for general $n$, $\Cone(F_1)$ determines one  $\Aut(A_4)$-orbit of smooth cones.

From $\Cone(F_2)$ there are 3 $\Aut(A_4)$-orbits from the 3 minimal paths from $(2,3)$ to $(1,5)$.
The path $C_1=\{(2,3),(2,4),(2,5),(1,5)\}$ gives the
orbit representative:
$\sigma_{C_1}=\Cone(\gamma_{C_1})$
where
$$\gamma_{C_1}=\{\e_2-\e_3,2\e_2-\e_3-\e_4,3\e_2-\e_3-\e_4-\e_5,3(\e_1+\e_2)-2(\e_3+\e_4+\e_5)\}$$
The corresponding polytope from the unimodular triangulation of $F_2$ is 
$P_{C_1}=\Conv(\beta_{C_1})$,
where
$$\beta_{C_1}=\{\e_2-\e_3,\e_2-\e_4,\e_2-\e_5,\e_1-\e_5\}$$
From this comparison and from those of the subpaths (or directly), we see that
$\Z\gamma_{(C_1)[3]}=\Z\beta_{(C_1)[3]}$.
Numbering the primitive generators in $\gamma_{C_1}$ as $\gamma_i,i=1,\dots,4$ and those in $\beta_{C_1}$ as $\beta_i,i=1,\dots,4$, we see that
$$\gamma_1=\beta_1,\gamma_2=\beta_1+\beta_2,\gamma_3=\beta_1+\beta_2+\beta_3,\gamma_4=2\beta_1+2\beta_2-\beta_3+3\beta_4$$
Then $\Z\gamma_{C_1}=\sum_{i=1}^4\Z\gamma_i=\sum_{i=1}^3\Z\beta_i+\Z\gamma_4=\sum_{i=1}^3\Z\beta_i+3\beta_4$.  We conclude that $\mult(\sigma_{C_1})=3$.
Looking at the faces of $\sigma_{C_1}$, we easily see that all those contained in the cone $\Cone(\gamma_1,\gamma_2,\gamma_3)$ corresponding to the subpath $(C_1)[3]=\{(2,3),(2,4),(2,5)\}$  are smooth.  It is not hard to see that $\Cone(\gamma_i,\gamma_4),i=1,2$ are also smooth.
So the only non-smooth faces of $\sigma_{C_1}$ are those containing $\Cone(\gamma_3,\gamma_4)$.

We remark that finding the interior lattice points of $P_{\sigma_{C_1}}=\{\sum_{i=1}^4\lambda_i \gamma_i:\lambda_i\in [0,1)\}$ is an easy calculation that can be done by hand using the $\Z$-basis $\beta_{C_1}$ for $N$, due to the fact (shown earlier in general) that the change of basis matrix from $\gamma_{C_1}$ to $\beta_{C_1}$ is triangular.  
  The 3 lattice points are $\frac{r}{3}(\gamma_3+\gamma_4)=r(\beta_1+\beta_2+\beta_4)=r(\e_1+2\e_2-\e_3-\e_4-\e_5),r=0,1,2$.
  We can do the star-subdivision corresponding to the non-trivial interior lattice point $\v_1=\e_1+2\e_2-\e_3-\e_4-\e_5$.
  From the above comments, it is clear that this replaces the maximal cone $\sigma_{C_1}$ with 2 cones:
  $\Cone(\gamma_1,\gamma_2,\gamma_3,\v_1)$ and $\Cone(\gamma_1,\gamma_2,\v_1,\gamma_4)$.

  The path $C_2=\{(2,3),(1,3),(1,4),(1,5)\}$ gives the
orbit representative:
$\sigma_{C_2}=\Cone(\gamma_{C_2})$
where
$$\gamma_{C_2}=\{\e_2-\e_3,\e_1+\e_2-2\e_3,\e_1+\e_2-\e_3-\e_4,3(\e_1+\e_2)-2(\e_3+\e_4+\e_5)\}$$
The corresponding polytope from the unimodular triangulation of $F_2$ is 
$P_{C_2}=\Conv(\beta_{C_2})$,
where
$$\beta_{C_2}=\{\e_2-\e_3,\e_1-\e_3,\e_1-\e_4,\e_1-\e_5\}$$
We again  see that the cone corresponding to the subpath $(C_2)[3]=\{(2,3),(1,3),(1,4)\}$ is smooth so that 
$\Z\gamma_{(C_2)[3]}=\Z\beta_{(C_2)[3]}$.
Numbering the primitive generators in $\gamma_{C_2}$ as $\gamma_i,i=1,\dots,4$ and those in $\beta_{C_2}$ as $\beta_i,i=1,\dots,4$, we see that
$$\gamma_1=\beta_1,\gamma_2=\beta_1+\beta_2,\gamma_3=\beta_1+\beta_3,\gamma_4=3\beta_1-\beta_2+2\beta_3+2\beta_4$$

Then $\Z\gamma_{C_2}=\sum_{i=1}^4\Z\gamma_i=\sum_{i=1}^3\Z\beta_i+\Z\gamma_4=\sum_{i=1}^3\Z\beta_i+2\beta_4$.  We conclude that $\mult(\sigma_{C_2})=2$.
Looking at the faces of $\sigma_{C_2}$, we easily see that all those contained in $\sigma_{(C_2)[3]}$ are smooth.  It is not hard to see that $\Cone(\gamma_i,\gamma_4),i=1,3$ are also smooth.
So the only non-smooth faces of $\sigma_{C_2}$ are those containing $\Cone(\gamma_2,\gamma_4)$.

  The 2 interior lattice points of $P_{\sigma_{C_2}}$ are $\frac{r}{2}(\gamma_2+\gamma_4)=r(2\beta_1+\beta_3+\beta_4)=r(2\e_1+2\e_2-2\e_3-\e_4-\e_5),r=0,1$.
  We can do the star-subdivision corresponding to the non-trivial interior lattice point $\v_2=2\e_1+2\e_2-2\e_3-\e_4-\e_5$.
  From the above comments, it is clear that this replaces the maximal cone $\sigma_{C_2}$ with 2 smooth maximal cones:
  $\Cone(\gamma_1,\gamma_2,\gamma_3,\v_2)$ and $\Cone(\gamma_1,\v_2,\gamma_3,\gamma_4)$.

    The path $C_3=\{(2,3),(2,4),(1,4),(1,5)\}$ gives the
orbit representative:
$\sigma_{C_3}=\Cone(\gamma_{C_3})$
where
$$\gamma_{C_3}=\{\e_2-\e_3,2\e_2-\e_3-\e_4,\e_1+\e_2-\e_3-\e_4,3(\e_1+\e_2)-2(\e_3+\e_4+\e_5)\}$$
The corresponding polytope from the unimodular triangulation of $F_2$ is 
$P_{C_3}=\Conv(\beta_{C_3})$,
where
$$\beta_{C_3}=\{\e_2-\e_3,\e_2-\e_4,\e_1-\e_4,\e_1-\e_5\}$$
We again  see that the cone corresponding to the subpath $(C_3)[3]=\{(2,3),(2,4),(1,4)\}$ is smooth so that 
$\Z\gamma_{{C_3}[3]}=\Z\beta_{{C_3}[3]}$.
Numbering the primitive generators in $\gamma_{C_2}$ as $\gamma_i,i=1,\dots,4$ and those in $\beta_{C_2}$ as $\beta_i,i=1,\dots,4$, we see that
$$\gamma_1=\beta_1,\gamma_2=\beta_1+\beta_2,\gamma_3=\beta_1+\beta_3,\gamma_4=2\beta_1+\beta_2+\beta_3+2\beta_4$$

Then $\Z\gamma_{C_3}=\sum_{i=1}^4\Z\gamma_i=\sum_{i=1}^3\Z\beta_i+\Z\gamma_4=\sum_{i=1}^3\Z\beta_i+2\beta_4$.  We conclude that $\mult(\sigma_{C_3})=2$.
Looking at the faces of $\sigma_{C_3}$, we easily see that all those contained in $\sigma_{{C_3}[3]}$ are smooth.  It is not hard to see that $\Cone(\gamma_i,\gamma_4),i=1,2,3$ are also smooth.
So the only non-smooth faces of $\sigma_{C_3}$ are those containing $\Cone(\gamma_2,\gamma_3,\gamma_4)$.

  The 2 interior lattice points of $P_{\sigma_{C_3}}$ are $\frac{r}{2}(\gamma_2+\gamma_3+\gamma_4)=r(2\beta_1+\beta_2+\beta_3+\beta_4)=r(2\e_1+3\e_2-2\e_3-2\e_4-\e_5),r=0,1$.
  We can do the star-subdivision corresponding to the non-trivial interior lattice point $\v_3=2\e_1+3\e_2-2\e_3-2\e_4-\e_5$.
  From the above comments, it is clear that this replaces the maximal cone $\sigma_{C_3}$ with 3 smooth maximal cones:
  $\Cone(\gamma_1,\gamma_2,\gamma_3,\v_3)$, $\Cone(\gamma_1,\gamma_2,\v_3,\gamma_4)$,
$\Cone(\gamma_1,\v_3,\gamma_3,\gamma_4)$.

  In summary, we have obtained a smooth $\Aut(A_4)$-invariant projective model for the algebraic torus corresponding to the character lattice $(\Lambda(A_4),\Aut(A_4))$.  This model has 8 $\Aut(A_4)$-orbits of (smooth) maximal cones.
  The cone over the standard parabolic facet $F_1$ of the root polytope produces a unique $\Aut(A_4)$ orbit of smooth maximal cones.
  The cone over the standard parabolic facet $F_2$ of the root polytope  produces 3 $\Aut(A_4)$-orbits of simplicial maximal cones.  These in turn are split into 2,2,3 smooth maximal cones, respectively.

  We can determine all the $\Aut(A_n)$-orbits of cones in our fan.
  As before the rays of our simplicial fan were determined as
  $$\{s_{I,J}: \emptyset\ne I,J\subseteq [1,n+1], I\cap J=\emptyset\}$$
  For the $\Aut(A_n)$ action, we can set $s_{r,t}=s_{[1,r],[r+1,t+l]}$ and note that
  $$\{s_{r,t}: 1\le r\le t, r+t\le n+1\}$$
  are a set of $\Aut(A_n)$-representatives, whereas for the $W(A_n)=S_{n+1}$-action, we note that
  $$\{s_{r,t}: r,t\ge 1, r+t\le n+1\}$$
  are a set of $W(A_n)$-representatives.
  
  For the case of $n=4$, our smooth fan has 3 more $\Aut(A_4)$-orbits of rays with representatives $\v_1=\e_1+2\e_2-\e_3-\e_4-\e_5$, $\v_2=2\e_1+2\e_2-2\e_3-\e_4-\e_5$, $\v_3=2\e_1+3\e_2-2\e_3-2\e_4-\e_5$.
  
\begin{rem}
The representatives of the new $\Aut(A_4)$-orbit of rays can be described as the three possible vectors of the form 
$$\{\v-\w: \v\in \oplus_{i\in I}\N\e_i, \w\in \oplus_{j\in J}\N \e_j: \wt_I(\v)=\wt_J(\w)<6, v_i\le v_{i+1}\le 3, 2\ge w_j\ge w_{j+1}, i\in I,j\in J \}$$
for $I=\{1,2\}$ and $J=\{3,4,5\}$ where the weight of a vector is defined as $\wt_I(\v)=\wt_I(\sum_{i\in I}v_i\e_i)=\sum_{i\in I}v_i$.
We observe that $\gamma_4=s_{I,J}=3(\e_1+\e_2)-2(\e_3+\e_4+\e_5)$ is the difference of two vectors of weight 6=$\lcm(2,3)$.
\end{rem}

    If $Y(A_4)$ is the Demazure model for the algebraic torus $T$ corresponding to $(\widehat{T_L},\Gal(L/\Bbbk))=(\Lambda(A_4),\Aut(A_4))$ that we have constructed, then
$$0\to \Lambda(A_4)\stackrel{\divT}{\to}  \Div_{T_L}(Y(A_4))\to \Pic(Y(A_4))\to 0$$
    determines a $G$-flasque resolution of $\Lambda(A_4)$ for any $G\le \Aut(A_4)$
    where $\Div_{T_L}(Y(A_4))$ is a permutation lattice  determined by the $G$-orbits of rays in the fan.
    Note that the rays of our smooth projective $G$-invariant fan have primitive generators which are elements of $\Z A_4$.
    
    Each representative $\alpha$ of the $G$-orbits $\Omega$ on the rays of the smooth projective $G$-invariant fan determines a transitive permutation summand $\oplus_{\beta\in \Orb_G(\alpha)}D_{\beta}\cong \Z[G/\Stab_G(\alpha)]$.
    So $\Div_{T_L}(Y(A_4))=\oplus_{\alpha\in \Omega}\oplus_{\beta\in \Orb_G(\alpha)}D_{\beta}$.
    Note that the divisor class map is dual to the natural $G$-invariant surjection that induced by mapping  $\alpha\in \Omega$ to $D_{\alpha}\in \Div_{T_L}(Y(A_4))$.

    We will end this section by describing the stabilizer subgroups of the representatives of the $\Aut(A_4)$-orbits of the rays of our fan when $G=\Aut(A_4)$ in order to describe the corresponding $G$-flasque resolution.  Note that it suffices to find the stabilizer subgroups of the corresponding primitive generators which are elements of $\Z A_4$.

    We first determine stabilizer subgroups of vectors in $\Z A_n$ with respect to $S_{n+1}$ and $\Aut(A_n)$ actions.
    Note that if $\x\in \Z A_n$, then $\x=\sum_{a\in [1,n+1]}c_a\e_a$ for some $c_a\in \Z$ such that $\sum_{a\in [1,n+1]}c_a=0$.  Let $$C(\x)=\{r\in \Z: c_a=r \mbox{ for some }a\in [1,n]\}$$ be the (finite) set of coefficients of $\x$.  For each $r\in C(\x)$, let
    $I_r=\{a\in [1,n+1]: c_a=r\}$.  Then $\x=\sum_{r\in C(\x)}\sum_{a\in I_r}\e_a$.   Then $\Stab_{S_{n+1}}(\x)\cong \prod_{r\in C(\x)}S_{|I_r|}$.
    To express the stabilizer subgroup of $\x$ in $\Aut(A_n)$, we need to define $C_s(\x)=\{r: r>0,|I_r|=|I_{-r}|\}$ and $C_a(\x)=C(\x)\setminus C_s(\x)$.
    Then $$\Stab_{\Aut(A_n)}(\x)\cong \prod_{r\in C_s(\x)}(S_{|I_r|}^2\rtimes C_2)\times \prod_{r\in C_a(\x)}(S_{|I_r|})$$

    Since the $\Aut(A_4)$-orbits of rays are represented by the 6 rays $\{s_{r,t}:1\le r,t,r+t\le 5\}$ where $s_{r,t}=s_{[1,r],[r+1,r+t]}$ from the simplicial projective fan
    plus the 3 $\Aut(A_4)$-orbits of rays: $\v_1=\e_1+2\e_2-(\e_3+\e_4+\e_5)$, $\v_2=2(\e_1+\e_2)-2\e_3-(\e_4+\e_5)$ and $\v_3=2\e_1+3\e_2-2\e_3-2\e_4-\e_5$, we can use the above observations to determine their stabilizer subgroups.
    From these observations, we see that $\Stab_{S_{n+1}}(s_{r,t})=S_r\times S_t \times S_{n+1-r-t}$ where $r,t\ge 1$ and $r+t\le n+1$.
    If $r<t$ then $\Stab_{\Aut(A_n)}(s_{r,t})=\Stab_{S_{n+1}}(s_{r,t})=S_r\times S_t\times S_{n+1-r-t}$.
    If $r=t$ then $\Stab_{\Aut(A_n)}(s_{r,r})=[(S_r\times S_r)\rtimes S_2]\times S_{n+1-2r}$.
    We then note that $\Stab_{\Aut(A_4)}(\v_1)=\Stab_{S_{5}}(\v_1)=S_3$, $\Stab_{\Aut(A_4)}(\v_2)=\Stab_{S_{5}}(\v_2)=S_2\times S_2$, $\Stab_{\Aut(A_4)}(\v_3)=\Stab_{S_{5}}(\v_3)=S_2$.

    In summary, we have proven the following:
    \begin{prop} There exists a Demazure model $Y(A_4)$ of the algebraic torus $T$ with splitting field $L/\bk$, character lattice and splitting group $(\Lambda(A_4),\Aut(A_4))$.  The corresponding fan has 8 $\Aut(A_4)$-orbits of smooth maximal cones, resulting from the subdivision of the simplicial $\Aut(A_4)$-invariant fan.  $\Aut(A_4)$-orbits of maximal cones in the simplicial fan correspond to minimal staircase paths from $(i,i+1)$ to $(1,5)$, $i=1,2$.      
      \begin{itemize}
      \item For $C=\{(1,2),(1,3),(1,4),(1,5)\}$: $\Cone(\gamma_{C})$
      \item For $C_1=\{(2,3),(2,4),(2,5),(1,5)\}$: 2 cones from the subdivision of $\Cone(\gamma_{C_1})$ with respect to $\v_1=(\e_1+2\e_2)-(\e_3+\e_4+\e_5)$.
      \item For $C_2=\{(2,3),(1,3),(1,4),(1,5)\}$: 2 cones from the subdivision of $\Cone(\gamma_{C_2})$ with respect to $\v_2=(2\e_1+2\e_2)-(2\e_3+\e_4+\e_5)$.
      \item For $C_3=\{(2,3),(2,4),(1,4),(1,5)\}$: 3 cones from the subdivision of $\Cone(\gamma_{C_3})$ with respect to $\v_3=(2\e_1+3\e_2)-(2\e_3+2\e_4+\e_5)$.
      \end{itemize}
    \end{prop}

    \begin{cor}
      The fan corresponding to the Demazure model $Y(A_4)$ of the algebraic
      of the algebraic torus $T$ with splitting field $L/\bk$, character lattice and splitting group $(\Lambda(A_4),\Aut(A_4))$
      has 6 $\Aut(A_4)$-orbits of rays corresponding to
      $$s_{r,t}=s_{[1,r],[r+1,r+t]},1\le r\le t, r+t\le 5$$
      and 3 $\Aut(A_4)$-orbits of rays corresponding to
      \begin{itemize}
      \item $\v_1=(\e_1+2\e_2)-(\e_3+\e_4+\e_5)$,
      \item $\v_2=2(\e_1+\e_2)-2\e_3-(\e_4+\e_5)$
      \item $\v_3=2\e_1+3\e_2-2(\e_3+\e_4)-\e_5$
      \end{itemize}
      with stabilizer subgroups:
      \begin{itemize}
      \item $\Stab_{\Aut(A_4)}(s_{r,t})=S_{r}\times S_t\times S_{n-(r+t)},r<t$
      \item $\Stab_{\Aut(A_4)}(s_{r,r})=[(S_r\times S_r)\rtimes S_2]\times S_{n-2r}$
      \item $\Stab_{\Aut(A_4)}(\v_1)=S_3$
      \item $\Stab_{\Aut(A_4)}(\v_2)=S_2\times S_2$
      \item $\Stab_{\Aut(A_4)}(\v_3)=S_2$.
\end{itemize}
This information determines an $\Aut(A_4)$-coflasque resolution of $\Z A_4$ and so by duality an $\Aut(A_4)$-flasque resolution of $\Lambda(A_4)$.

    \end{cor}

 \section{Comparison with Kunyavskii's models}~\label{section:kunyavskii}

 This section refers to the construction of Demazure models for the algebraic $\bk$-tori in dimension 3 corresponding to the maximal finite subgroups of $\GL(3,\Z)$ by Kunyavskii in~\cite{Kun87}.
 The algebraic $\bk$-tori of dimension 3 corresponding to maximal finite subgroups are determined as in Section~\ref{section:dade}.
 
  For the algebraic torus corresponding to $(\Z A_2\oplus \Z A_1,\Aut(A_2)\times C_2)$, Kunyavskii takes the model $\dP_6\times \PP^1$, where $\dP_6$ is the Del Pezzo surface of degree 6. We will take $X(A_2)\times X(A_1)$.  Since $X(A_1)=\PP^1$ and $X(A_2)$ is known to coincide with $\dP_6$, these are equivalent.

  For the algebraic torus corresponding to $((\Z A_1)^3,\Aut((A^1)^3)=(\Z B_3,W(B_3))$, Kunyavskii takes $(\P^1)^3$ and we take $X(A_1)^3$, which are the same.

  For the algebraic torus corresponding to $(\Z A_3,\Aut(A_3))$, Kunyavskii takes a different point of view but the result is equivalent.

  For $n=3$, the root systems $A_3$ and $D_3$ are isomorphic.
  For $\Z D_3$, Kunyavskii considers the dual lattice $\Lambda(D_3)$ with respect to $\Aut(D_3)$.  We note that $\Aut(D_3)=\Aut(B_3)=W(B_3)$ and $\Lambda(D_3)=\Lambda(B_3)$.  The fundamental dominant weights of $B_n$ are $\omega_r=\sum_{i=1}^r\e_i,r=1,\dots,n-1$ and $\omega_n=\frac{1}{2}(\sum_{i=1}^n\e_i)$.
  Then $\Lambda(B_n)=\sum_{i=1}^n\Z\e_i+\Z\omega_n$ and $\Z W \omega_n=\Lambda(B_n)$.  The $W=\Aut(B_3)$-orbit of $\omega_3$ is then $\{\frac{1}{2}\sum_{i=1}^3\epsilon_i\e_i:\epsilon_i=\pm 1,i=1,2,3\}$.
  Kunyavskii takes his lattice polytope for his partial model of Case $S$ to be:
  $$\Conv\{\sum_{i=1}^3\epsilon_i\e_i:\epsilon_i=\pm 1,i=1,2,3\}$$
  So the cones over faces of this polytope give a normal $\Aut(D_3)$-invariant projective toric variety.  The facets are cut out by the 6 affine hyperplanes $H_{\pm \e_i,1}$.  He obtains cones
  $$\sigma_{ri}=\Cone((-1)^r\e_i+\sum_{j\ne i}\epsilon_j\e_j:\epsilon_j=\pm 1,j\ne i),r=1,2;i=1,2,3.$$  $\sigma_{ri}$ has barycentre $(-1)^r2\e_i$.  The corresponding affine toric varieties can be seen to correspond to the hypersurface $xy-zt=0$ in $\A^4$. Each maximal cone is barycentrically subdivided into 4 smooth maximal cones. The resulting polytope that he obtained in this process is the permutahedron.  So his smooth projective toric model of type $S$ is $X(A_3)$.

  For Kunyavskii's partial model corresponding to Case P, his polytope is $\Conv(\pm \e_i\pm \e_j: 1\le i\ne j\le 3)$.  This is the $D_3$-root polytope.  As $D_3=A_3$, we see that this corresponds to $(\Lambda(A_3),\Aut(A_3))$ since the polytope is a lattice polytope for the dual lattice $\Z A_3$.
  Our earlier discussion of this case shows that there are 2 $\Aut(A_3)=\Aut(D_3)=C_2^3\rtimes S_3$-orbits of facets of this polytope.
  Under the root system isomorphism from $A_3$ to $D_3$, we see that the two standard parabolic facets are $F_1=\Conv(\e_2-\e_3,\e_1-\e_3,\e_1+\e_2)$
cut out by the hyperplane $H_{\omega_2,1}$ where $\omega_2=\frac{1}{2}(\e_1+\e_2-\e_3)$ and $F_2=\Cone(\e_1\pm \e_r:r=2,3)$ cut out by the hyperplane $H_{\omega_1,1}$ where $\omega_1=\e_1$.
    As the $\Aut(D_3)$-orbit of $\omega_1$ contains 6 vectors, there are 6 maximal cones in the orbit of the smooth cone $\Cone(F_1)$.
    As the $\Aut(D_3)$-orbit of $\omega_2$ contains 8 vectors, there are 8 maximal cones in the orbit of the singular cone $\Cone(F_2)$.  The corresponding affine toric variety is again isomorphic to the hypersurface $xy-zt=0$ in $\A^4$.
    As $2\e_1$ may be taken as the barycentre, $\Cone(F_2)$ can be subdivided into 4 smooth $\Aut(D_3)$-cones.  The resulting fan has 32 smooth maximal cones and 12+6=18 rays.  Note that the rays are in bijection with the roots of $C_3$.  This model is smaller than ours for $n=3$, but its generalization would not work for $n\ge 4$, since $A_n$ and $D_n$ are no longer isomorphic for $n\ge 4$.
    
 \section{Generalized Del Pezzo Varieties, Klyachko Varieties and variants}
    
    Let $M=\Z A_{m-1}\otimes \Z A_{n-1}$.
    Note that $(\Z A_{m-1})^0=\Lambda(A_{m-1})\cong \oplus_{i=1}^m\Z \e_i/\Z(\sum_{i=1}^m\e_i)=\sum_{i=1}^m\Z\overline{\e_i}$, where
    $\sum_{i=1}^m\overline{\e_i}=\0$.
    Similarly we can express $(\Z A_{n-1})^0=\sum_{i=1}^n\Z\overline{\f_i}$ where $\sum_{i=1}^n\overline{\f_i}=\0$.
    Then $N=M^0=\Lambda(A_{m-1})\otimes \Lambda(A_{n-1})=\sum_{i,j}\Z\overline{\e_i}\otimes \overline{\f_j}$ has basis $\{\overline{\e_i}\otimes \overline{\f_j}:i=1,\dots,m-1,j=1,\dots,n-1\}$.
    It is convenient to set $\e_{ij}=\overline{\e_i}\otimes\overline{\f_j}$.
    In ~\cite{Kly83}, Klyachko introduced a family of lattice polytopes
    $$P_{m,n}=\Conv(\e_{i,j}:1\le i\le m, j=1,\dots,n\}$$
    
    The polar dual of $P_{m,n}$ is the central transportation polytope
    $$\Omega_{m,n}=\{(y_{ij})\in M_{mn}(\R):\sum_{i=1}^my_{ij}=m,\sum_{j=1}^ny_{ij}=n,y_{ij}\ge 0\}$$

    This is a special case of the more general transportation polytope
    $$\Omega_{m,n}(\bc,\d)=\{(y_{ij})\in M_{mn}(\R):\sum_{i=1}^my_{ij}=d_j,\sum_{j=1}^ny_{ij}=c_i,y_{ij}\ge 0\}$$
    where $\bc\in \N^m,\d\in \N^n$ and $\sum_{i=1}^mc_i=\sum_{j=1}^nd_j$
    which is an important polytope in optimization problems. See, for example, ~\cite{DLK14}. 
    Note that $\Omega_{m,n}=\Omega_{m,n}(\n,\m)$ where $\n=(n,\dots,n)\in \N^m$ and $\m=(m,\dots,m)\in \N^n$.

    $\Omega_{m,n}$ is a lattice polytope of dimension $d=(m-1)(n-1)$ and is simple (every vertex is the intersection of precisely $d$ facets) if and only if $\gcd(m,n)=1$.  Thus $P_{m,n}$ is a simplicial polytope (every facet is a simplex) when $\gcd(m,n)=1$.  Let $\Sigma_{m,n}$ be the fan corresponding to cones over faces of $P_{m,n}$ (or equivalently the normal fan of $\Omega_{m,n}$) and let $X_{m,n}=X_{\Sigma_{m,n}}$.
If $\gcd(m,n)=1$, $\Sigma_{m,n}$ is smooth.  

If $\gcd(m,n)=1$, it was shown in \cite{Kly83,KV84}, using the above reasoning, that the algebraic torus corresponding to the lattice
$(\Z A_{m-1}\otimes \Z A_{n-1},S_m\times S_n)$ has Demazure model $X_{m,n}$.
We will call these toric varieties $X_{m,n}$, Klyachko varieties following ~\cite{CFST18}.

It is also clear that in the case of general $m,n$ that $X_{m,n}$ is a projective $S_m\times S_n$-equivariant toric $T_L$-model of the algebraic $\bk$-torus with Galois splitting extension $L/\bk$ and character lattice $(\Z A_{m-1}\otimes \Z A_{n-1},S_m\times S_n)$, although it is no longer smooth or even simplicial if $m,n$ are not relatively prime.

An important special case of Klyachko varieties are the generalized del
Pezzo varieties $V_{2r}=X_{2,2r+1}$, $r\ge 1$.  $V_2$ coincides with the
del Pezzo surface of degree 6.

For general $m,n$, $X_{mn}$ is a normal projective toric variety with
character lattice $\Z
A_{m-1}\otimes \Z A_{n-1}$ and $\Div_T(X_{mn})=\Z[S_m/S_{m-1}]\otimes
\Z[S_n/S_{n-1}]$.  So its divisor class group $\Cl(X_{mn})$ satisfies the $S_m\times S_n$-exact sequence of lattices: 
$$\0\to \Z A_{m-1}\otimes \Z A_{n-1}\stackrel{i_m\otimes i_n}{\to}
\Z[S_m/S_{m-1}]\otimes \Z[S_n/S_{n-1}]\to \Cl(X_{mn})\to \0$$

Since $$\0\to \Z A_{m-1}\stackrel{i_m}{\to} \Z
[S_m/S_{m-1}]\stackrel{\epsilon_m}{\to} \Z\to \0$$
is a natural $S_m$-exact sequence, we have the following $S_m\times S_n$-exact sequence: 
$$\0\to \Z A_{m-1}\otimes \Z A_{n-1}\to \Z[S_m/S_{m-1}]\otimes
\Z[S_n/S_{n-1}]\to \Z[S_m/S_{m-1}]\oplus \Z[S_n/S_{n-1}]\to \Z\to \0$$   
Then the following $S_m\times S_n$ sequence is exact:
$$\0\to \Cl(X_{mn})\to \Z[S_m/S_{m-1}]\oplus \Z[S_n/S_{n-1}]\to \Z\to
\0$$
If $\gcd(m,n)=1$, $X_{mn}$ is smooth so that $\Pic(X_{mn})=\Cl(X_{mn})$
and the sequence above splits, so that
$$\Pic(X_{mn})\oplus \Z\cong \Z[S_{m}/S_{m-1}]\oplus \Z[S_n/S_{n-1}]$$
is stably permutation as an $S_m\times S_n$-lattice.

So in the case $\gcd(m,n)=1$, $X_{mn}$ is a smooth projective toric model
of the algebraic torus $T_{mn}$ corresponding to the character lattice
$(\Z A_{m-1}\otimes \Z A_{n-1},S_m\times S_n)$.  Since $\Pic(X_{mn})$ is 
stably permutation, $T_{mn}$ is stably rational.  In the case in which
$n\equiv 1\mod m$, Klyachko and Voskresenskii in ~\cite{KV84} also show that $T_{mn}$ is indeed rational.  
Note that the generalized del Pezzo varieties $V_{2r}=X_{2,2r+1}$ satisfy
this condition, so that the associated algebraic torus $T_{2,2r+1}$ is
rational.

We note that $(\Z A_{n},\Aut(A_n))=(\Z A_{n}\otimes \Z
A_1,S_{n+1}\times S_2)$.    So $V_{2r}=X_{2,2r+1}$ gives an alternative
Demazure model for the algebraic torus corresponding to $(\Z A_{2r},\Aut(A_{2r}))$ to that of $X(A_{2r})$.  In the smallest case, $V_2$
coincides with $X(A_2)$ and with $\dP_6$.  For  $r\ge 2$, $V_{2r}$
determines a smaller flasque resolution for $T_{2,2r+1}$ than does
$X(A_{2r})$.

For the remaining algebraic $\bk$-torus $T$ with Galois splitting extension $L/\bk$ corresponding to $(\Z A_2\otimes \Z
A_2,(S_3\times S_3)\rtimes C_2\times C_2)$, we remark that for its
restriction to $(\Z A_2\otimes \Z A_2,S_3\times S_3)$, we have a projective
toric $T_L$-model given by $X_{3,3}$.  It is unfortunately not simplicial.
But to find a toric model for the algebraic torus corresponding to $(\Z A_2\otimes \Z A_2,(S_3\times S_3)\rtimes C_2\times C_2)$ using Approach 2, we need to find a $G=(S_3\times S_3)\rtimes C_2\times C_2$-invariant finite subset of $N=\Lambda(A_2)\otimes \Lambda(A_2)$ which contains a $\Z$-basis of $N$.
We need to examine  the action of the group $(S_3\times S_3)\rtimes C_2\times C_2$ on $\Lambda(A_2)\otimes \Lambda(A_2)$, and on the polytope $P_{33}=\Conv(\e_{ij}:1\le i,j\le 3)$.  Thinking of $\{\e_{ij}:1\le i,j\le 3\}$ as the vector entries of a matrix, we see that the first copy of $S_3$ permutes the rows, the second copy of $S_3$ permutes the columns, the first copy of $C_2$ transposes the matrix and the last copy acts as negative the identity.
So we need to extend our polytope to $Q_{33}=\Conv(\pm \e_{i,j}:1\le i,j\le 3)$.

Now $Q_{33}$ is a lattice polytope which is invariant under the group action
of $G=[(S_3\times S_3)\rtimes C_2]\times C_2$ and whose rays include a basis of $\Lambda(A_2)\otimes \Lambda(A_2)$.  We need to verify that $Q_{33}$ is reflexive.  If this is the case, then the $G$-invariant projective normal toric variety which is determined by the fan consisting of cones over faces of $Q_{33}$ determines  a $G$-invariant projective toric $T_L$-model of our algebraic torus.

We will prove more generally that $Q_{mn}=\Conv(\pm \e_{ij}:1\le i\le m,1\le j\le n)$ is reflexive.
This is equivalent to showing that the facets of $Q_{mn}$ are cut out by affine hyperplanes in $M=\Z A_{m-1}\otimes \Z A_{n-1}$.
We will start with setting up some notation to describe the faces of $Q_{mn}$.
Let $V=\{\pm \e_{ij}:1\le i\le m,1\le j\le n\}$ be the set of vertices of $Q_{mn}$.

Suppose $\bu=\sum_{i=1}^m\sum_{j=1}^na_{ij}\e_{ij}$ 
determines the supporting hyperplane of a face $Q$ of $Q_{mn}$.
Then $Q_{mn}\subseteq H_{\bu,-1}^+$ and $Q=Q_{mn}\cap H_{\bu,-1}$.
We observe that $\bu\in Q_{mn}^0$.
We may identify $\bu$ with an $m\times n$ matrix $(a_{ij})$.
Since $\sum_{j=1}^n\e_{ij}=\0$ and $\sum_{i=1}^m\e_{ij}=\0$ for all $1\le i\le m, 1\le j\le n$ then we see that the matrix $(a_{ij})$ has all row and column sums 0.  Since $\langle \bu,\pm \e_{ij}\rangle \ge -1$, we see that $-1\le a_{ij}\le 1$.  Also note: $\pm \e_{ij}$ is a vertex of $Q$ if and only if $\mp a_{ij}=1$.

\begin{prop}  $Q_{mn}$ is a reflexive polytope for all $m,n\in \N$.
\end{prop}
\begin{proof}
The following argument uses ideas from the proof of Proposition 5 in ~\cite{KV84}.

Using the notation above, 
to show that $Q_{mn}$ is reflexive, we need to show that $Q_{mn}^0$ is a lattice polytope.  
The set of vertices of $Q_{mn}^0$  is precisely the set $\{\bu_F:F \mbox{ facet of } Q_{mn}\}$ where $u_F=\sum_{i=1}^m\sum_{j=1}^na_{ij}\e_{ij}\in M_{\R}$ is the facet normal of the facet $F$.  Let $\bu_F$, represented  by a $m\times n$ matrix $(a_{ij})$ with 0 row and column sums,  determine a supporting hyperplane for a facet $F$.
Then the set of vertices of $F$: $V_F=\{\pm \e_{ij}:a_{ij}=\mp 1\}$ must span $N_{\R}=(\Lambda(A_m)\otimes \Lambda(A_n))_{\R}$ which has dimension $(m-1)(n-1)$.  To show that $Q_{mn}$ is reflexive, we need to show that each such $\bu_F$ has integral coefficients or equivalently that if $\bu_F$ is represented by $(a_{ij})$ then $a_{ij}\in \Z$ for all $i,j$.

Suppose not.  Let $K_{mn}$ be a complete bipartite graph with disjoint vertex sets
$\{\overline{\e_i}:i=1,\dots,m\}$ and $\{\overline{f_j}:j=1,\dots,n\}$.  Recall that a subgraph $H$ of a graph $G$ has the same vertex set as $G$ but a subset of the edges, and the complement $\overline{H}$ of a subgraph $H$ of a graph $G$ has the same vertex set as $G$ but has edge set $E(\overline{H})=E(G)-E(H)$.

Note that any subgraph $G$ of $K_{mn}$ with edge set $E(G)$  defines a subset 
$$V[G]=\{\e_{ij}: (\overline{\e_i},\overline{\f_j})\in E(G)\}$$ of the generating set $\{\e_{ij}:1\le i\le m, 1\le j\le n\}$ of $N=\Lambda(A_{m-1})\otimes \Lambda(A_{n-1})$.
We define a subgraph $\Gamma$ of $K_{mn}$: $(\overline{\e_i},\overline{\f_j})$ is an edge of $\Gamma$ if and only if $a_{ij}\not\in \Z$.

Since each row and column of $(a_{ij})$ sums to 0, then any row or column with a non-integer entry must have at least 2 non-integer entries.  This shows that all vertices of $\Gamma$ have degree at least 2.  Then $\Gamma$ must contain a $2r$-cycle
$C=(\overline{\e_{a_1}},\overline{\f_{b_1}},\dots,\overline{\e_{a_r}},\overline{\f_{b_r}},\overline{\e_{a_1}})$ for some $r\ge 1$.
Let $\{a_1,\dots,a_m\}=\{1,\dots,m\}$ and $\{b_1,\dots,b_n\}=\{1,\dots,n\}$. 
Let $C'$ be the subgraph of $K_{mn}$ which contains all vertices of $K_{mn}$ and edge set that of $C$ plus additional edges $(\overline{\e_{a_1}},\overline{\f_{b_j}}):r+1\le j\le n$ and   $(\overline{\e_{a_i}},\overline{\f_{b_1}}):k+1\le i\le m$.  That is, the $\overline{\e_i}$ not in the cycle are connected to $\overline{\f_{b_1}}$ in the cycle and the $\overline{\f_j}$ not in the cycle are connected to $\overline{\e_{a_1}}$ in the cycle.
$C'$ has $2r+(m-r)+(n-r)=m+n$ edges.

Now let $F$ be a facet of $Q_{mn}$ with facet normal $\bu_F$ represented by the $m\times n$ matrix $(a_{ij})$.
Note that $\pm \e_{ij}\in V_F$ if and only if $a_{ij}\in \{\pm 1\}$ and
so $\e_{ij}\in V_F$ implies that $a_{ij}\in \Z$.
Then $V_F\subseteq V[K_{mn}-\Gamma]\subseteq V[K_{mn}-C]$.
Since $C$ is a subgraph of $C'$, we also have $V[K_{mn}-C']\subseteq V[K_{mn}-C]$.
We want to show that $\Span_{\Z}(V[K_{mn}-C'])=\Span_{\Z}(V[K_{mn}-C])$.

Since $E(K_{mn})-E(C)=(E(K_{mn})-E(C'))\cup (E(C')-E(C))$, we only need to show that $\e_{ij}\in \Span_{\Z}(V[K_{mn}-C'])$ if $(\overline{\e_i},\overline{\f_j})$ is an edge of $C'$ which isn't an edge of $C$.

The new edges of $C'$ which aren't edges of $C$ take the form $(\overline{\e_{a_1}},\overline{\f_{b_j}})$ where $r+1\le j\le n$ or $(\overline{\e_{a_i}},\overline{\f_{b_1}})$ where $r+1\le i\le m$.

In the first case, $\e_{a_1,b_j}=-\sum_{a\ne a_1}\e_{a,b_j},j=r+1,\dots,n$
and in the second case, $\e_{a_i,b_1}=-\sum_{b\ne b_1}\e_{a_i,b},i=r+1,\dots,n$.
In either case, the result is contained in $\Span_{\Z}(V[K_{mn}-C'])$.
But note that the set $V[K_{mn}-C']$ has size $|E(K_{mn})-E(C')|=mn-m-n$
and so cannot span $N=\Lambda(A_m)\otimes \Lambda(A_n)$ of rank $(m-1)(n-1)$.
So $V_F$ also cannot span $N_{\R}$, and so $F$ cannot be a facet of $Q_{mn}$.

By contradiction, we must have $a_{ij}\in \Z$ for all $1\le i\le m,1\le j\le n$.
\end{proof}

Let  $H_{\bu,-1}$ be a supporting hyperplane of a face of $Q_{mn}$ where we identify $\bu$ as a $m\times n$ matrix
$u=(a_{ij})$ with rows and column sums zero.
Then $Q_{mn}\subseteq H^+_{\bu,-1}$ and so $\pm a_{ij}\ge -1$ for all $i,j$ and so $-1\le a_{ij}\le 1$.
We have just shown that if $F$ is a facet of $Q_{mn}$ then its facet normal  $\bu_F=(a_{ij})\in M_{mn}(\Z)$ has only -1,0,1 as possible entries and rows and columns sum to 0.

We will now specialize to the case of $Q_{33}$. Let $\Sigma_{33}$ be the fan in $(\Lambda(A_3)\oplus \Lambda(A_3))_{\R}$ consisting of cones over faces of $Q_{33}$.  Note that $\Sigma_{33}$ is projective as the normal fan of $Q_{33}^0$ and is $G=(S_3\times S_3)\rtimes C_2\times C_2$-invariant.

\begin{rem} An alternative proof of this proposition appears in Jason Palombaro's thesis~\cite{Pal25}.  It is part of ongoing joint work with Jason Palombaro to further investigate Klyachko varieties and their variants.~\cite{LP}
\end{rem}

\begin{prop}\
  \begin{enumerate}

\item[(a)]
 There are 2 $G$-orbits of maximal cones of $\Sigma_{33}$ with representatives:
  \begin{itemize}
    \item
      $\Cone(\e_{11},-\e_{12},-\e_{21},\e_{22}\}$
    \item $\Cone(\e_{12},\e_{23},\e_{31},-\e_{13},-\e_{21},-\e_{32}\}$.
    \end{itemize}
\item[(b)]
  There are 2 $G$-orbits of cones of dimension 3 with representatives:
\begin{itemize}
\item $\Cone(\e_{11},\e_{22},\e_{33})$
\item $\Cone(\e_{11},-\e_{12},\e_{22})$
\end{itemize}
  \item[(c)] There are 2 $G$-orbits of cones of dimension 2 with representatives:
  \begin{itemize}
  \item $\Cone(\e_{11},\e_{22})$
  \item $\Cone(\e_{11},-\e_{12})$.
  \end{itemize}
  \item[(d)] There is 1 $G$-orbit of rays with representative $\Cone(\e_{11})$.
\end{enumerate}
  \end{prop}

\begin{proof}
  Since the cones of $\Sigma_{33}$ are cones over faces of $Q_{33}$, it is equivalent to find the $G$-orbits of faces of $Q_{33}$.

 (a)  We start with the facets of $Q_{33}$ to determine $G$-orbits of maximal cones of $\Sigma_{33}$.
  
Let $F$ be a facet and $\bu_F=(a_{ij})$ its facet normal.
Then the only possible rows and columns of $\bu_F$ are permutations of the vector $[1,-1,0]$ or the zero vector.
Note that if $a_{ij}=\pm 1$ if and only if  $\mp e_{ij}\in H_{\bu_F,-1}$.
Since $H_{\bu_F,-1}\cap Q_{33}$ must span $N$ we must have at least 4 non-zero entries in $\bu_F$ and at least 2 non-zero rows.

There are 2 cases, each leading to an orbit of maximal cones under the action of $G=[(S_3\times S_3)\rtimes C_2]\times C_2$.

\noindent\textbf{Case I:} $\bu_F$ has a zero row.

Then the two non-zero rows must be negatives of each other and so must have a zero in the same column.  Then $\bu_F$ must also have a zero column. 
There are 18 normal vectors $\bu_F$ of this form, determined by
the choice of the zero row, the choice of the zero column and the choice of sign.
The corresponding cone over the facet $F$ with zero row $r$, zero column $s$ and sign $(-1)^t$
is  
$(-1)^t\Cone(\e_{ip},-\e_{iq},-\e_{jp},\e_{jq})$ for $i<j,p<q, i,j\ne r;p,q\ne s$
and is clearly smooth.
We note that these cones form one orbit with representative
$\Cone(\e_{11},-\e_{12},-\e_{21},\e_{22})$
as we can choose $(-1)^t(\sigma,\tau)\in G$ with $\sigma(1)=i,\sigma(2)=j,\tau(1)=p,\tau(2)=q$.    
The stabilizer subgroup of our representative is $\langle (-(1,2),\id),(\id,-(1,2)),\tr\rangle\cong D_8$ (where the transpose map $\tr$ sends $\e_{ij}$ to $\e_{ji}$).

\noindent\textbf{Case II:} $\bu_F$ has all rows non-zero.

Then since all rows and columns of $\bu_F$  must be permutations of the vector $[1,-1,0]$,  the cones of this form must lie in one $G$-orbit.
Note that there is exactly one $r=-1,0,1$ in each row or column.
Recording the position of $r=-1,0,1$ as $(i,\sigma_r(i)),i=1,2,3$, we see that each $\sigma_r\in S_3$.
Choosing $\sigma_0\in S_3$ (that is: fixing the positions of the 3 zeros), we see that $\sigma_1(j)\ne \sigma_0(j)$ for all $j=1,2,3$.  So $\sigma_0^{-1}\sigma_1$ cannot have fixed points, and so must be one of the two  3-cycles. 
There are then 6 choices for $\sigma_0$ and 2 choices for $\sigma_0^{-1}\sigma_1$.  This determines $\sigma_0,\sigma_1$ and hence $\sigma_{-1}$.  So we have 12 possibilities.

Choosing $\sigma_0$ to be $\id$ and $\sigma_1$ to be $(123)$, we get an orbit representative:
$$\Cone(\e_{12},\e_{23},\e_{31},-\e_{13},-\e_{21},-\e_{32})$$
The resulting 12 maximal cones in the orbit are not simplicial.
The stabilizer subgroup of this cone in $G$ contains $\sgn(\sigma)(\sigma,\sigma)$ and $-\tr$ and has size 12.
This shows that the stabilizer subgroup is $\langle (-1)((12),(12)),(-1)((23),(23)),-\tr\rangle \cong S_3\times C_2\cong D_{12}$.

We need to determine the orbits of the lower dimensional faces.  Suppose $H_{\bu,-1}$ with $\bu=(a_{ij})\in M_{33}(\R)$ cuts out a face of $Q_{33}$.  We want to determine the $\pm \e_{ij}$ in $H_{\bu,-1}\cap Q_{33}$.

We have already observed that $-1\le a_{ij}\le 1$ and that rows and columns sum to 0.  Since $a_{ij}=1$ if and only if $-\e_{ij}\in H_{\bu,-1}$ and $a_{ij}=-1$ if and only if $\e_{ij}\in H_{\bu,-1}$, we may have at most one of $\{\e_{ij},-\e_{ij}\}$ in $H_{\bu,-1}$.

If $\e_{ij}\in H_{\bu,-1}$ or equivalently $a_{ij}=-1$, then all other entries in row i and column j must be between 0 and 1, inclusive.
Similarly, if $-\e_{ij}\in H_{\bu,-1}$ or equivalently $a_{ij}=1$, then all other entries in row i and column j must be between -1 and 0, inclusive.
In particular, $\bu$ may have at most two $\pm 1$ entries in any row or column, and if so, they need to be of opposite sign.

This also shows that if  $\e_{ij},-\e_{pq}\in H_{\bu,-1}$ for $i\ne p,j\ne q$, then $a_{ij}=-1,a_{pq}=1,a_{iq}=a_{pj}=0$.
But this completely determines $\bu\in M_{33}({-1,0,1})$ and shows that $\bu$ is the normal vector of a facet of $Q_{33}$.  This implies that  $\e_{ij},-\e_{pq}$, $i\ne p, j\ne q$ cannot be vertices of a face of dimension less than 4.

We wish to determine the possible subsets $H_{\bu,-1}\cap Q_{33}$.

(d) Note that it is possible to find a $\bu$ such that $H_{\bu,-1}\cap Q_{33}=\{\e_{11}\}$.  
For example, the face $\Conv(\e_{11})$ can be cut out by
$$\bu=\begin{bmatrix}-1&\frac{1}{2}&\frac{1}{2}\\\frac{1}{3}&-\frac{1}{4}&-\frac{1}{12}\\\frac{2}{3}&-\frac{1}{4}&-\frac{5}{12}\end{bmatrix}$$
Since we may map $\e_{11}$ to $\pm \e_{ij}$ under the group action, any $\Conv(\pm \e_{ij})$ is a 1-dimensional face of $Q_{33}$.

(c) Now assume that $|H_{\bu,-1}\cap Q_{33}|=2$.  This means that $|\{(i,j):a_{ij}=\pm 1\}|=2$. 
Let $$r(\bu)=\{i: \exists j, a_{ij}=\pm 1\};\qquad 
c(\bu)=\{j:\exists i, a_{ij}=\pm 1\}$$ be the set of row indices (resp. column indices) whose row (column) contains a $\pm 1$.
There are 2 cases:

\noindent\textbf{Case I:} $|r(\bu)|=1$ or $|c(\bu)|=1$

This means that the 2 $(\pm 1)$-entries lie in one row or column.
For example, the two-dimensional  face $\Conv(\e_{11},-\e_{12}\}$ can be cut out by
$$\bu=\begin{bmatrix}-1&1&0\\\frac{1}{2}&-\frac{1}{3}&-\frac{1}{6}\\\frac{1}{2}&-\frac{2}{3}&\frac{1}{6}\end{bmatrix}$$
and $\e_{11},-\e_{12}$ can be extended to the $\Z$-basis $\e_{11},-\e_{12},-\e_{21},\e_{22}$.  This shows that this is a smooth face.
All other faces of the form $\pm \Conv(\e_{ij},-\e_{iq})$ or $\pm \Conv(\e_{iq},-\e_{jq})$ are in the $G$-orbit of $\Conv(\e_{11},-\e_{12})$.

\noindent\textbf{Case II:} $|r(\bu)|=|c(\bu)|=2$.

Let $r(\bu)=\{i,p\}$, $c(\bu)=\{j,q\}$.
We have observed that $\Conv(\e_{ij},-\e_{pq})$, $i\ne p$,$j\ne q$, cannot be a face.
However if $\epsilon=\pm 1$, $\Conv(\epsilon\e_{ij},\epsilon\e_{pq})$, $i\ne p$,$j\ne q$ are two-dimensional faces.
For example, the two-dimensional  face $\Conv(\e_{11},\e_{22}\}$ can be cut out by
$$\bu=\begin{bmatrix}-1&\frac{1}{2}&\frac{1}{2}\\\frac{2}{3}&-1&\frac{1}{3}\\\frac{1}{3}&\frac{1}{2}&-\frac{5}{6}\end{bmatrix}$$
and $\e_{11},\e_{12}$ can be extended to the $\Z$-basis $\e_{11},-\e_{12},-\e_{21},\e_{22}$.  This shows that this is a smooth face.
All other faces of the form $\Conv(\epsilon\e_{ij},\epsilon\e_{pq}),i\ne p,j\ne q$  are in the $G$-orbit of $\Conv(\e_{11},\e_{22})$.

(b) Next assume that $|H_{\bu,-1}\cap Q_{33}|=3$, or $|\{(i,j): a_{ij}=\pm 1\}|=3$.  Note that the dimension of the face determined could be at most 3.
Since the 3 $(\pm 1)$-entries cannot lie in the same row or column, we must have $2\le |r(\bu)|,|c(\bu)|\le 3$.

There are 3 cases:

\noindent\textbf{Case I:} $|r(\bu)|\ne |c(\bu)|$.

We claim that this case cannot occur.  Suppose $|r(\bu)|=2<|c(\bu)|=3$.  Since there are 3 $(\pm 1)$-entries and $|r(\bu)|=2$, there exists a row $i$ with 2 $(\pm 1)$-entries.  In fact then this means that $a_{ir}=1$,$a_{is}=-1$ and $a_{it}=0$ for some distinct $r,s,t$. Since $|c(\bu)|=3$, there exists $p\ne i$ such that $a_{pt}=\pm 1$.
If $a_{pt}=1$ then $\e_{is},-\e_{pt}\in H_{\bu,-1}$ and so $\bu$ determines a facet, and so gives a contradiction.
If $a_{pt}=-1$ then $-\e_{ir},\e_{pt}\in H_{\bu,-1}$ and so $\bu$ determines a facet, and hence a contradiction.

If  $|c(\bu)|=2<|r(\bu)|=3$, we can apply the transpose map to get into the previous case, and so this cannot occur as well.

\noindent\textbf{Case II:} $|r(\bu)|=|c(\bu)|=2$.

$r(\bu)=\{i,p\}$, $c(\bu)=\{j,q\}$.  Since there are 3 $(\pm 1)$-entries, 1 row, say $i$, and 1 column, say $j$, must have 2 $(\pm 1)$-entries.
So we have $\pm \e_{ij},\mp \e_{iq},\mp \e_{pj}$ in our face determined by $\bu$.

Note that $\Conv(\e_{11},-\e_{12},-\e_{21}\}$ is cut out by
$$\bu=\begin{bmatrix}-1&1&0\\1&-\frac{1}{2}&-\frac{1}{2}\\0&-\frac{1}{2}&\frac{1}{2}\end{bmatrix}$$
It is easy to see that $\pm \Conv(\e_{ij},-\e_{iq},-\e_{pj})$ is in its $G$-orbit and since  $\e_{11},-\e_{12},-\e_{21},\e_{22}$ is a $\Z$-basis of $N$, the 3-dimensional cones in this orbit are all smooth.

\noindent\textbf{Case III:} $|r(\bu)|=|c(\bu)|=3$.

$r(\bu)=\{1,2,3\}=c(\bu)$.  Then $\epsilon_i \e_{i,\sigma(i)},i=1,2,3\in H_{\bu,-1}$ where $\epsilon_i=\pm 1,\sigma\in S_3$.  Since we have seen that
$\e_{i,\sigma(i)}$ and $\e_{j,\sigma(j)}$, $i\ne j$ cannot have different signs in a face of dimension $< 4$, we see that we have $\epsilon\e_{i,\sigma(i)}\in H_{\bu,-1}$

Note that $\Conv(\e_{11},\e_{22},\e_{33}\}$ is cut out by
$$\bu=\begin{bmatrix}-1&\frac{1}{2}&\frac{1}{2}\\\frac{1}{2}&-1&\frac{1}{2}\\\frac{1}{2}&\frac{1}{2}&-1\end{bmatrix}$$
It is easy to see that $\Conv(\epsilon\e_{1\sigma(1)},\epsilon\e_{2\sigma(2)},\epsilon\e_{3\sigma(3)})$ is in its $G$-orbit.
Since $\e_{33}=\e_{11}+\e_{12}+\e_{21}+\e_{22}$, $\e_{11},\e_{22},\e_{33},\e_{12}$ is a $\Z$-basis of $N$ , and the 3-dimensional cones in this orbit are all smooth.

Now we need only show that if $|H_{\bu,-1}\cap Q_{33}|\ge 4$, then the resulting face is 4-dimensional, and we have already dealt with that case.

Again $|r(\bu)|,|c(\bu)|\ge 2$.

If $|r(\bu)|=2$ then we have at least 2 rows with 2 $(\pm 1)$-entries and so are permutations of $[1,-1,0]$.  As rows and columns sum to zero, $\bu$ is determined as a matrix with only entries in $\{1,-1,0\}$ and so $\bu$ determines a facet.  The same argument applies to the $|c(\bu)|=2$ case.

If $|r(\bu)|=|c(\bu)|=3$, then we have seen that $\epsilon\e_{i,\sigma(i)},i=1,2,3\in H_{\bu,-1}$ for some $\epsilon=\pm 1$ and $\sigma\in S_3$, and we have at least one row with 2 $(\pm 1)$-entries: say $-\epsilon\e_{p,q}$, $q\ne \sigma(p)$.  Then there exists $j\ne p$, $\sigma(j)\ne q$, with $-\epsilon\e_{p,q}$ and $\epsilon\e_{j,\sigma(j)}\in H_{\bu,-1}$ and so $\bu$ determines a facet.
\end{proof}

We see from the proposition that our fan $\Sigma_{33}$ is not smooth, but isn't far from being so.
In terms of cones in our fan, we have 1 orbit of 1-dimensional (smooth) cones, 2 orbits of 2-dimensional smooth cones, 2 orbits of 3-dimensional smooth cones and 1 smooth and 1 non-simplical orbit of 4-dimensional cones.

\begin{prop}
  There exists a smooth $G$-invariant projective toric variety $SK_{3,3}$ determined by subdividing each of  4-dimensional cones in the non-simplicial $G$-orbit by a star-subdivision with respect to its barycentre.
  The resulting flasque $G$-resolution is
  $$0\to (\Z A_2)^{\otimes 2}\to \Z[G/D_8]\oplus \Z[G/D_{12}]\to \Pic(SK_{3,3})\to 0$$
\end{prop}

\begin{proof}

Looking at our representative of the orbit of 4-dimensional non-simplicial cones:
$$\Cone(\e_{12},\e_{23},\e_{32},-\e_{13},-\e_{21}-\e_{32}),$$ its barycentre in terms of the $\Z$-basis $\{\e_{11},\e_{12},\e_{21},\e_{22}\}$ is
$\e_{12}-\e_{21}$.  We note that the stabilizer subgroup of this barycentre is the same as that of the maximal cone and hence is isomorphic to $D_{12}$.

The star subdivision of this cone with respect to $\e_{12}-\e_{21}$ leads to smooth cones of the form
$\Cone(\e_{12}-\e_{21},\tau)$, $\tau$ 3-dimensional subcone of the original cone.

Effectively, the fan of $\SK_{33}$ is obtained from the fan of faces over $Q_{33}$ by performing one star subdivision for each non-simplicial 4-dimensional cone with respect to its barycentre.  Each such star subdivision preserves the projectivity of the fan. As these cones are all in one orbit and the stabilizer subgroup of the cone and the barycentre agree, the resulting fan is $G$-invariant. 
    
$\e_{11}$ has orbit $$\{\pm \e_{ij}:1\le i\ne j\le 3\}$$ of size 18.  The group $G=(S_3\times S_3)\rtimes C_2\times C_2$ has size 144.  The stabilizer subgroup is  $\langle \tr,((23),\id),(\id,(23))\rangle\cong D_8$.
As previously observed, $\e_{12}-\e_{21}$ has stabilizer subgroup isomorphic to $D_{12}$.

So the flasque resolution corresponding to our smooth projective $G$-invariant toric variety $\SK_{3,3}$ is
$$0\to (\Z A_2)^{\otimes 2}\to \Z[G/D_8]\oplus \Z[G/D_{12}]\to \Pic(\SK_{3,3})\to 0$$
\end{proof}

\section{Comparison of Techniques for determining Coflasque/Flasque Resolutions}

We have discussed two methods for constructing coflasque/flasque resolutions of a $G$-lattice:

As first proven in~\cite{CTS77}, one can construct a flasque $G$-resolution of a $G$-lattice $M$ by finding a coflasque $G$-resolution of the $\Z$-dual lattice $N$ by 
taking $P=\oplus_{H\le G}\Ind_H^G(N^H)\to N$ where $N=M^0$.

This can be done whenever one knows the fixed point lattice $N^H$ for all conjugacy classes of subgroups $H$ in $G$.
So this may be done either with the help of a computer, or with a fair amount of information about the lattice perhaps coming from combinatorics or geometry.  

Note that the resolution given above is by no means minimal.  More information would be required to find a minimal such resolution.  We remark that Hoshi and Yamasaki in ~\cite{HY17} use some clever techniques to reduce their flasque resolutions so that the flasque lattices are of more reasonable ranks.

One could also determine a flasque resolution for a $G$-lattice $M$ by finding a smooth projective model
$X$ of an algebraic $\bk$-torus $T$ split by a Galois extension $L/\bk$ with Galois group $G=\Gal(L/\Bbbk)$.  Implicitly this assume
s that one has chosen a base field $\bk$ such that the inverse Galois problem is solvable for $G$ over $\bk$.

By our earlier discussion, there exists a smooth projective fan $\Sigma$ in $N_{\R}$ where $N=M^0$ such that $X_L=X_{\Sigma}$ and $G\le \Aut(X_{\Sigma})$.
A flasque $G$-resolution of $M$ is then given by 
$$0\to M\stackrel{\divT}{\to} \Div_{T_L}(X_{\Sigma})\to \Pic(X_{\Sigma})\to 0$$

We know from earlier discussion that any two flasque resolutions of the same $G$-lattice $M$ (regardless of the construction) produce flasque lattices which are similar as $G$-lattices, and whose stable equivalence class determines information about the stable birational invariants of the associated algebraic $\bk$-torus.

However, the following is an observation about the construction which may point to some geometric differences between flasque resolutions of a given $G$-lattice.

Suppose $\varphi:Q\to N$ is a surjection of a  $G$-permutation lattice $Q$ with $\Z$-basis $Y$ onto a given $G$-lattice $M$.  We note that the map $\varphi$ is completely determined by $\{n_y=\varphi(y): y\in Y/G\}$ where $Y/G$ denotes a set of representatives of the $G$-orbits of $G$ on the $G$-set $Y$. 

We will assume that there exists a Galois extension $L/\bk$ with Galois group $G$. By our earlier discussion, if $S_Y=\{gn_y:g\in G,y\in Y/G\}$ contains a $\Z$-basis of $N$, and $P_Y=\Conv(S_Y)$ is a reflexive polytope in $N_{\R}$ containing 0 in its interior, then the fan $\Sigma_Y$ determined by cones over faces of $P_Y$ is a normal projective $G$-invariant toric model of the split torus $T_L$ corresponding to the algebraic $\bk$-torus $T$ with character lattice $\widehat{T_L}=M$. 
The divisor class group sequence of the normal projective toric variety $X_{\Sigma_Y}$ is then
by construction, the $\Z$-dual of the $G$-exact sequence
$$0\to \ker(\varphi)\to Q\to N\to 0$$

In particular, the map $Y/G\to N, y\to n_y$ which determines the $G$-map $\varphi:Q\to N$ also determines the dual map $\varphi^*:M\to Q$: $m\in M$ determines uniquely an element $f_m\in N^0$, where $f_m(n)=\langle m,n\rangle$.
Then  $\varphi^*(m)(y)=f_m\circ \varphi(x)=\langle m,n_y\rangle$.
So $\varphi^*(m)$ corresponds to the map $Q\to \Z,y\to \langle m,n_y\rangle$ and hence
$\varphi^*(m)=\sum_{g\in G}\sum_{y\in Y/G}\langle m,gn_y\rangle gy$.

Note that since $M$ and $M^0$ are $G$-modules,  the pairing $M\times M^0\to \Z$ is $G$-invariant.
If $g\in G$ and $n\in N=M^0$, $gf_n:M\to \Z$ is the unique map $f_{gn}:M\to \Z$ corresponding to $gn$.  So $gf_n(m)=f_n(g^{-1}m)=\langle g^{-1}m,n\rangle$ should be equal to $f_{gn}(m)=\langle m,gn\rangle$.

We wish to observe that although a smooth projective model of an algebraic $\bk$-torus $T$ split by $L/\bk$ with character lattice $\hat{T_L}=M$ produces a flasque resolution of $M$, it is not true that all flasque resolutions of $M$ arise in this way, although they all provide the same stable birational information about the algebraic $\bk$-torus $T$.

We will show that it is  possible to construct a coflasque resolution of the $G$-lattice $N=M^0$ 
using a surjection $\varphi:Q\to N$ such that $Q=\Z Y$ and such that $P_Y=\Conv(\varphi(Y))$ is a polytope satisfying the above conditions, but such that the resulting fan $\Sigma_Y$ consisting of cones over faces of $P_Y$ is not even simplicial, let alone smooth.

Let $M$ be the $S_n$-lattice $\Lambda(A_{n-1})$. Its $\Z$-dual lattice is $\Z A_{n-1}$.
The augmentation sequence 
$$0\to \Z A_{n-1}\to \Z[S_n/S_{n-1}]\to \Z\to 0$$
is an $S_n$-exact sequence, as is its dual sequence
$$0\to \Z\to \Z[S_n/S_{n-1}]\to \Lambda(A_{n-1})\to 0$$

Note that the corresponding algebraic $\bk$-torus to the $S_n$-lattice $\Lambda(A_{n-1})$ is a norm-1 torus.  Let $K/\bk$ be a Galois extension with Galois group $S_n$ and let $E=K^{S_{n-1}}$ be the correponding fixed field.
  Then the corresponding sequence of algebraic $\bk$-tori is:
  $$1\to R^{(1)}_{E/\bk}(\Gm_{m,E})\to R_{E/\bk}(\Gm_{m,E})\stackrel{N_{E/\bk}}{\to} \Gmk\to 1$$
  where $R_{E/\bk}$ indicates the Weil restriction of scalars functor and $N_{E/\bk}$ is the norm map.

Since $\Z A_{n-1}$ restricted to $S_{n-1}$ is permutation with permutation basis $\{\e_i-\e_n:i=1,\dots,n-1\}$, one can see that $\Z A_{n-1}\otimes \Z[S_n/S_{n-1}]\cong \Z[S_n/S_{n-2}]$, where the permutation basis is given by $y_{ij}=(\e_i-\e_n)\otimes \e_j, 1\le i\ne j\le n$.

Tensoring the augmentation sequence  with $\Z A_{n-1}$ one obtains an $S_n$-exact sequence:
$$0\to \Z A_{n-1}^{\otimes 2}\to \Z[S_n/S_{n-2}]\to \Z A_{n-1}\to 0$$

For $n=p$ prime, this is known to be a coflasque resolution of the $S_p$-lattice $\Z A_{p-1}$.
[This sequence is important for its connection to the question of the stable rationality of the centre of the generic division algebra. See ~\cite{BL91,Leb95}.]

The following quick proof will be useful for further observations:
To check that a $G$-lattice $C$ is coflasque, it suffices to check that for each prime $p$ dividing the group order of $G$, it is coflasque as a $G_p$-lattice where $G_p$ is a Sylow $p$-subgroup of $G$. (See, eg. ~\cite[1.8.25]{ShaGroupCoh}).

Note that  the lattice $\Z A_{n-1}$ is permutation as an $H$-lattice if $H\le S_n$ has a common fixed point on its action on $\{1,\dots,n\}$, the same is true for $\Z A_{n-1}^{\otimes 2}$.

So for $n=p$, $p$ prime, the Sylow $q$-subgroups of $S_p$, for $q\ne p$ all have common fixed points on their action on $[1,p]$.  So $\Z A_{p-1}^{\otimes 2}$ is permutation as a lattice for any Sylow q-subgroup of $S_q$ and so is coflasque.
A Sylow $p$-subgroup $G_p$ of $S_p$ is cyclic of order $p$ and acts transitively on $[1,p]$.  Then $\Z A_{p-1}^{G_p}=0$.
Since $\Z[S_p/S_{p-2}]$ is permutation and hence coflasque, we see from the LES of group cohomology for $G_p$ that
$H^1(G_p,\Z A_{p-1}^{\otimes 2})=0$.

[In fact it is additionally true that $\Z A_{p-1}^{\otimes 2}$ is permutation projective as an $S_p$-lattice.~\cite[9.12]{CTS87} (or see: ~\cite[Prop 3]{BL91}.)]

The polytope that corresponds to the surjection $\varphi:\Z[S_p/S_{p-2}]\to \Z A_{p-1}$ is the root polytope $P(A_{p-1})$ and so the toric variety of $P(A_{p-1})^0$ is a normal projective model of the algebraic $\bk$-torus corresponding to $(\Lambda(A_{p-1}),S_p)$.  However, as we have observed, the corresponding fan is not even simplicial for $p\ge 5$.
The $S_p$-orbits of facets are isomorphic to $\Delta_{i-1}\times \Delta_{p-i},i=1,\dots,p$ and cones over non-trivial products of simplices are not simplicial.

So the $S_p$-flasque resolution of $\Lambda(A_{p-1})$ that results from the dual sequence of this $S_p$-flasque resolution does not arise from a smooth projective model of the corresponding norm-1 algebraic torus.   

Based on this construction, it seems likely that many of the minimal flasque resolutions arising from the representation-theoretic construction also do not arise from smooth projective models of their corresponding algebraic tori, as this is a fairly restrictive condition.  Yet they still can be used to derive stable birational invariants.  It would be interesting to see whether flasque resolutions which arise from smooth projective models of algebraic tori have any stronger geometric properties or can be distinguished by some birational invariant.

To conclude, we will show that representation-theoretic techniques can be used to determine smaller flasque resolutions of $\Lambda(A_{n-1})$ as an $S_n$-lattice even in the cases in which $n$ is not prime.

We will actually construct smaller coflasque resolutions of $\Z A_{n-1}$ as an $S_n$-lattice.

For any subgroup $H$ of $S_n$,
$$H^1(H,\Z A_{n-1}^{\otimes 2})\cong \coker(\Z[S_n/S_{n-2}]^H\to (\Z A_{n-1})^H)$$
since permutation lattices are coflasque.

Let $\Orb^n_H$ be the set of $H$-orbits from the action of $H$ on $[1,n]$.
Now  $\Z[S_n/S_{n-1}]^H$ has $\Z$-basis $\{\sum_{i\in I}\e_i:I\in \Orb^n_H\}$
and $(\Z A_{n-1})^H$ is generated by
$$\{s_{I,J}=\frac{|J|\sum_{i\in I}\e_i-|I|\sum_{j\in J}\e_j}{\gcd(|I|,|J|)}: I\ne J\in \Orb^n_H\}$$

Let $Q$ be a permutation $S_n$-lattice with basis
$$\{y_{I,J}:\emptyset\ne I,J\subseteq [1,n],I\cap J=\emptyset\}$$

Then it is clear that the surjection $Q\to \Z A_{n-1}$ determined by sending $y_{I,J}\to s_{I,J}$ produces a coflasque resolution of $\Z A_{n-1}$ whose dual would give the divisor class sequence determined by the simplicial $S_n$-invariant projective model of this algebraic $\bk$-torus.

However, it is possible to reduce the size of this coflasque resolution.
We will start with the $S_n$-exact sequence:
$$0\to \Z A_{n-1}^{\otimes 2}\to \Z[S_n/S_{n-2}]\to \Z A_{n-1}\to 0$$
where $\Z[S_n/S_{n-2}]$ has basis $\{y_{ij}:1\le i\ne j\le n\}$ and $y_{ij}$ maps to $\e_i-\e_j$ under the second map  of the sequence.

Suppose $H$ is a subgroup of $S_n$ and $I\ne J\in \Orb_n^H$ have relatively prime orders.
Then $\sum_{i\in I,j\in J}y_{ij}$ is fixed by $H$ and maps to $s_{I,J}=|J|\sum_{i\in I}\e_i-|I|\sum_{j\in J}$.
We see that $\Z[S_n/S_{n-2}]^H\to \Z A_{n-1}^H$ is surjective.

We have observed that it suffices to show that our sequence is coflasque as restricted to each Sylow $p$-subgroup, $G_p$ of $S_n$.
If $H$ is a subgroup of $G_p$, then it is necessarily a $p$-group and its orbits on $[1,n]$ must have sizes which are powers of $p$.  Since $\Z A_{n-1}^{\otimes 2}$ is permutation as a $H$-lattice for any subgroup of $S_n$ which has a fixed point in its action on $[1,n]$, we see that we need only consider primes $p$ which divide $n$, as Sylow $q$-subgroups with $q$ not dividing $n$ must have a fixed point on their action on $[1,n]$.

Let $p$ be a prime divisor of $n$ and let $Q_p$ be the permutation $S_n$-lattice with basis
$$\{y_{I,J}:\emptyset\ne I,J\subseteq [1,n],I\cap J=\emptyset, |I|=p^r,|J|=p^s, 1\le r\le s\}$$
and let $Q_0=\Z[S_n/S_{n-2}]$.

Then $Q_0\oplus \oplus_{p|n}Q_p\to \Z A_{n-1},y_{I,J}\to s_{I,J}$ determines  a surjection that determines an $S_n$-coflasque resolution of $\Z A_{n-1}$.

For $n=4$, this resolution recovers the coflasque resolution corresponding to Kunyavskii's smooth projective model of $\Lambda(A_3)$ restricted to $S_4$.

The corresponding permutation lattice is $Q_0\oplus Q_2\cong \Z[S_4/S_2]\oplus \Z[S_4/S_2\times S_2]$, so that the flasque lattice would be of rank 15.

For $n=6$, we can reduce the above coflasque resolution still further.
From our construction, the permutation lattice would be $Q_0\oplus Q_2\oplus Q_3\cong \Z[S_6/S_4]\oplus \Z[S_6/S_2\times S_2\times S_2]\oplus \Z[S_6/S_4\times  S_2]\oplus \Z[S_6/S_3\times S_3]$.

A Sylow 2-subgroup of $S_6$ is $D_8\times C_2=\langle (1,2,3,4),(1,2)(3,4),(5,6)\rangle$.  
Suppose $H\le S_6$ has orbits of size 4 and 2 on $[1,6]$.
Then $y_{\{1,3\},\{5,6\}}+y_{\{2,4\},\{5,6\}}$ is fixed by this Sylow 2-subgroup (and hence by any subgroup) and is mapped to $s_{\{1,2,3,4\},\{5,6\}}$.  This shows that we may reduce the permutation lattice to $Q_0\oplus Q_2'\oplus Q_3$ where $Q_2'=\Z[S_6/(S_2\times S_2\times S_2)]$.

\bibliographystyle{alpha}
\bibliography{demazurerefs}
\end{document}